\documentclass[reqno]{amsart}

\usepackage{tikz}
\usepackage{changebar}
\usepackage{pdfsync}
\usetikzlibrary{decorations.pathmorphing,patterns}
\usetikzlibrary{calc,arrows}

\usepackage{pgf}

\usepackage{MnSymbol}
\usepackage{enumerate}

\usetikzlibrary{automata}
\usetikzlibrary{positioning}

\tikzset{
    state/.style={
           rectangle,
           rounded corners,
           draw=black, very thick,
           minimum height=2em,
           inner sep=2pt,
           text centered,
           },
}

\usepackage{xcolor}

\usepackage{changebar}
\usepackage{pdfsync}

\usepackage[normalem]{ulem}

\makeindex
 \newtheorem{theorem}{Theorem}[section]
 \newtheorem{lemma}[theorem]{Lemma}
 \newtheorem{proposition}[theorem]{Proposition}
 \newtheorem{corollary}[theorem]{Corollary}
 
 \newtheorem{definition}[theorem]{Definition}
 \newtheorem{remark}[theorem]{Remark}
 
 \newtheorem{assumption}[theorem]{Assumption}

 \newcommand{\cal}{\mathcal}

\newcommand{\mc}[1]{{\mathcal #1}}

\newcommand\bbR{{\mathbb R}}
\newcommand\bbI{{\mathbb I}}
\newcommand\bbZ{{\mathbb Z}}

\newcommand\bp{{\mathbf p}}

\newcommand\al{\alpha}
\newcommand{\cD}{\mathcal{D}}

\newcommand{\bE}{\mathbb{E}}

\newcommand\lang{\llangle}
\newcommand\rang{\rrangle}

\newcommand\bbE{{\mathbb E}}

\newcommand\R{{\mathbb R}}
\newcommand\RR{{\mathbb R}}

\newcommand\bbP{{\mathbb P}}

\newcommand\Z{{\mathbb Z}}

\newcommand\EE{{\mathbb E}}

\newcommand \ga{\gamma}
\newcommand \om{\omega}
\newcommand \la{\lambda}

\newcommand\Om{\Omega}

\newcommand\br{{\mathbf r}}

\renewcommand{\ge}{\geqslant}
\renewcommand{\geq}{\geqslant}
\renewcommand{\le}{\leqslant}

\newcommand{\dd  }{\mathrm{d}}

\newcommand{\E}{\mathcal{E}}
\renewcommand{\hat}{\widehat}
\renewcommand{\tilde}{\widetilde}
\renewcommand{\bar}{\overline}
\numberwithin{equation}{section}

\newcommand{\rv}{{\bf r}}
\newcommand{\pv}{{\bf p}}

\newcommand{\cF}{{\mathcal F}} 
\newcommand{\cN}{\mathcal{N}}
\newcommand{\cC}{\mathcal{C}}

\newcommand{\oF}{\overline{ {F}}}

\makeatletter
\newcommand{\hathat}[1]{%
	\begingroup%
	\let\macc@kerna\z@%
	\let\macc@kernb\z@%
	\let\macc@nucleus\@empty%
	\hat{\mathchoice%
		{\raisebox{.2ex}{\vphantom{\ensuremath{\displaystyle #1}}}}%
		{\raisebox{.2ex}{\vphantom{\ensuremath{\textstyle #1}}}}%
		{\raisebox{.16ex}{\vphantom{\ensuremath{\scriptstyle #1}}}}%
		{\raisebox{.14ex}{\vphantom{\ensuremath{\scriptscriptstyle #1}}}}%
		\smash{\hat{#1}}}%
	\endgroup%
}
\makeatother

\newcommand{\mm}[1]{{\color{red}#1}}

\newcommand{\so}[1]{{\color{cyan}#1}}

\title{Heat flow in
   a periodically forced, unpinned thermostatted chain
\thanks{T.K.~acknowledges the support of the NCN grant 2020/37/B/ST1/00426.
S.O.~acknowledges the support of the Institut Universitaire de France. This work has been partially supported by the project CONVIVIALITY ANR-23-CE40-0003, and by the project MICMOV ANR-19-CE40-0012, of the French National Research Agency.}}

\author{%
  Tomasz Komorowski\address{ Institute of Mathematics, Polish Academy
  Of Sciences, Warsaw, Poland.}
    \email{tkomorowski@impan.pl}
  \and 
  Stefano Olla\address{Universit\'e Paris-Dauphine, PSL Research University,
    CNRS, CEREMADE, 
    \emph{and}  Institut Universitaire de France \emph{and} GSSI, L'Aquila} 
   \email{olla@ceremade.dauphine.fr}\and
  Marielle Simon\address{Universite Claude Bernard Lyon 1, CNRS, Centrale Lyon, INSA Lyon, Universit\'e Jean Monnet, ICJ UMR5208,
  	69622 Villeurbanne, France
    \emph{and} GSSI, L'Aquila}
    \email{msimon@math.univ-lyon1.fr}}

\begin{document}

\maketitle

\begin{abstract}We prove the hydrodynamic limit for a one-dimensional harmonic
	chain of interacting atoms with a
	random flip of the momentum sign.
	The system is open: at the left boundary it is attached to a heat
	bath at temperature $T_-$, while at the right endpoint it is subject to
	an action of a force which reads as $\bar F +  \frac 1{\sqrt n}
	\widetilde{\mathcal F} (n^2 t)$, where $\bar F \ge0$ and
	$\widetilde{\mathcal F}(t)$ is a periodic function. Here
	$n$ is the size of the microscopic system.
	Under a diffusive scaling of space-time, we prove that the empirical profiles of the two
	locally conserved quantities -- the volume stretch and the energy --
	converge, as $n\to+\infty$, to the solution of a non-linear
	diffusive system of conservative partial differential equations
	with
	a Dirichlet type and Neumann boundary  conditions on the left and
	the right endpoints, respectively.  \end{abstract}

\section{Introduction}
\label{sec:intro}

In the thermodynamics theory of energy transport there is an important distinction between the \emph{mechanical work}
done into the system and the \emph{heat}, \textit{i.e.}~the thermal energy exchanged with the external heat bath.
Our goal is to investigate the emergence of such a difference from a
microscopic dynamics where the
energy transport is diffusive. On the microscopic scale we make a distinction between the
\emph{mechanical energy}, that is concentrated on the long waves (low modes), and the
thermal energy, distributed on the short waves (high modes).
We consider a one-dimensional finite system 
where the bulk of the dynamics has a chaotic mechanism (in the present case a random flip of the sign of the velocities)
which transforms the energy from the low modes, corresponding to the mechanical energy transport,
to the higher ones, transporting the thermal energy.
At the two boundaries of the system the energy is exchanged with various mechanisms.
On the left endpoint there is a heat bath
at temperature $T_-$, modelled by a Langevin random dynamics.
On the
right, a time dependent deterministic force is applied.
This forcing is periodic and contains different frequencies:
\emph{low}
frequencies whose corresponding work contributes to the exchange of the mechanical energy,
while the higher frequencies exchange the
thermal energy, but in a very different way with respect to the action
of a heat bath.

\medskip

More precisely, the bulk dynamics is given by an \emph{unpinned} chain of harmonic oscillators
perturbed by the \emph{velocity-flip} stochastic mechanism:
each particle flips the velocity sign at independent exponential
times {with intensity $\gamma>0$}. The energy is conserved in the bulk,
but since the chain is unpinned and the random mechanism acts only on the velocities, also the volume is conserved.  After a diffusing scaling of space and time, 
the  macroscopic  energy density and  the volume stretch at the point $u\in [0,1]$ at
time $t \geqslant 0$, denoted by $e(t,u)$ and  $r(t,u)$, respectively,
 evolve in the bulk according to the diffusive system:
\begin{equation}
  \begin{split}
    \partial_t r(t,u) & = \frac{1}{2\gamma} \partial_{uu}r(t,u), \\ 
     \partial_t e(t,u) & = \frac{1}{4\gamma}\partial_{uu} \Big( e(t,u) +
   \frac12 r^2(t,u)\Big).
\end{split}
  \label{eq:eqmacro}
\end{equation}
To the solution of \eqref{eq:eqmacro} one can assign  the  corresponding evolution of the temperature profile (or twice the termal energy)
defined by $T(t,u) = e(t,u) - \frac 12 r(t,u)^2$, which is given by
\begin{equation}
  \label{eq:macroT}
  \partial_t T(t,u)  = \frac{1}{4\gamma}\partial_{uu} T(t,u) + \frac{1}{2\gamma} \Big(\partial_u r(t,u)\Big)^2.
\end{equation}
The diffusive system \eqref{eq:eqmacro} or \eqref{eq:macroT}  is a special case of the general
diffusive evolution for systems with multiple conserved quantities (see \cite{olla}).

The heat bath (at the left endpoint) and the  work performed upon the system (at the right endpoint) affect the
macroscopic boundary conditions for the equations
\eqref{eq:eqmacro} and \eqref{eq:macroT}.
Since the chain is unpinned, we have $r(t,0) = 0$. Besides, the heat bath fixes the temperature at $u=0$, resulting in the
Dirichlet boundary condition: $T(t,0) = e(t,0) = T_-$. 
The effect of the forcing on the other boundary is more complex, and
constitutes the main result of the paper.
As we have already mentioned the forcing acting on the right
  endpoint  is time dependent and periodic. It has a slow part {which varies} on the macroscopic time scale,
and a fast part {which evolves} on the microscopic   scale, namely it reads:
\begin{equation}
  \label{eq:6}
  \mathcal F_n(t) = \bar F(t) +  \frac 1{\sqrt n} \widetilde{\mathcal F} (n^2 t).
\end{equation}
Here $\widetilde{\mathcal F} (t)$ is a smooth periodic function of period $\theta$ and null average.
The factor $1/{\sqrt n}$ is a necessary scaling in order to have a finite  effect on
the macroscopic time scale. In what follows 
 we shall assume, with no loss of generality, that $\bar F$ is constant in time.
This  establishes the boundary condition for the stretch: $r(t,1) = \bar F$.

Moreover, the total work done  by the force at the macroscopic time $t$ is given by
\begin{equation}
  \label{eq:workbis}
  W(t):=
  \frac {\bar F}{2\gamma} \int_0^t (\partial_u r)(s,1)\; \dd s + \mathbb W^Q t,
\end{equation}
where $r(t,u)$ is the solution of \eqref{eq:eqmacro} with the boundary conditions $r(t,0) = 0, r(t,1) = \bar F$,
and $\mathbb W^Q$ is the contribution coming from the fast
fluctuating part $\widetilde{\mathcal F}$ and it is expressed by \eqref{eq:17}.
This yields the boundary conditions for the energy
evolution, or the temperature profile $T(t,u)$, at $u=1$:
\begin{equation}
  \label{eq:15}
  \partial_u e(t,1) 
  = {\oF} \partial_u r(t,1) +{4\ga}  \mathbb W^Q, \quad \text{or}\quad
  \partial_u T(t,1)  = {4\ga} \mathbb W^Q.
\end{equation}
{
  The boundary conditions can be obtained from \eqref{eq:eqmacro} by computing
  the total change of the energy (that is conserved in the bulk):
  \begin{equation}
    \label{eq:13}
    \begin{split}
      \int_0^1 \left[e(t,u) - e(0,u)\right] \dd u = J_0(t) - J_1(t),  
    \end{split}
  \end{equation}
  where $J_0(t)$ is the total flow of energy up to time $t$ entering the system from the thermostat
  on the left, while $J_1(t) = -W(t)$ is the flow of energy on the
  right. It equals  the work performed by the force, with the minus sign.
  From \eqref{eq:eqmacro} we have that
  \begin{equation}
    \label{eq:14}
    W(t) = - J_1(t) = \frac {1}{4\gamma} \int_0^t \left[(\partial_u e)
      (s,1) + \bar F (\partial_u r)(s,1)\right]\; \dd s.
  \end{equation}
  Comparing with \eqref{eq:workbis} we obtain the first boundary condition in \eqref{eq:15}.
  On the other hand, since $e= T + r^2/2$, from \eqref{eq:14} we have
  \begin{equation}
    \label{eq:16}
    W(t) = \frac {1}{4\gamma}\int_0^t \left[(\partial_u T) (s,1) + 2 \bar F (\partial_u r)(s,1)\right]\dd s
  \end{equation}
   and we conclude the second boundary condition in \eqref{eq:15}.
}
In other words, $\frac{\oF}{{2\gamma}} \partial_u r(t,1)$ is the contribution to the mechanical energy,
while $\mathbb W^Q$ is the contribution to the thermal energy (or
heat), due to the work performed by the fluctuating part of the force.
According to \eqref{eq:macroT},  the mechanical energy is
transformed,  in the bulk, into the thermal one at the rate
$\frac{1}{2\gamma} (\partial_u r(t,u))^2$.

\medskip

The result
discussed above has been
announced (without proof) in \cite{klos1}.
The present paper contains its rigorous proof,
under certain condition on the initial distribution of the system.
The usual condition on the initial probability distribution is that the relative entropy
with respect to the equilibrium distribution is bounded by the size
$n$ of the system, see Assumption \ref{ass3}.
We need to supplement it with an assumption about the distribution of the
thermal component of the potential energy, which states that the latter
 should not concentrate too much on the lower modes, see Assumption \ref{A10}.
This hypothesis does not follow from the usual entropy bound, see Remark \ref{rem:ass2}.
It is not clear to us
whether the assumption is indeed necessary for  the proof of the diffusive
hydrodynamic limit in case of more conserved quantities, besides the energy.
On the other hand, the hypothesis is satisfied by \emph{local Gibbs measures}, which are natural initial conditions,
see Remark \ref{rem:ass3}.

Dynamics of harmonic
oscillators perturbed by a conservative noise has recently been a subject of intense study by many authors, see 
\cite{Basile16} and the references therein. In particular, a chain of unpinned harmonic
oscillators with a stochastic perturbation not conserving momentum has \emph{two} conserved quantities
evolving in the same diffusive time scale (see
\cite{olla} and the references therein).  
The hydrodynamic limit has been studied in \cite{ber} and \cite{kos18} in the case of \emph{periodic} conditions,
with two different types of noises.
Open system with Langevin heat baths attached to the boundaries have been studied in \cite{BO, KOS19, kos3}.  
In \cite{kos3} we have considered the case when
heat baths at different temperatures are attached at the boundaries, while
only a constant forcing $\bar F$ is present at the right endpoint.
In that situation, the presence of the heat bath at temperature $T_+$
at the same right point imposes a local equilibrium, with temperature $T_+$.
Consequently, in \cite{kos3} the boundary condition is of the \emph{Dirichlet type},
precisely $ e(t,1)= \frac{\oF^2}{2\gamma} + T_+$.
The situation considered here  is quite different. The temperature at
the right extremity of the chain  is not fixed and
this  results in  the emergence of  the inhomogeneous
\emph{Neumann} boundary condition given in \eqref{eq:15}.
 The case of a \emph{pinned} chain in presence of the periodic forcing
 has been studied in   \cite{klo22} (stationary state) and
 \cite{klo22-2} (nonstationary initial condition).
The pinning destroys the translation invariance of the system and
only the energy keeps conserved by the dynamics.
Consequently, there is no mechanical component of the energy. The
only work that affects the system is performed  by the fluctuating
force with the microscopic time period.

\medskip

We now outline the contents of the paper. In Section \ref{sec:def} we
formulate the model and present the main results, that are shown in
the following sections.
In Sections \ref{sec:evol-aver-diff}--\ref{sec:work} 
we prove the results concerning  the evolution of the mechanical
energy and   the work performed by the forcing.
In  Sections \ref{sec:energy-bounds}--\ref{sec:proof-equipartition} we consider the evolution of the thermal energy.
Similarly to \cite{klo22-2}, in this part  our strategy is based on the full resolution of the
covariance matrix of the momenta and stretches of all particles of the
chain.
It should be noted that this method has been used for the first time
in \cite{RLL67}
to describe the energy distribution in the non-equilibrium stationary
state in a harmonic crystal. Due to the fact that in that case there
is no randomness in the bulk present, the total energy grows
proportionally to  the size of the system.
Contrary to the
case considered in \cite{klo22-2},  the
spectrum of the unpinned chain does not have a spectral gap. As a result
 the resolution of the covariance matrix is technically much more
 challenging here. It results in the already mentioned  additional
Assumption \ref{ass3} on the initial distribution. In particular we need to show
that this hypothesis  is maintained at any positive macroscopic time,
see Corollary \ref{prop012204-24}. Finally, in the Appendix section
we present some technical  results concerning the spectral analysis of
discrete gradient and divergence operators, as well as some auxiliary
facts dealing with the resolution of the aforementioned covariance matrix.

\section{Definition of the dynamics and results}
\label{sec:def}

\subsection{Description of the model}
The configuration of the system is described by
\begin{equation}
  \label{eq:1}
  (\mathbf r, \mathbf p) =
  (r_1, \dots, r_n, p_0, \dots, p_n) \in \Om_n:=\R^{n}\times\R^{n+1},
\end{equation}
where $\mathbf r=(r_1, \dots, r_n)$ and $\mathbf p =(p_0, \dots, p_n)$
correspond to the inter-particle stretches and particle momenta.
The total energy of the chain is given  by the Hamiltonian:
\begin{equation}\label{eq:hn}
  \mathcal{H}_n (\mathbf r, \mathbf p):=
  \sum_{x=0}^n {\cal E}_x (\mathbf r, \mathbf p),
\end{equation}
where the microscopic energy per  atom at any $x\in\{0,\ldots,n\}$ is given by
\begin{equation}
\label{Ex}
{\cal E}_x (\mathbf r, \mathbf p):=  \frac{p_x^2}2 +
\frac{ r_{x}^2}2
\end{equation}
with the convention here and in the following that $r_0:=0, r_{n+1} := 0$.

\medskip

The  microscopic dynamics of
the stretch/momenta process $\{(\mathbf r(t), \mathbf p(t))\}_{t\ge0}$ is 
given in the bulk by:
\begin{equation} 
\label{eq:flipbulk}
\begin{aligned}
  \dd   r_x(t) &=  (p_x(t)-p_{x-1}(t)) \dd t,  \\
  \dd   p_x(t) &=  (r_{x+1}(t)-r_x(t)) \dd t-   2  p_x(t^-) \dd N_x(\gamma t),
  \qquad x=1,\ldots,n -1.
\end{aligned}
\end{equation}
The atom  labelled  $x=0$ is in contact with a Langevin thermostat at temperature $T_-$, while atom $x=n$ is subject to a time-dependent force. Therefore at both boundaries we have
\begin{equation}
\label{eq:flipboundary}
\begin{aligned}
\dd p_0(t) & = r_1(t) \dd t - 2 \gamma p_0(t)\dd t + \sqrt{4\gamma T_-} \dd w_-(t), \\
\dd r_n(t) & = (p_n(t)-p_{n-1}(t)) \dd t,\\
\dd p_n(t) & = (\cF_n(t) -r_n(t)) \dd t - 2 p_n(t^-) \dd N_n(\gamma t).
\end{aligned}
\end{equation}
Hereinabove, $\{N_x(t)\; ; \;x=1,\ldots,n\}_{t\geqslant 0}$
are independent Poisson processes of intensity $1$, while
$\{w_-(t)\}_{t\geqslant 0}$ is a standard one dimensional Wiener process,
independent of the Poisson processes. These processes are defined
over a certain probability space $(\Xi,\mathfrak{F},\bbP)$. Moreover, the parameter $\gamma>0$ 
regulates the intensity of \emph{both} the random perturbations
and the Langevin thermostat\footnote{We have chosen the same parameter in order to simplify notations,
but  it does not affect the results {concerning} the macroscopic
properties of the dynamics. {In particular, for any different constant $\widetilde{\gamma}>0$ regulating the thermostat intensity, the results below would read exactly the same (the constant $\gamma$ there only coming from the flip noise)}.}, while $T_->0$ is the thermostat
temperature. Finally, the time-dependent force $ \cF_n(t)$ is assumed to be a
 smooth, $\theta$-periodic, with $\theta>0$, and it is  of the form
\begin{equation}
  \label{Fnt}
\cF_n(t)=\oF+\tilde{\cF}_n(t),\quad\mbox{where}\quad \tilde{\cF}_n(t):=\frac{1}{\sqrt{n}}\sum_{\substack{\ell \in \Z\\ \ell \not=0}}
\hat{\cF}(\ell)e^{i\ell\om t},\
\end{equation}
 with $\om =\frac{2\pi   }{ \theta}$.
We suppose moreover that $\cF_n(t)$ is real valued, so
$\hat{\cF}(-\ell)=\hat{\cF}^\star(\ell)$, and that the decay of the Fourier
coefficient is sufficiently fast (\textit{e.g.}~exponential) so that
$\cF_n$ is continuous and  also
\begin{equation}
  \label{031409-23}
C_\cF:=  \sum_{\ell \neq 0} |\hat{\cF}(\ell)|<+\infty .
  \end{equation}
Throughout the paper, we will shorten notation by introducing the discrete \emph{gradient} and {\em divergence} operators
 $\nabla:\bbR^{n}\to\bbR^{n+1}$ and
 $\nabla^\star:\bbR^{n+1}\to\bbR^{n}$ acting as follows (see  
 Section \ref{sec:gradients} of Appendix for some properties of these operators)
 \begin{equation}
   \label{011804-24}
   \begin{split}
     \nabla g_x&:=g_{x+1}-g_x,\qquad \text{ for } g=(g_1,\ldots,g_n)\in\R^n,\\
     \nabla^\star f_x&:=f_x-f_{x-1}, \qquad \text{ for } f=(f_0,\ldots,f_n)\in\R^{n+1},
     \end{split}
   \end{equation}
   with the convention $g_0=g_{n+1}=0$ and $f_{-1}=f_0$. 

\medskip

Note that $\{(\rv(t),\pv(t))\}_{t\geqslant 0}$ is a Markov process  on
$\Omega_n $ with the time-dependent generator $\mathcal{G}_t$ which we decompose in three parts as follows:
\begin{equation}
  \label{eq:7}
  \mathcal G_t : =  \mathcal A_t +  \gamma S_{\text{flip}}
  + 2   \gamma S_-,
\end{equation}
where
\begin{equation}
  \label{eq:8}
  \mathcal A_t := \sum_{x=1}^n \nabla^\star p_x \partial_{r_x}
  + \sum_{x=0}^n  \nabla r_{x} \partial_{p_x}
  + {\cal F}_n(t)  \partial_{p_n}.
\end{equation}
In addition, for any $f:\Omega_n\to\bbR$ bounded and measurable function,
\begin{equation}
  \label{eq:21}
   S_{\text{flip}} f (\rv,\pv): = \sum_{x=1}^{n}   \Big( f (\rv,\pv^x) - f (\rv,\pv)\Big),
 \end{equation}
 where  $\pv^x$ is the velocity configuration with sign flipped at the
 $x$ component, \textit{i.e.}~$\pv^x=(p_0',\ldots,p_n')$, with $p_y'=p_y$, for any
 $y\neq x $ and  $p_x'=-p_x$. Furthermore,
 \begin{equation}
   \label{eq:10}
   S_- := T_- \partial_{p_0}^2 - p_0 \partial_{p_0}.
 \end{equation}
One can easily verify that the microscopic \emph{energy currents}
$\{j_{x,x+1};x=-1,\dots,n\}$ which satisfy: 
\begin{equation}
\label{eq:current}
  \mathcal G_t \mathcal E_x  = j_{x-1,x} - j_{x,x+1} ,\quad\mbox{ for any $x=0,\dots,n$},
\end{equation} 
are given by 
\begin{equation}
  j_{x,x+1}(t)=
  \begin{cases} - p_x(t) r_{x+1}(t)  & \text{ if }x =0,...,n-1, \vphantom{\Big(} \\
    2 { \gamma} \left(T_- - p_0^2(t) \right)  & \text{ if } x=-1,\vphantom{\Big(} \\
    -  \cF_n(t)    p_n(t) & \text{ if }x=n.
  \end{cases}
  \label{eq:current-bis}
\end{equation}
We assume that the initial distribution of stretches and momenta
$(\mathbf r(0),\mathbf p(0))$ in $\Om_n$ is
random and distributed according to a probability measure $\mu_n$.
We then denote by $\mu_n(t)$ the probability measure on $\Omega_n$
of the configuration $(\rv(t),\pv(t))$ evolving according
to \eqref{eq:flipbulk}--\eqref{eq:flipboundary}.
Finally we denote by $\bbE_{\mu_n}$ the expectation with respect to the probability
measure $\bbP_{\mu_n}:=\mu_n\otimes \bbP$. {Recall that
$(\Xi,\mathfrak{F},\bbP)$ is the probability space over which both the
Wiener and Poisson processes are defined.}

We decompose the configurations as
\begin{equation}
  \label{011704-24}
   r_x(t):=r_x'(t) + \bar r_x(t) , \qquad p_x(t):=p_x'(t) + \bar p_x(t), \qquad {x =0,\dots,n,}
\end{equation}
where
\begin{equation}
  \label{012911-23} \begin{split}
\bar{\mathbf r}(t)&=(\bar r_1(t),\ldots, \bar r_n(t)):=\bbE_{\mu_n}[\mathbf
r(t)] ,\\ \bar{\mathbf p}(t)&=(\bar p_0(t),\ldots, \bar p_n(t)):=\bbE_{\mu_n}[\mathbf
p(t)], \end{split}
\end{equation}
while $ {\mathbf r}'(t), {\mathbf p}'(t)$ corresponds to the \emph{fluctuating parts} of the dynamics.
We adopt the convention that $r_0(t)=\bar r_0(t)= r_0'(t)\equiv 0$.
Finally, for any measurable $f:\Om_n\to\bbR$ and $t>0$, we introduce the following time average in the diffusive scale:
\begin{equation}
  \lang f\rang_t:=\frac{1}{t}\int_0^t \bbE_{\mu_n}\big[f\big(\rv(n^2s),\pv(n^2s)\big)\big]\dd s,
  \label{eq:bracket-t}
\end{equation}
provided that the integral on the right hand side makes sense.

\subsection{Formulation of results}
\label{sec:results}


Let us first formulate two main assumptions.

\begin{assumption}[Initial bound on the averages]\label{ass1}
We assume that
 there exists $\bar{\cal H}>0$ 
such that
    \begin{equation}
      \label{eq:delta} 
     \frac1{n} \sum_{x=0}^n \left( \bar r_x^2(0) +\bar p_x^2(0) \right) \leqslant
     \bar{\cal H},\qquad n=1,2,\ldots.
    \end{equation}
\end{assumption}
Furthermore, initially, the average profile of stretches
approximates a continuous function -- the \emph{initial stretch profile} $r_0(\cdot)$.  Namely: 
\begin{assumption}[Initial stretch profile]
  \label{ass2}
  We assume that there exists a continuous function $r_0:[0,1] \to \RR$ such that  
  \begin{equation}
  \lim_{n\to+\infty}\frac{1}{n}\sum_{x=1}^{n} \Big(\bar r_x(0) -r_0\Big(\frac{x}{n}\Big)\Big)^2
=0.
 \label{eq:conv-stretch}
\end{equation}
\end{assumption}

\subsubsection{Macroscopic evolution of stretch and mechanical energy}
\label{sec:volume-stretch-equat}

{Before formulating the main result we introduce the notion of a
measure-valued solution of the following Cauchy-Dirichlet boundary 
value problem. Suppose that $\bar F$ is a certain constant. Let $\mathcal M_{\rm fin}([0,1])$ be the space of signed measures $m$ on
$[0,1]$, with finite total variation, endowed with the topology of weak
convergence, and suppose that $r_0\in \cC([0,1])$. 
 \begin{definition}
  \label{df012701-23}
{We say that
a function $r:[0,+\infty)\to \mathcal M_{\rm fin}([0,1])$ is a weak (measure-valued) solution of the initial-boundary problem
\begin{equation}
  \begin{split}
    \partial_t r(t,u) & = \frac{1}{2\gamma} \partial_{uu}r(t,u), \qquad  (t,u) \in \bbR_+ \times (0,1),\\
    r(t,0) & = 0, \qquad r(t,1) = \oF, \qquad r(0,u) = r_0(u).
\end{split}
  \label{eq:HLstretch}
\end{equation}
if:  it belongs to
$\cC\big([0,+\infty); \mathcal M_{\rm fin}([0,1])\big)$ and
  for any $\varphi\in \cC^2([0,1])$ such that 
     $\varphi(0)=\varphi(1)=0$ we have
  \begin{equation}
    \label{eq:5w}
    \begin{split}
       \int_0^1\varphi(u)r(t,\dd u) - \int_0^1\varphi(u)r_0(u) \dd u 
      =\frac{1}{2\gamma} \int_0^t\dd s\int_0^1  \varphi''(u)r(s,\dd u)
      -\frac{\oF t}{2\ga}\varphi'(1).
    \end{split}
  \end{equation}}
\end{definition}
Using an analogous argument to the one  used in \cite[Appendix
E]{klo22-2} one can prove  the uniqueness of the solution of
\eqref{eq:5w}. In fact  the solution  is absolutely continuous with
respect to the Lebesgue measure and its density
is  $\cC^\infty$ smooth  on $(0,+\infty)\times[0,1]$. From this point on
we shall identify the solution with its density $r(t,u)$.}

Let us define  the microscopic mechanical energy as
\begin{equation}
  \label{eq:31}
  \E_x^{\text{mech}}(t): = \frac 12 \left(\bar p_x^2(t) + \bar
    r_x^2(t)\right),\qquad x=0,\ldots,n.
\end{equation}
Under the assumptions formulated in the foregoing
we obtain the following macroscopic limits:

\begin{theorem}[Limit of stretch and mechanical energy]
  \label{th-r}
  Assume Assumptions \ref{ass1} and \ref{ass2}. Then, for any
  continuous test function $\varphi:[0,1]\to\R$, $t>0$
  \begin{equation}
  \lim_{n\to+\infty}\frac{1}{n}\sum_{x=1}^{n} 
  \varphi\left(\frac {x}{n} \right) \bar r_x(n^2t) =\int_0^1 r(t,u)\varphi(u)\; \dd u
  \label{eq:conv-stretch-11}
\end{equation}
and
  \begin{equation}
  \label{eq:32}
   \lim_{n\to+\infty}\frac 1{n} \sum_{x=0}^{n} \varphi\left(\frac{x}{n}\right)
  \E_x^{\mathrm{mech}}(n^2t) =   \int_0^1 \varphi(u) \tfrac12  r^2(t,u) \dd u,
\end{equation}
where $r(t,u)$ is the solution of \eqref{eq:HLstretch}.

\end{theorem}

Theorem \ref{th-r} is proved in Section \ref{sec:evol-stretch}.

\subsubsection{Macroscopic work}
\label{sec:work-1}

Let $W_n(t)$ be the \emph{average work done by the force in the diffusive time scale}, namely:
\begin{equation}
  \label{eq:work}
  W_n(t) = \frac 1n \int_0^{tn^2} \cF_n(s) \bar p_n(s) \dd s.
\end{equation}
Observe that, with our notation \eqref{eq:bracket-t}, we have
$W_n (t) = - nt  \lang j_{n,n+1}\rang_t$.

\begin{theorem}
  \label{thm012209-23}
Under Assumptions \ref{ass1} and \ref{ass2} we have, {for any $t>0$,}
\begin{equation}
  \label{012209-23}
  \lim_{n\to+\infty} W_n(t)= W(t):=
  \frac {\bar F}{2\gamma} \int_0^t (\partial_u r)(s,1)\; \dd s
  + \mathbb W^Q t,
\end{equation}
with
 \begin{equation}
   \label{eq:17}
    \mathbb W^Q   = \sum_{\ell=1}^{+\infty}  |\hat{\cF}(\ell)|^2 \; (\ell\om)  {\rm Re}\,\sqrt{{\frac{4}{(\ell\om)^2- i 2\gamma \ell\om}-1}}.
\end{equation}
Here $w\mapsto \sqrt{w}$ is  the branch of the inverse of $z\mapsto z^2$
  such that ${\rm Re}\, \sqrt{w}>0$, when
  $w\in\mathbb C\setminus(-\infty,0]$.
\end{theorem}

The proof of Theorem \ref{thm012209-23} is given in Section \ref{sec:work}.

\begin{remark}
  Since $(\partial_u r)(s,1) \to \bar F$, as $s\to\infty$,
  the macroscopic work satisfies
 \[
\lim_{t\to+\infty}\frac{W(t)}{t}=\frac{\oF^2}{2\ga}+\mathbb W^Q.
\]
\end{remark}
\begin{remark}
 Notice that when $\gamma\to 0$ we have
  \begin{equation*}
     \mathbb W^Q \longrightarrow \sum_{\ell=1}^{[\frac 2\omega]}
    |\hat{\cF}(\ell)|^2 \; \,\sqrt{4- (\ell\omega)^2},
 \end{equation*}
 in agreement with the calculations for the deterministic dynamics (\emph{i.e.}~$\gamma=0$) in
 \cite{KGLO}. 
  \end{remark}

\subsubsection{The macroscopic limit of the  energy functional}
\label{sec:thermal-energy}

In order to obtain the macroscopic energy profile,
we need two additional assumptions: 1) on the initial \emph{entropy}
and 2) on the initial  \emph{stretch fluctuations}.

In order to introduce them, let us define $\nu_{T_-} (\dd{\bf r},\dd{\bf p})$ as the product
  Gaussian measure on $\Omega_n$
 of zero average and variance $T_->0$ given by
\begin{equation} \label{eq:nuT}
 \nu_{T_-} (\dd{\bf r},\dd{\bf p}) : =g_{T_-}(\br,\bp) \dd \br\dd \bp 
  \quad \mbox{where }g_{T_-}(\br,\bp)=\frac{e^{-\mc E_0/T_-}}{\sqrt{2\pi T_-}}\prod_{x=1}^n \frac{e^{-\mc E_x/T_-}}{2\pi T_-} .
\end{equation}
Let $f_n(t,\mathbf{r},\mathbf{p})$
be the density of $\mu_n(t)$ with respect to $\nu_{T_-}$.
We can now define the relative entropy of $\mu_n(t)$ with respect to $\nu_{T_-}$ as
\begin{equation}
  \label{eq:7-1}
 {\mathbf{H}}_n(t) := \int_{\Om_n}  f_n(t) \log f_n(t) \dd {\nu_{T_-}}.
\end{equation}

\begin{assumption}\label{ass3} 
  We assume that the initial measure $\mu_n$ is such that $f_n(0)$
  is {of class $\cC^2$} on $\Omega_n$, and there exists
  $C>0$ such that,  
  \begin{equation}
    \label{eq:ass2entropy}
    {\mathbf{H}}_n(0) \le C n,\qquad n=1,2,\ldots.\end{equation}
\end{assumption}
Assumption \ref{ass3} implies the following initial bound on the
energy:  there exists $C>0$ such that
\begin{equation}\label{eq:energyb}
 \mathbb E_{\mu_n}  \big[\mathcal{H}_n (0)\big]
  \le Cn,\qquad n=1,2,\ldots
\end{equation}
Here $\mathcal{H}_n (t) =\mathcal{H}_n (\mathbf r(t), \mathbf p(t))$ (recall \eqref{eq:hn}).
 Indeed, the entropy inequality, see \emph{e.g.}~
  \cite[p.~338]{KL}, gives
  \begin{equation}
\label{entropy-in-tz}
\begin{split}
  \mathbb E_{\mu_n} \big[ \mathcal{H}_n (0) \big]& =\int_{\Om_n}
\Big(\sum_{x=0}^n{\cal E}_x\Big)   f_n(0) \dd
 \nu_{T_{-}}
 \\
 &
 \le \frac{1}{\al}\left\{\log\bigg(\int_{\Om_n}
     \exp\left\{ \frac\al 2 \sum_{x=0}^n(p_x^2+r_x^2)
     \right\}\dd  \nu_{T_{-}} \bigg)+ {\mathbf{H}}_{n}(0)\right\}
\end{split}
 \end{equation}
for any $t\ge0$ and $\al>0$. Hence, we obtain \eqref{eq:energyb}  choosing $\al\in(0,T_-^{-1})$.

\medskip

We need a further assumption on the initial condition, that involves a weak decorrelations of the initial
stretches $r_x(0)$. Define (recall \eqref{011704-24})
\begin{equation}
  \label{eq:64}
  q'_x(t) = \sum_{y=1}^x r'_y(t), \quad x=1,\ldots,n\qquad\mbox{and}\qquad q_0'(t) \equiv 0.
\end{equation}

\begin{assumption}\label{A10}
We assume that
  \begin{equation}
    \lim_{n\to\infty} \frac 1{n^3} \sum_{x=1}^n \bbE_{\mu_n} \left[ \big(q'_x(0)\big)^2 \right] = 0.
 \label{060510-23}
\end{equation}
\end{assumption}
\begin{remark}\label{rem:ass2}
Assumption \ref{A10} implies that 
the potential energy should not concentrate too  much on the lower thermal modes (this will become more precise in the estimate
\eqref{eq:62} below). The entropy bound
\eqref{eq:ass2entropy} by itself does not prevent such concentration.
In fact finitness of the initial energy (that follows from \eqref{eq:ass2entropy}) implies only that
\begin{equation*}
  \sup_{n\ge 1} \frac 1{n^3} \sum_{x=1}^n \bbE_{\mu_n} \left[ \big(q'_x(0)\big)^2 \right] <+\infty.
\end{equation*}
We will need to prove that \eqref{060510-23} is maintained by the dynamics at all times.
This is  the goal of Proposition \ref{prop010510-23} below.
\end{remark}

\begin{remark}\label{rem:ass3}
 Note that  Assumption \ref{A10}  is satisfied by   the \emph{local
   Gibbs measures}, whose density (with respect to the Lebesgue measure) takes the form
  \begin{equation}
    \label{eq:12}
    \frac{e^{-\mc E_0/T_-}}{\sqrt{2\pi T_-}}\prod_{x=1}^n \frac{e^{-\mc E_x/T_x}}{2\pi T_x},
  \end{equation}
  where $(T_x)_{x\in\bbZ}$ is a bounded sequence of positive numbers.
  Indeed, under this distribution $r_1'(0),\ldots,r_n'(0)$ are independent, centered
Gaussian random variables,  and $\bbE [ q'_x(0)^2 ] = \sum_{y=1}^x T_y \le n \max_x T_x$.  
\end{remark}

{Before formulating our result on  the macroscopic behavior of the
energy profile we need the notion of a weak solution of the respective
initial boundary problem that shall be used in the limit
identification of the limiting energy profile. 
\begin{definition}
  \label{df012002-25}
  Let  $W(t)$ be the macroscopic work given in Theorem
\ref{thm012209-23}, while $r(t,u)$ is the macroscopic stretch profile
as in the statement of Theorem   \ref{th-r}. Suppose also that $e_0\in \cC([0,1])$.
We say that
a function $e:[0,+\infty)\to \mathcal M_{\rm fin}([0,1])$ is a weak (measured
valued) solution of the initial-boundary problem
\begin{equation}\label{eq:pde-energy}
  \begin{split}
   & \partial_t e(t,u) = \frac{1}{4\gamma}\partial_{uu} \Big( e(t,u) +
   \frac12 r^2(t,u)\Big),\qquad (t,u)\in \R_+\times(0,1),\\
&
e(t,0)  = T_-, \qquad \partial_u e(t,1) 
= {\oF} \partial_u r(t,1) + {4\ga} \mathbb W^Q,
\qquad
  e(0,u) = e_0(u),
  \end{split}
\end{equation}
if:  it belongs to
$\cC\big([0,+\infty); \mathcal M_{\rm fin}([0,1])\big)$ and
  for any $\varphi\in \cC^2[0,1]$ such that 
     $\varphi(0)=\varphi'(1)=0$ we have
 \begin{multline}
 	\label{eq:59z}
 	   \int_0^1\varphi(u)e(t,\dd u) - \int_0^1\varphi(u)e_0(u) \dd u -
 		\frac 1{4\gamma} \int_0^t\varphi''(u) e(s,\dd u)\dd s  \\
 		 =  \frac 1{8\gamma}\int_0^t \dd s \int_0^1 \dd u\; \varphi''(u) r^2(s,u)
 		+  \varphi'(0) \frac{T_-}{4\gamma} + \varphi(1) W(t).
 	\end{multline}
\end{definition}
Again, using a standard argument, see \textit{e.g.}~\cite[Appendix
E]{klo22-2} one can prove  that   the solution of
\eqref{eq:pde-energy} is unique, absolutely continuous with
respect to the Lebesgue measure and its density
is  $\cC^\infty$ smooth  on $(0,+\infty)\times[0,1]$.}

The macroscopic behavior of the energy profile is given by the following:

\begin{theorem}[Limit of total energy] \label{th:hl}
 Suppose that  Assumptions \ref{ass1}, \ref{ass2}, \ref{ass3}, and \ref{A10}  
 are in force. Assume furthermore that there exists a continuous
 function (the initial energy profile)
  $e_0:[0,1]\to (0,+\infty)$ such that, for any continuous test function $\varphi:[0,1]\to\R$,
  \begin{align}
    \label{E0}
  \lim_{n\to+\infty}\frac{1}{n}\sum_{x=0}^{n}
   \varphi\left(\frac x{n}\right) \bbE_{\mu_n}\big[{\cal E}_x(0)\big]
  =\int_0^1 e_0(u)\varphi(u)\dd u.
\end{align}
Then for any continuous  $\varphi:[0,1]\to\R$ and any $t\geqslant 0$, 
\begin{align}
  \lim_{n\to+\infty}\frac{1}{n}\sum_{x=0}^{n} 
  \varphi\left(\frac x{n} \right) \bbE_{\mu_n}\big[{\cal E}_x(n^2 t)\big]&=\int_0^1 e(t,u)\varphi(u)\dd u,
 \label{eq:conv-energy}
\end{align}
where $e(t,u)$ is the solution of the initial-boundary value problem \eqref{eq:pde-energy}.

\end{theorem}

We conclude this section by stating a last result on the macroscopic \emph{thermal energy} behavior. Let us define
  the macroscopic \emph{temperature} (or thermal energy) profile 
  $T(t,u) := e(t,u) - \frac 12 r^2(t,u)$. From \eqref{eq:HLstretch} and \eqref{eq:pde-energy} it satisfies the following initial-boundary value problem
 \begin{equation}
    \label{eq:HLtemp}
    \begin{split}
    \partial_t T(t,u) & =\frac{1}{4\gamma} \partial_{uu}T(t,u) +
                         \frac{1}{2\gamma}(\partial_u r(t,u))^2, \qquad (t,u)\in \R_+\times(0,1),\\ 
T(t,0)&=T_-,  \qquad \partial_u T(t,1) = {4\ga}\mathbb W^Q 
,  \qquad T(0,u) = T_0(u).\end{split}
\end{equation}
{
The notion of a weak solution of \eqref{eq:HLtemp} can be formulated
analogously as in Definition \ref{df012002-25}.}
\begin{remark}
The quantity   $ \frac{1}{2\gamma}(\partial_u r(t,u))^2$ appearing in \eqref{eq:HLtemp}
represents the local rate of conversion of the  mechanical energy into the thermal one,
\emph{i.e.}~a gradient of the volume stretch generates a local increase of
the thermal component of the energy.
\end{remark}

Let us now define, for any $t\ge 0$ and $x=0,\dots,n$,
\begin{equation}\label{eq:hn1}
 {\cal E}_x'(t) :=\frac{ (r_x'(t))^2}2+ \frac{(p_x'(t))^2}{2}.
\end{equation}
\begin{theorem}[The limit of thermal energy and equipartition]
  \label{thm012911-23}
  For any continuous test function $\varphi:[0,1]\to\bbR$ and any $t\ge 0$, we have
\begin{align}
   \lim_{n\to+\infty}\frac{1}{n}\sum_{x=0}^{n}  \varphi\left(\frac x{n} \right)
 \bbE_{\mu_n}\big[ {\cal E}_x'(n^2 t)\big] =\int_0^1 T(t,u)\varphi(u)\dd u,
 \label{eq:conv-temp}
\end{align}
{where $T(t,u)$ is the solution of \eqref{eq:HLtemp}.}

In addition, for any compactly supported,
continuous function $\Phi:\R_+\times [0,1]\to\bbR$
\begin{equation}
  \begin{split}
  &\lim_{n\to+\infty}\frac{1}{n}\sum_{x=0}^{n} \int_{\R_+} \Phi\left(t,\frac x{n} \right)
 \bbE_{\mu_n}\big[ p_x^2(n^2 t)\big] \dd t
 =
   \lim_{n\to+\infty}\frac{1}{n}\sum_{x=0}^{n} \int_{\R_+} \Phi\left(t,\frac x{n} \right)
 \bbE_{\mu_n}\big[ {\cal E}_x'(n^2 t)\big] \dd t.  
 \label{eq:conv-temp1}
\end{split}
\end{equation}
\end{theorem}

Both Theorems \ref{th:hl} and \ref{thm012911-23} will be proved in
Section \ref{sec:proof-macr-energy}
\textit{modulo}
several
technical results which will be shown in Sections \ref{sec:cov} and \ref{sec:proof-equipartition}.

\section{Evolution of the averages in diffusive scaling}
\label{sec:evol-aver-diff}

This section contains some results on the \emph{averages}
$(\bar{\br}(t),\bar{\bp}(t))$. They are obtained   thanks to an explicit
resolution of the system of equations satisfied by the averages. We start with the following:

\begin{proposition}\label{prop:main}
We assume Assumption \ref{ass1}. {Then, there} exists a constant $C=C(\ga)>0$ such that, for any $n\ge 1$, any $t\ge 0$, we have

\begin{enumerate}[(i)]
\item  \emph{(Control of $\ell^2$ norms)}
  \begin{equation}
    \label{012409-23main}
   \sum_{x=0}^n \bar p_x^2( t)  \le  Cn, \quad \text{ and } \quad  \sum_{x=0}^n\int_0^t\bar p_x^2(n^2s) \dd s\le
\frac{C(t+1)}{n},
\end{equation} 
and 
    \begin{equation}
    \label{012409-23main1}
    \sum_{x=1}^n \bar r_x^2( t)  \le  Cn.
  \end{equation}
{  In addition, for any $t>0$ we have
  \begin{equation}
    \label{012409-23main2}
   \lim_{n\to+\infty}\sum_{x=0}^n \bar p_x^2( t)  =0.
\end{equation} }
\item \emph{(Boundary behavior)}
\begin{equation}
   \label{eq:18main}
   \left|\int_0^t \bar p_0(n^2 s) \dd s \right|\le \frac{C(t+1)}{n}, \qquad  \left|\int_0^t \bar p_n(n^2 s) \dd s \right|\le \frac {C(t+1)} n,
 \end{equation}
and  furthermore, for any $t>0$
 \begin{equation}
   \label{eq:22zmain}
   \lim_{n\to\infty}  \lang r_n \rang_t = \lim_{n\to\infty} \frac{1}{t}\int_0^{t}
   \bar r_n( n^2 s)  \dd s = \bar F.
 \end{equation}
 \end{enumerate}
\end{proposition}

The proofs of \eqref{012409-23main} and \eqref{012409-23main1}  are presented in Sections
\ref{sec3.2.1} and \ref{sec3.3}, respectively. The proof of
\eqref{eq:18main} and \eqref{eq:22zmain} is given in Section  \ref{sec:bound}.
We start with some preliminaries. 
The strategy of the proof consists in solving explicitly the dynamics satisfied by the averages in terms of Fourier transforms. Indeed, from \eqref{eq:flipbulk}--\eqref{eq:flipboundary}   we get
\begin{equation}
  \label{eq:pbdf1b}
  \begin{split}
     \frac{\dd   \bar r_x(t)}{\dd t} &=  \nabla^\star\bar p_x(t),\qquad x=1,\ldots,n,\\
\frac{\dd\bar{  p}_x(t)}{\dd t} &=   \nabla\bar r_x(t)  -   2
\ga \bar p_x(t) + \delta_{x,n}{\cal F}_n(t)  ,\qquad x=0,\ldots,n ,
\end{split}
\end{equation}
with the convention $\bar r_0\equiv 0$ and $\bar r_{n+1}\equiv 0$. We denote by $\delta_{x,y}$ the
	Kronecker delta function, namely $\delta_{x,y}=0$, if $x\not= y$  and  $\delta_{x,x}=1$. 
      Its generator can be written in the form
  \begin{align}
    \label{032004-24}
    \bar{\mathcal G}_t : =  \bar{\mathcal A}_t +  \gamma \bar{S}
      \end{align} 
  where
    \begin{equation*}
       \bar{\mathcal A}_t := \sum_{x=1}^n \nabla^\star \bar p_x \partial_{\bar r_x}
  + \sum_{x=0}^n  \nabla \bar r_{x} \partial_{\bar p_x}
  + {\cal F}_n(t)  \partial_{\bar p_n}, \qquad  \bar{S}: =-2\sum_{x=0}^n\bar p_x\partial_{\bar p_x}.
    \end{equation*}
System \eqref{eq:pbdf1b} can be rewritten in the matrix form
\begin{equation}
\label{011406-22}
\frac{\dd }{\dd t}\begin{pmatrix}\bar{\bf r}(t)\\
    \bar{\bf p} (t)\end{pmatrix}=-A \begin{pmatrix}\bar{\bf r} (t)\\
    \bar{\bf p} (t)
    \end{pmatrix}+ 
\;\cF_n(t){\mathbf e}_{p,n}.
\end{equation}
Here $A$ is the block matrix of the form
\begin{equation}
\label{bA}
  A:=
\begin{bmatrix}
    0_n&-\nabla^\star\\
   -\nabla& 2\ga {\rm Id}_{n+1}
\end{bmatrix},
\end{equation}
where $\nabla$ and $\nabla^\star$ are the matrices corresponding to
the discrete divergence and gradient operators, see \eqref{011804-24}, defined as:
  \begin{equation}
  \label{eq:nablamatrices}
\nabla^\star  = \underbrace{\begin{pmatrix}-1&1&0&\cdots& \cdots &0\\
0 & -1 & 1 & \ddots & &\vdots \\
   \vdots&\ddots &\ddots&\ddots& \ddots & \vdots \\
                        \vdots&& \ddots &-1&1&0\\
0                      & \cdots& \cdots &0&-1&1
    \end{pmatrix}}_{n+1} \left. \; \rule{0pt}{11.5mm} \right\} {\text{\tiny $n$}}, \qquad  \nabla  = \underbrace{\begin{pmatrix}1&0&\cdots & \cdots & 0 \\
-1 & 1 & \ddots  &  &  \vdots \\
0 & \ddots & \ddots & \ddots & \vdots \\
   \vdots&\ddots &\ddots&\ddots& 0  \\
                        \vdots&&  \ddots  &-1&1\\
0                      & \cdots & \cdots &0&-1
    \end{pmatrix}}_{n} \left. \; \rule{0pt}{15.5mm} \right\} {\text{\tiny $n\!+\!1$},}
  \end{equation}
 and $0_n$,  ${\rm Id}_{n+1}$  are  the
$n\times n$--null and $(n\!+\!1)\times (n\!+\!1)$--identity matrices
respectively. 
In \eqref{011406-22} the column vector $\mathbf e_{p,n}$ of size $2n+1$ is given by 
\begin{equation}
\label{epq}
({\mathbf e}_{p,n})_x=\delta_{2n+1,x}, \qquad x=1,\dots,  2n+1.
\end{equation}
The solution of  \eqref{011406-22} is therefore
\begin{equation}
  \label{011605-22r}
  \begin{pmatrix}\bar{\bf r}(t)\\
    \bar{\bf p}(t)
    \end{pmatrix}=e^{-  At}\;
\begin{pmatrix}\bar{\bf r}(0)\\
    \bar{\bf p}(0)
  \end{pmatrix}
  +   \int_0^t
\cF_n\left(s\right) e^{-  A(t-s)}{\mathbf e}_{p,n}\dd s.
\end{equation}

\subsection{Homogeneous evolution and key lemma}
Let us define the homogeneous term as
\begin{equation}
  \label{010412-23}
\begin{pmatrix}
\bar {\bf y} (t) \\
    \bar {\bf z} (t)
\end{pmatrix}:=e^{-  At}\; \begin{pmatrix}
\bar{\bf r}(0) \\ \bar{\bf p}(0)
\end{pmatrix}, \qquad \text{with } \quad \bar{\bf y}(t) \in \R^n, \bar{\bf z}(t) \in \R^{n+1}.
\end{equation}
We first obtain an immediate bound on the $\ell^2$-norm
of the averages for the homogeneous term:
	\begin{lemma}[$\ell^2$-norm bound] 
		\label{barr-p}
		\begin{equation}
			\label{011109-23}
			\sum_{x=0}^n \bar z_x^2(t)+\sum_{x=1}^n \bar y_x^2(t) \le
			\Big( \sum_{x=0}^n \bar p_x^2(0)+ \sum_{x=1}^n \bar r_x^2(0)\Big).\end{equation}
	\end{lemma}
	\proof
	This follows immediately by taking the time derivative of the mechanical energy:
	\begin{equation*}
		\frac \dd{\dd t} \left(\sum_{x=0}^n \bar z_x^2(t)+\sum_{x=1}^n \bar y_x^2(t) \right) =
		-4\gamma 	\sum_{x=0}^n \bar z_x^2(t) \le 0.
	\end{equation*}
\qed

Now, let us introduce  the respective Fourier transforms of $\bar{\bf y}$ and
$\bar{\bf z}$ in the two orthonormal  basis $\{\psi_j\}_{j=0,\dots,n}$ and $\{\phi_j\}_{j=1,\dots,n}$ defined in \eqref{eq:vecneu} and \eqref{eq:vecdir}, namely:
\[
\tilde{\bf y} (t)= 
\begin{pmatrix}\tilde y_1 (t)\\
    \vdots
    \\
    \tilde y_n (t)
    \end{pmatrix},\qquad \tilde{\bf z} (t)=
\begin{pmatrix}\tilde z_0 (t)\\
    \vdots
    \\
     \tilde z_n  (t)
    \end{pmatrix},
  \]
  where
  \[
  \tilde y_j (t):=\sum_{x=1}^n \bar y_x (t)\phi_j(x),\quad \text{and} \quad \tilde z_j (t):=\sum_{x=0}^n \bar z_x (t)\psi_j(x),
\]
Solving the systems of equations for the Fourier transforms, that arises from \eqref{eq:pbdf1b}, a direct computation detailed in Appendix \ref{sec:gradients} (see
\eqref{gaps} and \eqref{gaph} in particular), leads to the following system: for $j=0,\dots,n$ 
\begin{align}\label{eq:ode2}
\frac{\dd^2\tilde z_j (t)}{\dd t^2}+2\ga \frac{\dd\tilde
 z_j (t)}{\dd t}+\la_j\tilde z_j(t)=0,
  \end{align} with the initial conditions \begin{align}
  \frac{\dd \tilde z_j}{\dd t}(0)  = \la_j^{1/2}\,\tilde r_j(0) - 2\gamma \,
                                     \tilde p_j(0),\qquad \tilde z_j(0) =\tilde p_j(0),                                 
\end{align}
where $\la_j$ is the square eigenvalue (see \eqref{eq:70bis})
\begin{equation}
  \label{eq:70}
  \la_j:=   4\sin^2\Big(\frac{j\pi}{2(n\!+\!1)}\Big).
\end{equation}
The characteristic equation of \eqref{eq:ode2} takes the form
$
\la^2+2\ga \la+\la_j=0,
$
which yields two solutions $-\la_{j,\pm}$, where 
\[
\la_{j,\pm}:=\ga\pm\sqrt{\ga^2-\la_j}.
\]
Note the following important relations: 
\begin{align} \label{eq:relat}
    &|\la_{j,+}\la_{j,-}|=\la_j, &\la_{j,+}+\la_{j,-}=2\ga, &
   \qquad \Delta\la_{j}=\la_{j,+}-\la_{j,-}=2\sqrt{\ga^2-\la_j},\\
  & {\rm Re}\, \la_{j,-}\ge0 & {\rm Re}\, \la_{j,+}\ge \ga, \quad \, \,  \, \,\, & \qquad |\la_{j,\pm}|\le
  \ga+\sqrt{\ga^2+4}. \label{eq:relat2}
  \end{align}
Solving \eqref{eq:ode2}
we easily obtain: first, for $j=0,\dots,n$
such that $\la_{j,+}\not=\la_{j,-}$ (\textit{i.e.}~$\la_j\not=\ga^2$)
  \begin{equation}
    \label{zjD2}
     \begin{split}
      \tilde z_j (t) &= \tilde p_j(0)\; \frac{\lambda_{j,+} e^{-\lambda_{j,+}t}-\lambda_{j,-}e^{-\lambda_{j,-}t}}{\Delta \lambda_j} + \tilde r_j(0)\;\lambda_j^{1/2}\; \frac{e^{-\lambda_{j,-}t}-e^{-\lambda_{j,+}t}}{\Delta \lambda_j}
\\
  \tilde y_j(t)  &   =  \tilde p_j(0) \; \lambda_j^{1/2}\; \frac{e^{-\lambda_{j,+}t}-e^{-\lambda_{j,-}t}}{\Delta \lambda_j} + \tilde r_j(0) \; \frac{\lambda_{j,+} e^{-\lambda_{j,-}t}-\lambda_{j,-}e^{-\lambda_{j,+}t}}{\Delta \lambda_j}.
\end{split}
\end{equation}
Next, if $\gamma^2= \lambda_j$ for some $j$, then
$\lambda_{j,+} = \lambda_{j,-} = \gamma$. To avoid a complicated
notation, by convention, we shall interpret the above formulas as
$0/0$ symbols, remembering that $\lambda_{j,+} = \lambda_{j,-}+
\Delta\la_{j}$ and $\Delta\la_{j}\to0$. This leads to
\begin{equation}
  \label{eq:4}
  \begin{split}
  &  \tilde z_j (t)= \Big((1-\gamma t) \tilde p_j(0) +
  \ga t \, \tilde r_j(0) \Big) e^{-\gamma t},\\
  &
 \tilde y_j(t)     =  \Big( (1+\gamma t) \tilde r_j(0) -
    \ga  t \, \tilde p_j(0) \Big) e^{-\gamma t}.
  \end{split}
\end{equation}
In what follows we shall repeatedly make use of the following
estimate, which we call the \emph{key lemma}. It enables us to understand the behavior of the quotients appearing in \eqref{zjD2} when $n$ goes to infinity.
\begin{lemma}[Key lemma]
  \label{lm012911-23}
There exist constants $c,C>0$, depending only on $\ga$, such that, for any $n=1,2,\ldots$, 
\begin{equation}
\label{eq:boundla-}
 c\big(\tfrac{j}{n}\big)^2 \leqslant \mathrm{Re}\,\la_{j,-} \leqslant
 |\la_{j,-}|\leqslant C\big(\tfrac{j}{n}\big)^2,\quad j=0,\ldots,n.
\end{equation}
 Therefore, for any $q\in[0,+\infty)$, there exists constants $C,c>0$,
 depending only $q$ and $\ga$,  such that, for any $t>0$ and any
 $j=0,\dots,n$,
  \begin{equation}
    \label{032911-23q}
   |\la_{j,-}|^{q} \, \bigg| e^{-\la_{j,-}t}\, \frac{
     1-e^{-\Delta\la_{j}t}}{\Delta\la_{j}}\bigg|\le
  C\left(\tfrac{j}{n}\right)^{2q}e^{-ctj^2/n^2}.
 \end{equation}
 Moreover,  since from  \eqref{eq:relat2} we have $\lambda_j=|\la_{j,+}\la_{j,-}| \leqslant C(\ga)\la_{j,-}$,   an analogous estimate holds also when $|\la_{j,-}|^{q}$ is replaced
 by $\la_{j}^{q}$,
\end{lemma}
\proof Recall first the standard inequality $2x/\pi\leqslant \sin
x\leqslant x$ that holds for $x\in[0,\pi/2]$.
Note that 
\[|\la_{j,-}| = \frac{\la_j}{|\la_{j,+}|} \le \frac{\la_j}{{\rm Re}\,\la_{j,+}} \le \frac{4\sin^2(\frac{j \pi}{2(n+1)})}{\ga} \le C(\ga) \big(\tfrac{j}{n}\big)^2,\] where we used \eqref{eq:relat2} in the second inequality. Moreover,
\[ \mathrm{Re}\,\lambda_{j,-} = \mathrm{Re}\, \frac{\lambda_j}{\lambda_{j,+}} = \frac{4\sin^2(\frac{j \pi}{2(n+1)})}{|\la_{j,+}|^2}\,\mathrm{Re}\,\la_{j,+} \geqslant c(\gamma) \big(\tfrac{j}n\big)^2, \] where we used once again \eqref{eq:relat2}. This proves \eqref{eq:boundla-}.

Now, let us prove \eqref{032911-23q}.
Note that $\Delta\la_j$ is either real non-negative, or purely imaginary, therefore $\mathrm{Re}\, \Delta\la_j\geqslant 0$.  We distinguish two cases:
\begin{itemize} \item If ${\rm Re}\,\Delta\la_{j}\ge
\ga/2$, then, using the crude estimate $ |1 - e^{-\Delta\la_j t}| \leqslant 2$ and $\Delta \lambda_j \geqslant \ga/2$ we 
may bound the left hand side of \eqref{032911-23q} by 
\[ 
 \ga^{-1} | \lambda_{j,-}|^q \, e^{-\mathrm{Re}\,\la_{j,-}t} \] and from \eqref{eq:boundla-} we conclude that \eqref{032911-23q} holds.
\medskip

\item If $0\leqslant {\rm Re}\,\Delta\la_{j} <
\ga/2$, then in that case we get a better lower estimate than \eqref{eq:boundla-}, more precisely we can write ${\rm
  Re}\,\la_{j,-}=\ga - \frac12\mathrm{Re}\,\Delta\la_j \ge \ga/2$. 

\noindent We now use the estimate $|1-e^{-z}| \leqslant |z|$ valid for ${\rm Re}\, z \geqslant 0$, and we note that $|te^{-\lambda_{j,-}t}| = te^{-{\rm Re}\,\lambda_{j,-}t}$. We therefore obtain that the expression on the left hand side of
\eqref{032911-23q} can be estimated by
\begin{align}
   \label{032911-23ab}
   |\la_{j,-} |^q\, t e^{-\ga t/2} \le C_* |\la_{j,-} |^q\,
    e^{-\ga t/4}, \qquad \text{with }
   C_{*}:=\sup_{t\ge0}\big\{te^{-\ga t/4}\big\}.\end{align}
This together with \eqref{eq:boundla-} implies  \eqref{032911-23q} (since $j/n\le 1$ and taking $c=\ga/4$).
\end{itemize}
\qed

From now on, the constants $C>0$ which appear in the statements and proofs may depend on $\ga$, $\bar F$, $C_{\cF}$, $T_-$, namely the parameters of the model. Unless explicitly stated, they do not depend on $t>0$ nor on $n$.

The rest of the section is dedicated to the proof of Proposition
\ref{prop:main}. Therefore, we will always assume that Assumption
\ref{ass1} holds, which implies that the right hand side of
\eqref{011109-23} in  Lemma \ref{barr-p} is of order $\mathcal{O}(n)$.

\subsection{Control of the $\ell^2$-norm of momenta averages}
\label{sec3.2.1}
Here we prove \eqref{012409-23main}, for the momenta. 
Let us compute explicitly the momenta averages.

We use the first equation in \eqref{zjD2} to write the decomposition
\begin{equation}
	\label{zjD2y}
	\tilde z_j (t)={\rm I}_j(t)+{\rm II}_j(t) +{\rm III}_j(t), 
\end{equation} where ${\rm I}_j(t):= \tilde p_j(0) \, e^{-\la_{j,+}t}$ and \begin{equation*} 
		{\rm II}_j(t) := -\la_{j,-} \, \tilde p_j(0)\,  e^{-\la_{j,-}t} \frac{
			1-e^{-\Delta\la_{j}t}}{\Delta\la_{j}}, \qquad
		{\rm III}_j(t) :=  \la_j^{1/2}\;\tilde r_j(0) \,  e^{-\la_{j,-}t}\, \frac{
			1-e^{-\Delta\la_{j}t}}{\Delta\la_{j}}.
\end{equation*}
Then, using \eqref{011605-22r} and \eqref{zjD2y} (or \eqref{zjD2}) in order to express $e^{-At}$, we
obtain, after a direct calculation,
\begin{equation}
  \label{eq:9}
    \bar p_x(t) =  {\bar z}_x(t)
    +{\rm I}^{(p)}_x(t)+{\rm II}^{(p)}_{x,1}(t) + {\rm II}^{(p)}_{x,2}(t) .
\end{equation}
where $\bar z_x(t)$ can be explicitely written by taking the inverse Fourier transform of expression \eqref{zjD2}, and besides
\begin{equation*}
  \begin{split}
        &
   {\rm I}^{(p)}_x(t):= \bar F \sum_{j=0}^n
\psi_j(n)\psi_j(x)
  e^{-\la_{j,-}t}\,\frac{ 1-e^{-\Delta\la_{j} t}}{\Delta\la_{j}},\\
   &{\rm II}^{(p)}_{x,1}(t)
    := \frac 1{\sqrt n} \sum_{\ell\not=0} \hat{\cF}(\ell) 
    \sum_{j=0}^n \psi_j(n)\psi_j(x)\, \frac{ (e^{i  \ell\om t}-e^{-\la_{j,-}t}) i \ell\om}
    {\lambda_j - (\ell\om)^2 + 2  i  \ell\om\gamma},\\
    &
    {\rm II}^{(p)}_{x,2}(t) :=\frac 1{\sqrt n} \sum_{\ell\not=0} 
    \hat{\cF}(\ell) \sum_{j=0}^n \psi_j(n)\psi_j(x)
    \, \frac{\lambda_{j,+}e^{-\la_{j,-}t}(1-e^{- \Delta\la_{j}t})}{ \Delta\la_j ( i \ell\om+\la_{j,+}) }.
    \end{split}
  \end{equation*}
We recall the notation $\om= \frac {2\pi}{\theta}$.

Let us now start the proof of Proposition \ref{prop:main}, beginning with \eqref{012409-23main}.

  \proof[Proof of \eqref{012409-23main}]

Recall \eqref{eq:9}, and let us start with the first contribution, namely let
us estimate $\sum_{x=0}^n |\bar z_x(n^2s)|^2$. Note that we cannot use
the previous bound obtained in Lemma \ref{barr-p}, instead we need to
improve it.  Recall the decomposition of $\tilde z_j(t)$ given in
\eqref{zjD2y}. From the Parseval identity, we are going to bound in a more refined way the following three terms:
\begin{equation}\sum_{j=0}^n |{\rm I}_j(n^2s)|^2, \qquad \sum_{j=0}^n |{\rm II}_j(n^2s)|^2, \qquad \sum_{j=0}^n |{\rm III}_j(n^2s)|^2.\end{equation}
First of all, using the fact that $ {\rm Re}\,\la_{j,+}\ge \ga$, and Parseval identity, we can write
\begin{align}
  \label{021904-24}
	\sum_{j=0}^n |{\rm I}_j(n^2s)|^2 = \sum_{j=0}^n(\tilde p_j(0))^2  e^{-2{\rm Re}\,\la_{j,+}n^2s} \le \sum_{j=0}^n(\tilde p_j(0))^2  e^{-\ga n^2s}   = \sum_{x=0}^n \bar p_x^2(0) e^{-\ga n^2s} .
\end{align}
Therefore, from Assumption \ref{ass1}, we get
\begin{equation} \label{eq:p-one}	\sum_{j=0}^n |{\rm
    I}_j(n^2s)|^2 \leqslant \bar{\cal H} n, \quad s \ge
  0 \end{equation} and, for $t\ge 0$, 
\begin{align}  \int_0^t 	\sum_{j=0}^n |{\rm I}_j(n^2s)|^2 \; \dd s \le \int_0^t  \sum_{x=0}^n \bar p_x^2(0) e^{-\ga n^2s} \; \dd s  =  \frac{1}{\ga n^2} \sum_{x=0}^n \bar p_x^2(0) (1-e^{-\ga n^2 t}) \le \frac{\bar{\cal H}}{\ga n}.  \label{eq:p-two}\end{align}
{In addition
\begin{equation} \label{eq:p-one1}
  \sum_{j=0}^n |{\rm I}_j(n^2s)|^2\le \frac{\bar{\cal H}}{\ga }ne^{-\ga n^2s}.
\end{equation}}
Note that both estimates \eqref{eq:p-one} and \eqref{eq:p-two} are exactly of  the needed order to prove \eqref{012409-23main}.  We will now proceed in the same way for all the other contributions. We have
  \begin{align}
	  \sum_{j=0}^n |{\rm II}_j(n^2s)|^2 & = 4\sum_{j=0}^n (\tilde p_j(0))^2\, |\la_{j,-}|^2\, \Big|  e^{-\la_{j,-}n^2s} \frac{
                                                    1-e^{-\Delta\la_{j}n^2s}}{\Delta\la_{j}}\Big|^2 \label{011904-24}\\
                                                  & \le
                                                    C \sum_{j=0}^n
                                                    (\tilde
                                                    p_j(0))^2\left(\tfrac{j}{n}\right)^{4}e^{-cj^2s}
                                                    \quad \text{from
                                                    the key Lemma
                                                    \ref{lm012911-23},
                                                    \eqref{032911-23q}.}\notag 
\end{align}
Therefore, from Assumption \ref{ass1} and the fact that $(\frac j
n)^4\; e^{-cj^2 s} \le 1$, we obtain again that $\sum_{j=0}^n |{\rm
  II}_j(n^2s)|^2\le Cn$ {and
\begin{equation} \label{eq:p-one2}
  \sum_{j=0}^n |{\rm
  II}_j(n^2s)|^2\le  \frac{\bar{\cal H}}{\ga
  n^3}\sup_{j\ge1}j^4e^{-\ga j^2s}\le \frac{C}{ s^2  n^3}.
\end{equation}}
Integrating in time
\[
\int_0^t \sum_{j=0}^n |{\rm II}_j(n^2s)|^2 \; \dd s \le 
\frac{C}{n^2} \sum_{j=0}^n (\tilde p_j(0))^2  \big(\tfrac j n\big)^2(1-e^{-cj^2t}) \le \frac{C}{n}.
\]
Similarly,  invoking again Lemma
\ref{lm012911-23}, namely \eqref{032911-23q} with $\la_j$ instead of $\la_{j,-}$, we  write 
\begin{equation}
    \label{031904-24} \sum_{j=0}^n |{\rm III}_j(n^2s)|^2 \leqslant  C \sum_{j=1}^n
    (\tilde r_j(0))^2\left(\tfrac{j}{n}\right)^{2}e^{-cj^2s}
    \end{equation}
and from this point we estimate as before obtaining
\[ 
    \sum_{j=0}^n |{\rm III}_j(n^2s)|^2 \le Cn, \qquad {\sum_{j=0}^n
      |{\rm III}_j(n^2s)|^2  \le \frac{C}{ s  n}} \qquad\text{and}\qquad \int_0^t   \sum_{j=0}^n |{\rm III}_j(n^2s)|^2 \; \dd s \le
    \frac{C}{n}.
  \]
The remaining contributions coming from the right hand side of
\eqref{eq:9}  are treated in a  similar way: first, using once again the Plancherel identity and the key Lemma \ref{lm012911-23} (with $q=0$), we get
\begin{align}
  \label{041904-24}
     \sum_{x=0}^n | {\rm
     I}^{(p)}_x(n^2s)|^2 &=\bar F^2   \sum_{j=0}^n
   \psi_j^2(n)\left|
  e^{-\la_{j,-}n^2s}\,\frac{ 1-e^{-\Delta\la_{j}n^2 s}}{\Delta\la_{j}} \right|^2   \le  C\bar F^2 \sum_{j=0}^n
       \psi_j^2(n) e^{-cj^2 s} \notag\\ & \le\frac{C\bar F^2 }{n} \sum_{j=0}^n e^{-cj^2s} 
     \end{align}
where we have used  the fact that $\psi_j^2(n) \le C/n$. The bound $ \sum_{x=0}^n | {\rm
	I}^{(p)}_x(n^2s)|^2 \le Cn$ comes easily. Integrating in time, we obtain
\[ 
\int_0^t  \sum_{x=0}^n | {\rm
	I}^{(p)}_x(n^2s)|^2 \; \dd s \le \frac{C}{n}\bigg\{ t + \sum_{j=1}^n \frac{1}{cj^2} (1-e^{-cj^2t})\bigg\} \le \frac{C(t+1)	}{n}.
\]
  Then, we have   
     \begin{align}
        \sum_{x=0}^n |{\rm
       II}^{(p)}_{x,1}(n^2s)|^2
       =\frac 1{ n}\sum_{j=0}^n \psi_j^2(n) 
       \left| \sum_{\ell\not=0}     \frac{\hat{\cF}(\ell) (e^{i
       \ell\om n^2s}-e^{-\la_{j,-}n^2s})i \omega(\ell)}
     {\lambda_j - (\ell\omega)^2 + 2  i  \omega(\ell)
       \gamma}\right|^2. \label{eq:IIp1}
     \end{align}
     Note that $|\la_j-(\ell\om)^2+2i\omega(\ell)\ga |^2 \ge 4
     \ga^2 (\ell\om)^2$, and $| e^{i  \ell\om n^2s}-e^{-\la_{j,-}n^2s}| \le 2$, and recall $\psi_j^2(n)\le C/n$ from \eqref{eq:psibnd}. Therefore  we may write
     \begin{equation}
       \label{051904-24}
        \sum_{x=0}^n | {\rm
       II}^{(p)}_{x,1}(n^2s)|^2 \le \frac{C}{n}  \Big( \sum_{\ell\not=0}     |\hat{\cF}(\ell)|\Big)^2  = \frac{CC_{\cF}^2}{n} 
     \end{equation} where $C_{\cF}$ was defined in \eqref{031409-23}, 
 and the bounds in this case become trivial.
     Finally, we use once again \eqref{032911-23q} with $q=0$, together with the fact that $|\la_{j,+}|\le C$ and $|i\omega(\ell)+\la_{j,+}| \geqslant \mathrm{Re}\,\la_{j,+}\ge \ga$, and we obtain
     \begin{align}
        \sum_{x=0}^n | {\rm
       II}^{(p)}_{x,2}(n^2s)|^2&=\frac 1{ n }\sum_{j=0}^n
       {\psi_j^2(n)}
      \left|\sum_{\ell\not=0} 
    \hat{\cF}(\ell)   \frac{\lambda_{j,+}\,e^{-\la_{j,-}n^2s}(1-e^{- \Delta\la_{j}n^2s})}{  ( i \ell\om+\la_{j,+}) \Delta\la_j}
    \right|^2  \notag   \label{eq:IIp2}\\
    &    \le  \frac{C}{n}\Big( \sum_{\ell\not=0}     |\hat{\cF}(\ell)|\Big)^2 \sum_{j=0}^n\psi_j^2(n)
      e^{-cj^2s}   \le \frac{CC_{\cF}^2}{n} \sum_{j=0}^n\psi_j^2(n)
           e^{-cj^2s} 
  \end{align} 
and it is treated similarly to \eqref{041904-24}. 
 Hence, summing all the contributions, the proof of
 \eqref{012409-23main} {and \eqref{012409-23main2}}  easily follows.
 \qed

\subsection{Control of the $\ell^2$-norm of the stretches averages}
\label{sec3.3}
Now, let us do the same procedure for the stretch averages, and prove  \eqref{012409-23main1}. 
 From \eqref{011605-22r} and \eqref{zjD2}, we get
\begin{equation}
  \label{eq:9a}
    \bar r_x(t) =  \bar y_x(t)
    +\sum_{m=1}^2\Big( {\rm I}^{(r)}_{x,m}(t)+{\rm II}^{(r)}_{x,m}(t) \Big).
  \end{equation}
We can write $\bar y_x(t)$   explicitly   taking the inverse Fourier
transform of the second identity in \eqref{zjD2}. Furthermore,
  \begin{align*}
  & {\rm I}^{(r)}_{x,1}(t):=  \bar F \sum_{j=1}^n
   \psi_j(n)\phi_j(x)\frac{e^{-\la_{j,-}t}-1}{\la_j^{1/2}},\\
    &
    {\rm I}^{(r)}_{x,2}(t):= \bar F \sum_{j=1}^n \frac{\psi_j(n)\phi_j(x)}{\la_{j,+}^{1/2}}\,
  \la_{j,-}^{1/2}\, e^{-\la_{j,-}t}\frac{
      1-e^{-\Delta\la_{j}t} }{\Delta\la_{j}}
  \end{align*}
  and
    \begin{equation*}
  \begin{split}
  {\rm II}^{(r)}_{x,1}(t)
   & := \frac 1{\sqrt n} \sum_{\ell\not=0}
    \hat{\cF}(\ell)  \sum_{j=1}^n  \psi_j(n)\phi_j(x)
     \frac{\la_j^{1/2} (e^{-\la_{j,-}t}-e^{i\ell\om t}) }
    {\lambda_j - (\ell\om)^2 + 2  i  \ell\om\gamma},\\    
    {\rm II}^{(r)}_{x,2}(t)& :
        =\frac 1{\sqrt n} \sum_{\ell\not=0}
    \hat{\cF}(\ell) \sum_{j=1}^n 
      \psi_j(n)\phi_j(x) \frac{\la_{j}^{1/2}\, e^{-\la_{j,-}t}(1-e^{-\Delta\la_{j}t})}{(i \ell\om+\la_{j,+})\Delta\la_j}.
  \end{split}
\end{equation*}
We claim the following: there exists a constant $C>0$ such that, for
any $n=1,2,\ldots $ and  $t\geqslant 0$,
  \begin{equation}
  \label{eq:11r}
  \begin{split}
    \sum_{x=1}^n \bar r_x^2( t)  \le  C
  \bigg(\sum_{x=0}^n\big(\bar p_x^2(0)+\bar r_x^2(0)\big)+n \oF^2 \bigg).
  \end{split}
\end{equation}
Then, from Assumption \ref{ass1}, this concludes the proof of \eqref{012409-23main1}.

\medskip 

It remains to prove \eqref{eq:11r}, using \eqref{eq:9a}.
The first contribution coming from $\bar y_x^2(t)$ can be directly estimated using \eqref{011109-23}, and this gives the first term in the right hand side of \eqref{eq:11r}.
Then, by the Plancherel identity we have
  \begin{align*}
\sum_{x=1}^n| {\rm I}^{(r)}_{x,1}(t)|^2=\bar F^2 \sum_{j=1}^n
                   \frac{\psi_j^2(n)}{\la_j}\big|e^{-\la_{j,-}t}-1\big|^2
      \le \frac{\bar F^2}{2(n+1)}
      \sum_{j=1}^n\cot^2\left(\tfrac{\pi j}{2(n+1)}\right)  \le C \bar F^2  n,
  \end{align*} where we used the expressions of $\lambda_j$ given in \eqref{eq:70} and $\psi_j(n)$ given in \eqref{eq:vecneu}, together with the crude bound $|e^{-\lambda_{j,-}t}-1|\le 2$.
 Likewise, using the Plancherel identity and then  
 the key Lemma, estimate \eqref{032911-23q}, together with $|\lambda_{j,+}| \ge \mathrm{Re}\,\la_{j,+} \ge \ga$, we get
  \begin{align*}     \sum_{x=1}^n|{\rm I}^{(r)}_{x,2}(t) |^2= \bar F^2 \sum_{j=1}^n
  \frac{\psi_j^2(n)}{|\la_{j,+}|} 
                   \bigg| \la_{j,-}^{1/2}\,e^{-\la_{j,-}t}\,\frac{
      1-e^{-\Delta\la_{j}t} }{\Delta\la_{j}}\bigg|^2
                   \le  C\bar F^2 \sum_{j=1}^n
  \psi_j^2(n)
                  \left(\tfrac{j}{n}\right)^2e^{-cj^2t}\le C\oF^2,
  \end{align*} where in the last estimate we used the fact that $\psi_j^2(x) \leqslant C/n$ from \eqref{eq:psibnd} and we have simply bounded $\frac jn$ and $e^{-cj^2t}$ by $1$. Therefore this contribution is smaller than the previous one.
The last two estimates are very similar to the ones conducted for  ${\rm II}^{(p)}_{x,1}(t)$ and ${\rm II}^{(p)}_{x,2}(t)$ (see \eqref{eq:IIp1} and \eqref{eq:IIp2} in particular) in the previous section. More precisely,  we easily get (using the  bounds $|\lambda_{j}-(\ell\om)^2+2i\ell\om \ga|^2 \ge 4\ga^2 \omega^2(\ell)$ for any $\ell \neq 0$ and $\lambda_j \le 4$):
  \begin{align}
    \sum_{x=1}^n| {\rm II}^{(r)}_{x,1}(t)|^2&=\frac 1{n} \sum_{j=1}^n \la_j \psi_j^2(n)  \left|
    \sum_{\ell\not=0}
    \frac{\hat{\cF}(\ell)(e^{-\la_{j,-}t}-e^{i\ell\om t})  }
    {\lambda_j - (\ell\om)^2 + 2  i  \ell\om 
    \gamma}\right|^2 \notag \\
    &\le \frac C{n} \sum_{j=1}^n \psi_j^2(n) \Big(\sum_{\ell\not=0}
    | \hat{\cF}(\ell)  |\Big)^2 \le \frac{C {C_{\cF}^2}}{n} \label{eq:ex}
  \end{align}and from the key lemma, namely \eqref{032911-23q} but bounded crudely by $C$, plus \eqref{eq:relat2}, we obtain similarly
 \begin{align*}
   \sum_{x=1}^n |{\rm II}^{(r)}_{x,2}(t)|^2 & = \frac{1}{n}\sum_{j=1}^n \psi_j^2(n) \bigg| \sum_{\ell\neq 0} \hat\cF(\ell)\frac{\lambda_{j,+}^{1/2}}{i\ell\om +\lambda_{j,+}} \frac{\lambda_{j,-}^{1/2}\,e^{-\la_{j,-}t}(1-e^{-\Delta\la_{j}t})}{\Delta\la_j}\bigg|^2 \le \frac C n.
 \end{align*} Those last two estimates are smaller than the previous ones, in particular smaller than $C\bar F^2 n$, therefore we conclude the proof of \eqref{eq:11r}.
\qed

\subsection{Estimates at the boundaries}
\label{sec:bound}

Finally, we investigate the behavior at both boundaries, and we prove in this section the second point \emph{(ii)} of Proposition \ref{prop:main}, namely \eqref{eq:18main} and \eqref{eq:22zmain}. 

We first prove \eqref{eq:18main} for $x=n$, the proof for $x=0$ is analogous.
Using formula \eqref{eq:9} we can write
\begin{equation}
  \label{eq:9n}
  \begin{split}
    \bar p_n(t) = P_0(t) + P_{\bar F}(t) + P_{fl}(t) + P_{dp}(t),
  \end{split}
\end{equation}
where
\begin{align}
  \label{eq:P0}
    P_0(t) &= \sum_{j=0}^n \psi_j(n) \tilde z_j(t) = \bar z_n(t), \\
    P_{\bar F}(t) &=   \bar F \sum_{j=0}^n \psi_j^2(n) e^{-\la_{j,-}t}\frac{1- e^{-\Delta\la_{j}t} }{\Delta\la_{j}}, \label{eq:PF}
    \\
    P_{fl}(t) &= \frac 1{\sqrt n} \sum_{\ell\not=0} \hat{\cF}(\ell) i \ell\om 
    \sum_{j=0}^n \psi_j^2(n) \frac{ (e^{ i  \ell\om t}-e^{-\la_{j,-}t})  }
    {\lambda_j - (\ell\om)^2 + 2 i  \ell\om\gamma}, \label{eq:Pfl}\\
     P_{dp}(t) &= \frac 1{\sqrt n} \sum_{\ell\not=0}
    \hat{\cF}(\ell) \sum_{j=0}^n \psi^2_j(n)
   \frac{\lambda_{j,+}e^{-\la_{j,-}t}(1-e^{-\Delta\lambda_j t}) }{(i \ell\om+\la_{j,+})\Delta\la_{j}}. \label{eq:Pdp}
  \end{align}
It is quite clear that the contributions coming from $P_0$ and $P_{\oF}$ are the biggest ones: in fact, thanks to the extra term $1/\sqrt n$ in front of the sums in \eqref{eq:Pfl} and \eqref{eq:Pdp}, and using the fact that $\sum |\hat{\cF}(\ell)|<\infty$, we will see that the last two terms always have a smaller order.
Let us start with the contribution coming from $P_0$, and  use the
decomposition of $\tilde z_j(t)$ given in \eqref{zjD2y}. We can bound
by the triangle inequality,
\begin{align*}
  \bigg| \int_{0}^{t}  P_0(n^2s) \dd s\bigg|\le & \; \sum_{j=0}^n |\psi_j(n)||\tilde
  p_j(0)| \int_{0}^{t}  \bigg\{ |e^{-\la_{j,+}n^2 s}|+  |\la_{j,-}|\Big| e^{-\la_{j,-}n^2
  s}\frac{e^{-\Delta\la_{j}n^2 s}-1}{\Delta\la_{j}}\Big|\bigg\}\dd s\\
  & 
     + \sum_{j=1}^n |\psi_j(n)||\tilde
  r_j(0)| \; \int_{0}^{t}   \la_j^{1/2}\Big| e^{-\la_{j,-}n^2
  s}\frac{e^{-\Delta\la_{j}n^2 s}-1}{\Delta\la_{j}}\Big|\dd s .\notag
  \end{align*}
 We use the key Lemma \ref{lm012911-23} in both terms containing quotients, and then we integrate over $s$. For the first term in the first integral we use the fact that ${\rm Re}\,\la_{j,+}\ge \ga$ and then integrate over $s$. Finally,  the right hand side can be estimated by
\begin{multline}
  C \sum_{j=0}^n |\psi_j(n)||\tilde
  p_j(0)|   \Big(
   \frac{1-e^{-\gamma n^2 t} }{n^2}+\Big(\frac{j}{n}\Big)^2\cdot\frac{1-e^{-cj^2t}
   }{j^2}\Big) 
     + C \sum_{j=1}^n |\psi_j(n)||\tilde
  r_j(0)| \; \frac{j}{n} \cdot \frac{1-e^{-cj^2t}
   }{j^2} \\
 \leqslant \frac{C}{n^2} \bigg(\sum_{j=0}^n \psi_j^2(n)\bigg)^{1/2} \bigg(\sum_{j=0}^n (\tilde p_j(0))^2\bigg)^{1/2} + \frac{C}{n} \bigg(\sum_{j=1}^n \frac{\psi_j^2(n)}{j^2}\bigg)^{1/2} \bigg(\sum_{j=1}^n(\tilde r_j(0))^2\bigg)^{1/2} \label{eq:rhs}
\end{multline}
thanks to the Cauchy-Schwarz inequality. Using the initial bound \eqref{eq:delta}, and also $\psi_j^2(x)\leqslant C/n$, we can
estimate the whole right hand side of \eqref{eq:rhs} by $C/n$ (with $C$ independent of $t$).

The other estimates are quite similar. 
In order to bound
$| \int_{0}^{t}  P_{\bar F}(n^2s) \dd s|$, we
use the key Lemma \ref{lm012911-23} and integrate over $s$ (and also
recall that $\la_{0,-}=0$), then we get it is
less than, or equal to
\begin{align*}
 \frac{|\bar F|t}{2\ga n} + |\bar F| \sum_{j=1}^n \psi_j^2(n)
 \frac{1-e^{-cj^2t}}{j^2} \le \frac{C(t_*)}{n} \qquad \text{for any $t \in [0,t_*]$}.
\end{align*}
Finally, the last two estimates are even smaller, thanks to the fact that $\psi_j^2(n)\le C/n$: for instance, integrating over
$s$ we can write
\begin{align*}
\bigg| \int_{0}^{t}  P_{fl}(n^2s) \dd s\bigg|\le \frac 1{\sqrt n} \sum_{\ell\not=0}  |\hat{\cF}(\ell)|  \; \bigg|\frac{e^{ i
   \ell\om n^2t} - 1}{n^2  \ell \om } \bigg|\le \frac{C}{n^{3/2}}
\end{align*}
and similarly for  
$| \int_{0}^{t}  P_{dp}(n^2s) \dd s|$. Estimate \eqref{eq:18main} then follows.

\medskip

Now let us show \eqref{eq:22zmain}. We 
integrate both sides of  \eqref{eq:pbdf1b}  for $x=n$  and obtain
 \begin{equation}
   \label{eq:19}
   \frac 1{n^2} \left(\bar p_n(n^2 t) - \bar p_n(0)\right) =
   \frac{1}{n^2}\int_0^{n^2t} \cF_n(s) \dd s -\frac{1}{n^2}\int_0^{n^2t}
   \left(\bar r_n( s)+2\gamma\bar p_n(s)\right) \dd s .
 \end{equation}
 {Recall the decomposition of $\cF_n(t)$ from \eqref{Fnt}. A straightforward computation gives
 \begin{align}
  \bigg| \frac{1}{n^2}\int_0^{n^2t} \cF_n(s) \dd s - \bar F t \bigg| & = \bigg| \int_0^t \frac{1}{\sqrt{n}} \sum_{\ell \neq 0} |\hat{\cF}(\ell)| e^{i\ell n^2 \omega r} \dd r \bigg| \le \frac{C_\cF t}{\sqrt{n}}  \xrightarrow[n\to\infty]{} 0. \label{eq:20}
 \end{align}}
 {Besides,} note that, from  \eqref{012409-23main}
 \[
 |\bar p_n(n^2t)| \le \bigg(\sum_{x=0}^n \bar p_x^2(n^2t)\bigg)^{1/2} \le C\sqrt{n}.
 \]
 Therefore, formula \eqref{eq:22zmain} follows directly from  
   \eqref{eq:18main}, \eqref{eq:19}, \eqref{eq:20}.\qed

 \section{Stretch and mechanical energy: Proof of Theorem \ref{th-r}}
\label{sec:evol-stretch}
\subsection{Macroscopic evolution of the stretch}
\label{sec:proof-equat-volume}

Now we have all the ingredients at hand to prove the first convergence result, namely \eqref{eq:conv-stretch-11}.

Let $\mathcal M_{K}[0,1]$ be the space of signed measures $m$ on
$[0,1]$, whose total variation
$|m|[0,1]$ is bounded by a constant $K>0$, endowed with the topology of weak
convergence. The space 
is metrizable 
and    compact. Given $t_*>0$, we consider the space
${\mathcal C}\left([0,t_*], {\mathcal M}_{K}[0,1]\right)$ endowed with the corresponding
uniform topology.
Define    
  \begin{equation}
    \label{eq:23}
    \xi^r_n(t,\varphi) = \frac 1n \sum_{x=1}^n
    \varphi\left(\frac{x}{n}\right) \bar r_x(n^2 t).
  \end{equation}
  By Assumption \ref{ass1} there exists a finite $K>0$
  such that $\xi^r_n\in \mathcal C\left([0,t_*], \mathcal
    M_{K}[0,1]\right)$ --  the space of  all continuous functions from
  $[0,t_*]$ into $ \mathcal
    M_{K}[0,1]$.
Since   the latter is compact under the weak topology,
  the compactness of the sequence $\{\xi^r_n\}$ in
  $\mathcal C\left([0,t_*], \mathcal M_{K}[0,1]\right)$
  follows from a bound on the modulus of continuity in time, by an
  extention of  the Ascoli-Arzel\`a
 Theorem, see \emph{e.g.}~\cite[p. 234]{kel}. This in particular implies
  compactness of the sequence of the empirical measures corresponding
  to $\{\bar r_x(n^2 t) \; ; \;t\in[0,t_*], x=0,\ldots,n\}$. In what follows we identify the limit.
  
  \medskip
  
{From an   argument quite similar to the one of \cite[Appendix
	E]{klo22-2} one can conclude that the weak solution to \eqref{eq:HLstretch} is unique, therefore we now prove that the limit measure satisfies the weak formulation given in Definition \ref{df012701-23}.}

 Let $\mathcal C[0,1]$, resp.  $\mathcal C^k[0,1]$ for a positive integer
 $k$ (or $k=\infty$),  
 be the space of  all continuous, resp.  $k$ (or infinitely many) times continuously  differentiable,  functions on $[0,1]$.
{Now, we assume that the test function}
  $\varphi \in  \mathcal C^2[0,1]$ {is} such that $\varphi(0) = 0 = \varphi(1)$.
  Denote in the following $\varphi_x :=
  \varphi\left(\frac{x}{n}\right)$. Note that
  $\varphi_0 = 0=\varphi_n$.  Then, by the integration by parts,
  \begin{equation}
    \label{eq:24}
    \begin{split}
      \xi^r_n(t,\varphi) - \xi^r_n(0,\varphi)& =
      n \int_0^t \sum_{x=1}^n \varphi_x \left(\bar p_x(n^2 s) - \bar
        p_{x-1}(n^2 s)\right) \dd s\\
     & =  - n \int_0^t \sum_{x=0}^{n-1} \left(\varphi_{x+1} -
        \varphi_{x}\right) \bar p_{x}(n^2 s) \dd s+ n\varphi_{n} \int_0^t 
         \bar p_{n}(n^2 s)\dd s.
    \end{split}
  \end{equation}
Since $\varphi_n=0$, the last term equals $0$.
Recall the evolution equations \eqref{eq:pbdf1b}. Then,  for $x=0,\ldots,n-1$ we can write
  \begin{equation}
    \label{eq:25}
    \int_0^t \bar p_{x}(n^2 s)\dd s = \frac 1{2\gamma} \int_0^t \left(\bar r_{x+1}(n^2s) - \bar r_x(n^2 s)\right) \dd s
    - \frac 1{2\gamma n^2} \left(\bar p_x(n^2 t) - \bar p_x(0)\right).
  \end{equation} 
  Substituting into \eqref{eq:24} we obtain 
  \begin{equation}
    \label{eq:26}
    \begin{split}
      \xi^r_n(t,\varphi) - \xi^r_n(0,\varphi) =
     & - \frac n{2\gamma} \int_0^t \sum_{x=0}^{n-1}
      \nabla \varphi_{x} \left(\bar r_{x+1}(n^2s) - \bar r_x(n^2 s)\right) \; \dd s\\
     & + \frac 1{2\gamma n}  \sum_{x=0}^{n-1} \nabla \varphi_{x}  \left(\bar p_x(n^2 t) -\bar p_x(0)\right)+o_n(t),
    \end{split}
  \end{equation}
 where the expression $o_n(t)$ means that  $\sup_{t\in[0,t_*]}|o_n(t)|\to
 0$, as $n\to+\infty$ for any $t_*>0$.  Since
  \begin{equation}
    \label{eq:27}
    \sum_{x=0}^{n-1}  \nabla \varphi_{x} \left(\bar r_{x+1}- \bar r_x\right)
    = - \sum_{x=1}^{n-1}  \Delta\varphi_{x}   \bar r_x
    + \nabla \varphi_{n-1} \bar r_n,
  \end{equation}
  we have
  \begin{equation}
    \label{eq:28}
    \begin{split}
        \xi^r_n(t,\varphi) - \xi^r_n(0,\varphi) = &\;
         \frac 1{2\gamma n} \int_0^t \sum_{x=1}^{n-1} n^2\Delta\varphi_x \bar r_x(n^2 s) \; \dd s
        - \frac {n\nabla \varphi_{n-1}}{2\gamma} \int_0^t \bar r_n(n^2 s) \;\dd s \\
       & + \frac 1{2\gamma n^2}  \sum_{x=0}^{n-1} n\nabla \varphi_{x}
        \left(\bar p_x(n^2 t) - \bar p_x(0)\right)+o_n(t).
    \end{split}
  \end{equation}
  By approximating $n\nabla \varphi_{x} \sim \varphi'\left(\frac{x}{n}\right)$ and
  $n^2 \Delta\varphi_x \sim \varphi''\left(\frac{x}{n}\right)$, with all errors controlled unifomly by
  \eqref{012409-23main} and \eqref{012409-23main1} from Proposition \ref{prop:main}. Using the limit obtained in \eqref{eq:22zmain} for the boundary term we conclude that
  \begin{equation}
    \label{eq:29}
    \lim_{n\to\infty} \left[\xi^r_n(t,\varphi) - \xi^r_n(0,\varphi) -
    \frac 1{2\gamma} \int_0^t \xi^r_n(s,\varphi'') \dd s \right]
    +\frac 1{2\gamma} \varphi'(1) \bar F t = 0.
  \end{equation}
  This corresponds exactly to the weak formulation of \eqref{eq:HLstretch}.

  The bound of the modulus of continuity in time follows by a similar argument.

  \subsection{Macroscopic evolution of the mechanical energy}
\label{sec:macr-evol-mech}

Now we prove the convergence of the mechanical energy density stated in \eqref{eq:32}. 
Let ${\varphi}\in \mathcal C([0,1])$ be continuous. 
By Proposition \ref{prop:main}, {in particular \eqref{012409-23main2},} we already know that for any $t>0$
\begin{equation}
  \label{eq:34}
   \frac 1{n} \sum_{x=0}^{n}
   {\varphi}\left(\frac{x}{n}\right) \bar p_x^2(n^2t)
  \mathop{\longrightarrow}_{n\to\infty} 0,
\end{equation}
so we have only to prove that
\begin{equation}
  \label{eq:35}
  \frac 1{n} \sum_{x=1}^{n} {\varphi}\left(\frac{x}{n}\right) \bar r_x^2(n^2t)
  \mathop{\longrightarrow}_{n\to\infty} \int_0^1 {\varphi}(u) r^2 (t,u) \dd u.
\end{equation}
We start with the following:

\begin{lemma}
   \label{lm022009-23}
  There exists a constant $C$ such that, for any $n= 1,2,\ldots$ and  $t \ge 0$, 
  \begin{equation}
    \label{eq:36}
    n\sum_{x=0}^{n-1} \int_0^t \left(\nabla\bar r_{x}\right)^2(n^2 s) 
    \dd s \le C.
  \end{equation}
\end{lemma}
\begin{proof}
  From \eqref{eq:pbdf1b} we obtain
  \begin{equation}
    \begin{split}
     \frac 12 \frac \dd{\dd t} \bar p_x^2(n^2 t) = n^2 
     \nabla\bar r_{x}(n^2 t) \bar p_x(n^2 t) 
    - 2\gamma n^2 \bar p_x^2(n^2 t)
      \label{eq:37}
\end{split}
\end{equation}
for $x=0,\ldots,n-1$.  In addition we also have  
\begin{equation}
  \label{eq:38}
  \bar p_x(n^2 t) = \frac 1{2\gamma} \nabla \bar r_{x}(n^2 t) 
  -\frac 1{2\gamma n^2} \frac \dd{\dd t}\bar p_x(n^2 t).
\end{equation}
Substituting from \eqref{eq:38} for  $\bar p_x(n^2
t) $ into the first expression on the right hand side of \eqref{eq:37}
we obtain: for $x=0,\ldots,n-1$
\begin{equation}
  \label{eq:40}
  \begin{split}
  &  \frac{n^2}{2\gamma}
    \left(\nabla\bar r_{x}(n^2 t) \right)^2 -
    \frac 1{2\gamma} \nabla\bar r_{x}(n^2 t) \frac \dd{\dd t} \bar
    p_x(n^2 t) 
    - 2\gamma n^2 \bar p_x^2(n^2 t)
    = \frac 12 \frac \dd{\dd t} \bar p_x^2(n^2 t).
  \end{split}
\end{equation}
Summing up  over $x$  and integrating in time we obtain
\begin{equation}\begin{split}
  \label{eq:39}
    n\sum_{x=0}^{n-1} \int_0^t \left(\nabla r_{x}(n^2 s)\right)^2 \dd s
    = &\;\frac 1n \sum_{x=0}^{n-1} \int_0^t \nabla\bar r_{x}(n^2 s) \frac
    \dd{\dd s} \bar p_x(n^2 s)\; \dd s\\
   & + 4\gamma^2 n \sum_{x=0}^{n-1} \int_0^t  \bar p_x^2(n^2 s) \dd s
    + \frac{\gamma}{n}  \sum_{x=0}^{n-1} \left(\bar p_x^2(n^2 t)- \bar
      p_x^2(0)\right).
\end{split}
  \end{equation}
By Proposition \ref{prop:main}, see \eqref{012409-23main}, the second and third terms  on the right hand side of \eqref{eq:39} are bounded.
The only term in question is the first one.

 After the integration by parts in time and using \eqref{eq:pbdf1b}  we
conclude that it equals
\begin{equation}
  \label{eq:33}
  \begin{split}
    -  n \sum_{x=0}^{n-1} \int_0^t \bar p_x(n^2 s)\;
    \nabla\nabla^*\bar p_{x}(n^2 s) \; \dd s
    + \frac 1n \sum_{x=0}^{n-1} \left(\nabla\bar r_{x}(n^2 t) \bar p_x(n^2 t)
      - \nabla\bar r_{x}(0) \bar p_x(0)\right).
  \end{split}
\end{equation}
Using first the Cauchy-Schwarz inequality and then
\eqref{012409-23main}--\eqref{012409-23main1} from Proposition \ref{prop:main}, both
terms in \eqref{eq:33} are bounded uniformly in $t$ and $n$. Hence
\eqref{eq:36} follows.
\end{proof}

 Define    $\bar r^{(n)}_{\rm int}:[0,+\infty)\times[0,1]\to\R$ as the function obtained by
 the piecewise linear interpolation between the nodal points \[ \big(\tfrac{x}{n}, \bar r_x(n^2 t)\big),
 \qquad x=0,\ldots,n, \] of the piecewise constant  function
 $\bar r^{(n)} :[0,+\infty)\times[0,1]\to\R$, given by  \[\bar
 r^{(n)}(t,u)= \begin{cases} \bar r_x(n^2 t), & \text{ for any } u\in[\frac{x-1}{n},\frac{x}{n}), \; x=1,\ldots,n-1 \\ \bar r_n(n^2 t) & \text{ for any } u\in[\frac{n-1}{n},1].\end{cases}\]
By convention we let  $\bar r_0(n^2
t)=0$.
Let $H^1[0,1] $ be the  
completion of ${\cal C}^\infty_c(0,1)$ -- the space of smooth and
compactly supported functions -- {for} the norm 
\[
\|\varphi\|_{H^1[0,1]}^2:=\|\varphi\|_{L^2[0,1]}^2+ \|\varphi '\|_{L^2[0,1]}^2, \qquad \varphi\in {\cal C}^\infty_c(0,1).\]
As a consequence of Lemma \ref{lm022009-23} above we obtain
 the following.
 \begin{corollary}
\label{lm020401-20}
For any $t \ge 0$ we have
\begin{equation}
\label{030401-20}
\sup_{n\ge1}\int_0^{t}\big\| \bar r^{(n)}_{\rm
  int}(s,\cdot)\big\|_{H^1[0,1]}^2\dd s <+\infty .
\end{equation}
Moreover,
\begin{equation}
\label{090401-20z}
\lim_{n\to+\infty}\sup_{u\in[0,1]}\left|\int_0^t \bar r^{(n)}_{\rm int}(s,u)\dd s-\int_0^t  r(s,u)\dd s\right|=0
\end{equation}
and there exists a constant $C>0$ such that, for any $n\ge 1$, $u \in [0,1]$, $t\ge 0$,
  \begin{equation}
    \label{012909-23}
    \int_0^t\big(\bar r^{(n)}_{\rm int}(s,u)\big)^2\dd s\le {C u}.
    \end{equation}
  \end{corollary}

 \begin{proof} 
 It is easy
 to see that
\begin{equation}
\label{tbrn}
\big\|\bar r^{(n)}_{\rm int}(t,\cdot)-\bar r^{(n)}(t,\cdot)\big\|_{L^2[0,1]}^2
=\frac{1}{3(n+1)}\sum_{x=0}^{n-1}\big(\bar r_{x+1}(t)-\bar r_x(t)\big)^2,
 \qquad n\ge1.
\end{equation}
Estimate \eqref{030401-20} is a direct consequence of \eqref{eq:36}. 
Using  \eqref{tbrn} we also get
\begin{equation}
\label{020401-20xx}
\lim_{n\to+\infty}\int_0^t\big\|\bar r^{(n)}_{\rm int}(s,\cdot)-\bar r^{(n)}(s,\cdot)\big\|_{L^2[0,1]}^2 \;\dd s=0,\qquad t>0.
\end{equation}
From the convergence \eqref{eq:conv-stretch-11} proved in Section \ref{sec:proof-equat-volume} above we know
that the sequence
$\int_0^t \bar r^{(n)}_{\rm int}(s,u)\dd s$ weakly converges in
$L^2[0,1]$ to $\int_0^t  r(s,u)\dd s$ {for each $t>0$}. From \eqref{030401-20}
and the compactness of Sobolev embedding {of $H^1[0,1]$} into $C[0,1]$ in dimension
$1$ we conclude \eqref{090401-20z}.

  Besides, by the Cauchy-Schwarz inequality we get
  \begin{align*}
    (\bar r_x)^2(n^2 s)=\left(\sum_{y=0}^{x-1} \nabla\bar r_y(n^2
    s)\right)^2\le x  \sum_{y=0}^{x-1} \left(\nabla\bar r_y(n^2
    s) \right)^2
  \end{align*}
  and \eqref{012909-23} follows from \eqref{eq:36}.
 \end{proof}

 We are now ready to prove \eqref{eq:35}.
 In light of \eqref{tbrn} it suffices only to show that
 \begin{equation}
  \label{eq:35a}
  \int_0^1 {\varphi}(u) \big(\bar r^{(n)}_{\rm int}(t,u)\big)^2 \dd u
  \mathop{\longrightarrow}_{n\to\infty} \int_0^1 {\varphi}(u) r^2 (t,u) \dd u
\end{equation}
for any ${\varphi}\in \mathcal C([0,1])$.
 For $\delta >0$ let us define
 \begin{equation*}
   r^{(n,\delta)}_{\rm int} (t,u) = \delta^{-1} \int_0^\delta \bar r^{(n)}_{\rm int}(t+s,u) \dd s.
 \end{equation*}
 We have that $r^{(n,\delta)}_{\rm int} (t,\cdot)$
 converges weakly to $r^{(\delta)}(t,\cdot) = \delta^{-1}
 \int_0^\delta r(t+s, \cdot) \dd s$, when $n\to+\infty$, for any
 $t \ge 0$ and $\delta>0$.

In fact, using \eqref{eq:36} and the Cauchy-Schwarz inequality, together with
  \eqref{030401-20},  we have that for any $t\ge 0$, $n=1,2,\ldots$
 \begin{equation}
   \label{eq:48}
   \big\| \bar r^{(n,\delta)}_{\rm int}(t,\cdot)\big\|_{H^1[0,1]}^2 \le \frac{C}{\delta},
 \end{equation}
 which implies that
  $r^{(n,\delta)}_{\rm int}(t,u)$ converges to $r^{(\delta)}(t,\cdot)$
  uniformly in $u\in[0,1]$, as $n\to+\infty$,
  for each $t\ge0$ and $\delta>0$. This in particular entails the
  strong convergence of $(\bar r^{(n,\delta)}_{\rm int}(t,\cdot))^2$ to
  $(r^{(\delta)}(t,\cdot))^2$ in $L^1[0,1]$. We can write therefore
  \begin{align*}
    \limsup_{n\to+\infty}&\bigg|\int_0^1 {\varphi}(u) \big(\bar r^{(n)}_{\rm int}(t,u)\big)^2 \dd u-
     \int_0^1 {\varphi}(u) r^2 (t,u) \dd u\bigg|\\
    &
      \le \limsup_{n\to+\infty}\bigg|\int_0^1 {\varphi}(u) \big(\bar r^{(n,\delta)}_{\rm int}(t,u)\big)^2 \dd u-
    \int_0^1 {\varphi}(u) \big(r^{(\delta)} (t,u)\big)^2 \dd
     u\bigg|\\
    & \quad 
   +  \bigg|\int_0^1 {\varphi}(u) \big( r^{(\delta)} (t,u)\big)^2 \dd u-
    \int_0^1 {\varphi}(u) r^2 (t,u) \dd
     u\bigg|  ,
  \end{align*}
the first term on the right hand side being equal to $0$.

{Recall that $r(t,u)$ is the solution of \eqref{eq:HLstretch} and it is regular (see Definition \ref{df012701-23} and below)}, therefore
 $r^{(\delta)}(t,u)$ converges to $r(t,u)$ pointwise, and strongly in
 $L^2[0,1]$, as $\delta\to 0$. Since $\delta>0$ is arbitrary,  therefore
 \eqref{eq:35a}, thus also  \eqref{eq:35}, follows.\qed

\section{Asymptotics of the work functional: Proof of Theorem \ref{thm012209-23}}
\label{sec:work}

In this section we prove  Theorem \ref{thm012209-23}. The total work done in the macroscopic time scale is defined  by \eqref{eq:work}, and we decompose it into
\begin{equation}
  \label{eq:tw1}
  W_n(t) = W^{\text{mech}}_n(t) + W^{\text{th}}_n(t),
\end{equation}
where
\begin{align}
  \label{eq:wmech}
  W^{\text{mech}}_n(t)=  \frac 1n \int_0^{n^2 t} \oF\; \bar p_n(s) \dd s, 
, \qquad 
  W^{\rm{th}}_n(t) = \frac 1n \int_0^{n^2 t} \tilde{\cF}_n(s)\; \bar p_n(s) \dd s.
\end{align}
We compute the limits of both terms in  \eqref{eq:wmech}  in the following two propositions, which straightforwardly imply Theorem \ref{thm012209-23}. 
\begin{proposition}[Macroscopic mechanical work]
  \label{prop-wmech} For any $t\geq 0$, 
  \begin{equation}
  \lim_{n\to\infty} W^{\rm{mech}}_n(t) =
  \frac {\oF}{2\gamma} \int_0^t (\partial_u r)(s,1) \dd s.
  \label{eq:2}
\end{equation}
\end{proposition}

\begin{proof} {Using \eqref{eq:pbdf1b} and summation by parts we obtain
  \begin{align*}
    \frac 1n \sum_{x=0}^n \frac xn \left(\bar r_x(n^2 t) - \bar
      r_x(0)\right)&= \sum_{x=0}^n x\int_0^t \nabla^\star\bar p_x(n^2s)
      \dd s \\
    &
      =n\int_0^t \bar p_n(n^2s) \dd s -\sum_{x=0}^{n-1} \int_0^t \bar p_x(n^2s)
      \dd s.
  \end{align*}
  Substituting for the last term on the right from \eqref{eq:25} and
  remembering that $\sum_{x=0}^{n-1}\nabla\bar r_x(n^2s)=\bar
  r_n(n^2s)$ we obtain}
  \begin{equation}
    \label{eq:3}
    \begin{split}
    n\int_0^t \bar p_n(n^2s) \dd s = &
    \frac 1n \sum_{x=0}^n \frac xn \left(\bar r_x(n^2 t) - \bar r_x(0)\right)
    + \frac{1}{2\gamma}\int_0^t\bar r_n(n^2 s) \dd s\\
   & - \frac{1}{2\gamma n^2} \sum_{x=0}^{n-1} \left(\bar p_x(n^2 t) - \bar p_x(0)\right).
  \end{split}
\end{equation}
By \eqref{012409-23main} the last term is negligible. Using \eqref{eq:22zmain}
and Theorem \ref{th-r} with $\varphi(u)=u$, we obtain
\begin{equation}
  \label{eq:5}
  \begin{split}
    \lim_{n\to\infty}  n\int_0^t \bar p_n(n^2s) \dd s
    &= \int_0^1 u \left(r(t,u) - r(0,u)\right) \dd u +
    \frac{\oF}{2\gamma} t\\
   & = \frac 1{2\gamma} \int_0^t \dd s \int_0^1 u \partial_{uu} r(s,u)  \dd u +
    \frac{\oF}{2\gamma} t\\
  &  = - \frac 1{2\gamma} \int_0^t \dd s \int_0^1 \partial_u r(s,u)  \dd u +
     \frac 1{2\gamma} \int_0^t \dd s \partial_u r(s,1) +
     \frac{\oF}{2\gamma} t\\
   &  = - \frac 1{2\gamma} \int_0^t r(s,1)  \dd s +
     \frac 1{2\gamma} \int_0^t\partial_u r(s,1)   \dd s +
     \frac{\oF}{2\gamma} t \\ & =  \frac 1{2\gamma} \int_0^t \partial_u
     r(s,1) \dd s, 
  \end{split}
\end{equation}
since $r(s,1)\equiv \bar F$ (recall \eqref{eq:HLstretch}).
\end{proof}

\begin{proposition}[Macroscopic thermal work]
  \label{prop-heat}
 Recall the decomposition \eqref{eq:9n} for the average momentum $\bar p_n$. We have
  \begin{align}
     \label{eq:11-z}
      \lim_{n\to \infty}  \frac 1n \int_0^{n^2 t} \tilde{\cF}_n(s)\; \bar z_n(s) \dd s &= 0\\
   \label{eq:16-n}
     \lim_{n\to\infty}\frac 1n \int_0^{n^2 t} \tilde{\cF}_n(s) P_{\oF}(s)   \dd s
     &= 0 \\
   \label{eq:16b-n}
     \lim_{n\to\infty}\frac 1n \int_0^{n^2 t} \tilde{\cF}_n(s) 
       P_{dp}(s) \dd s
    & = 0
\\ 
   \label{eq:15th-n}
   \lim_{n\to\infty}\frac 1n
   \int_0^{n^2 t} \tilde{\cF}_n(s) P_{fl}(s) \dd s
  & = t \mathbb W^Q,
 \end{align}
where $\mathbb W^Q$ has been defined in \eqref{eq:17}.
These four limits imply that 
\begin{equation}
	\label{eq:43}
	\lim_{n\to\infty} W^{{\rm th}}_n(t) = t \mathbb W^Q.
\end{equation}
\end{proposition}

\begin{proof} First of all,  \eqref{eq:11-z} follows easily from \eqref{eq:4}.
Let us go on with  \eqref{eq:16-n}. 
  This is equal to
  \begin{equation}
    \label{eq:11}
     \frac{\bar F}{n^{3/2}} \sum_{\substack{\ell\not=0}}
  \hat{\cF}(\ell)\sum_{j=0}^n \frac{\psi_j^2(n)}{\Delta\la_{j}}
    \left( \frac{1-e^{(i\om(\ell)-\la_{j,-})n^2t} }{\la_{j,-}-i\om(\ell)}- \frac{1-e^{(i\om(\ell)-\la_{j,+})n^2t}}{\la_{j,+}-i\om(\ell)}\right). 
  \end{equation}
Invoke the following elementary identity:
\begin{equation}\label{eq:idphi}
{\frac{1}{B-A}\bigg(\frac{1-e^{-A\alpha}}{A}-\frac{1-e^{-B\alpha}}{B}\bigg)}=\frac{1-e^{-A\alpha}}{AB} + \frac{e^{-A\alpha}(e^{-(B-A)\alpha}-1)}{B(B-A)}.
\end{equation}
We use it  with $A=\la_{j,-}-i\ell \om $, $B=\la_{j,+}-i\ell \om $ and
$\alpha=n^2t$. The first contribution on the right hand side reads
\[ \frac{\bar F}{n^{3/2}}  \sum_{\ell \neq 0} \hat \cF(\ell) \sum_{j=0}^n  \psi_j^2(n)\frac{1-e^{(i\om(\ell)-\la_{j,-})n^2t}}{\lambda_j -
      (\ell\om)^2 -2  i  \ell\om\gamma } 
  \]  and from the exact same estimates that we already used in the
  previous proofs, see for instance \eqref{eq:ex}, we easily get that
  this contribution vanishes, as $n\to+\infty$. The remaining part is treated similarly invoking the key Lemma \ref{lm012911-23}, and we conclude
  that this term vanishes as $n \to\infty$. This ends the proof of \eqref{eq:16-n}.  

We now proceed with \eqref{eq:16b-n}.
 Recall formula \eqref{eq:Pdp} for $P_{dp}(t)$.  We have
 \[ \frac 1n \int_0^{n^2 t}
                   \cF_n(s)  
                   P_{dp}(s)  \dd s=  {\rm I}^{(dp)}_{n,1}(t)+ {\rm I}^{(dp)}_{n,2}(t)
                   \]
where
\begin{align}
   \label{022004-24}
                  {\rm I}^{(dp)}_{n,1}(t)  :=
                 \frac{\oF }{n^{3/2}}  \sum_{\ell\not=0 
              } \hat{\cF}(\ell)
                                   \sum_{j=0}^n \frac{\psi_j^2(n)}{\Delta\la_{j}}
   \bigg\{\frac{\lambda_{j,+}[1-e^{-
                 \la_{j,+}n^2t}]}{i \ell\om-\la_{j,+}   }  -\frac{\lambda_{j,-}[1-e^{-
                 \la_{j,-}n^2t}]}{i \ell\om-\la_{j,-}   } 
      \bigg\} 
 \end{align}
 and
 \begin{multline}
   \label{012004-24}
  {\rm I}^{(dp)}_{n,2}(t)  =\;
                  \frac{1}{n^{2}}  \sum_{\ell,\ell'\not=0 
              }\hat{\cF}(\ell') \hat{\cF}(\ell) 
    \sum_{j=0}^n \frac{\psi_j^2(n)}{\Delta\la_{j}}
    \bigg\{\frac{\lambda_{j,+}[1-e^{(i \omega(\ell')-
                 \la_{j,+})n^2t}]}{(i \ell\om-\la_{j,+} ) (i
                 \omega(\ell') -\la_{j,+} )} \\
   -\frac{\lambda_{j,-}[1-e^{(i \omega(\ell')- \la_{j,-})n^2t}]}{(i \ell\om-\la_{j,-} ) (i \omega(\ell') -\la_{j,-} )}
      \bigg\} .
 \end{multline}
The proof proceeds along the same lines as previously: there is a
way to decompose the terms which appear inside the brackets  on the
right hand sides of \eqref{022004-24} and \eqref{012004-24}, so that
to make the quotient $(1- e^{-\Delta\la_j n^2t})/\Delta\la_j$
appears. Then we apply the key Lemma \ref{lm012911-23}. All the
estimates are standard ones,   reminiscent to the  proof of   estimate \eqref{eq:18main}. In the end one concludes  that  both ${\rm I}^{(dp)}_{n,1}(t)$ and ${\rm I}^{(dp)}_{n,2}(t)$ vanish as $n\to\infty$. This concludes the proof of \eqref{eq:16b-n}. 

It remains to prove \eqref{eq:15th-n}. 
   Recall \eqref{eq:Pfl} for the definition of $P_{fl}(t)$. We have, similarly
\begin{align*}
  &\frac 1n
    \int_0^{n^2 t} \tilde{\cF}_n(s) P_{fl}(s) \dd s= {\rm  I}_{n,2}^{(fl)}(t)+{\rm I}_{n,3}^{(fl)}(t),
\end{align*}
where
\begin{align*}  
  {\rm I}_{n,2}^{(fl)}(t)&:=t\sum_{\ell\not=0}  |\hat{\cF}(\ell)|^2  i\ell \om 
  \sum_{j=0}^n \frac{\psi_j^2(n)}
    {\lambda_j - (\ell\om)^2 + 2 i  \ell\om\gamma},
\\
  {\rm I}_{n,3}^{(fl)}(t)&:=\frac{1}{n^{2}}\mathop{\sum_{\ell,\ell'\not=0}}_{\ell\not=-\ell'} \hat{\cF}(\ell) \hat{\cF}(\ell')  \frac{  \ell\om(1-e^{ i(  \ell\om+\ell'\om )n^2t})}{\ell\om -\ell'\om }
  \sum_{j=0}^n \frac{\psi_j^2(n)}
    {\lambda_j - (\ell\om)^2 + 2 i  \ell\om\gamma}.
\end{align*}
It is straightforward to argue that ${\rm I}_{n,3}^{(fl)}(t)$ vanish as $n\to\infty$.
Concerning ${\rm I}_{n,2}^{(fl)}(t)$, thanks to the fact that  
$\hat{\cF}(-\ell)=\hat{\cF}^\star(\ell)$, we have
\begin{align*}
  &
  {\rm I}_{n,2}^{(fl)}(t)=4\ga t\sum_{\ell=1}^{+\infty}   (\ell\om)^2|\hat{\cF}(\ell)|^2  
  \frac{1}{n+1} \sum_{j=0}^n\frac{    (2-\delta_{j,0})\cos^2\big(\frac{\pi j}{2(n+1)}\big)}
    {\big(4\sin^2\big(\frac{\pi j}{2(n+1)}\big)- (\ell\om)^2\big)^2
    + 4 (\ell\om)^2 \gamma^2}.
\end{align*}
Therefore
\begin{align*}
  &
  \lim_{n\to+\infty}{\rm I}_{n,2}^{(fl)}(t)=8\ga t\sum_{\ell=1}^{+\infty}   (\ell\om)^2|\hat{\cF}(\ell)|^2  
  \int_0^1\frac{   \cos^2\big(\frac{\pi u}{2}\big)\dd u}
    {\big(4\sin^2\left(\frac{\pi u}{2}\right)- (\ell\om)^2\big)^2
    + 4 (\ell\om)^2 \gamma^2}.
\end{align*}
We now need to prove that last expression equals \eqref{eq:15th-n}.
  Observe that
  \begin{equation*}
    \frac{1}
    {\big(4\sin^2\left(\frac{\pi u}{2}\right)- (\ell\omega)^2\big)^2
      + 4 \gamma^2 (\ell\omega)^2} = - \frac{1}{2\gamma\omega\ell}
    \text{Im} \bigg[   \frac{1}
    {4\sin^2\left(\frac{\pi u}{2}\right)- (\ell\omega)^2
      + i 2\gamma  \ell\omega} \bigg].
  \end{equation*} 
Using  contour integration one can show that   
  \begin{align*}
     &\int_0^1\frac{   \cos^2\left(\frac{\pi u}{2}\right)\dd u}
       {4\sin^2\left(\frac{\pi u}{2}\right) + \la }=
       \frac{1}{4}\bigg(1-i\sqrt{-\frac{4 +\la}{\la}}\bigg),
  \end{align*}
  for any complex valued $\la$ such that ${\rm Im}\,\la>0$.
  Using the above formula for
$\la =  -\left(\ell\om\right)^2 + 2\ga i\ell\om$,
    when $\ell$ is a positive integer, we get
    \begin{equation}
   \label{eq:17a}
    \mathbb W^Q = \sum_{\ell=1}^{+\infty}
     |\hat{\cF}(\ell)|^2 \; \ell\om  {\rm Re}\,\sqrt{{\frac{4}{(\ell\om)^2- i 2\gamma \ell\om}-1}}.
   \end{equation}
   which is the expression \eqref{eq:17} for $\mathbb W^Q$.
   \end{proof}

\section{Energy bounds from entropy production}
\label{sec:energy-bounds}

Before turning to the proof of the convergence for the total energy
profile, in this section  we provide a few important preliminary results. 
Recall that $\mu_n$ is the initial distribution of the momenta and
stretches. 
We  show that the macroscopic energy functional stays bounded
in the macroscopic time. Recall the notation $\mathcal{H}_n$
for the total microscopic energy introduced in \eqref{eq:hn} and the initial bound \eqref{eq:ass2entropy}.
{The proof of the following theorem is very similar to the one of Proposition 3.1 in \cite{klo22-2}
  in the case of a pinned chain. We include here for the convenience of the reader.}

\begin{theorem}[Energy bound] \label{thm:energy}
  Under Assumptions \ref{ass1}, \ref{ass2}, \ref{ass3}, 
  there exists a constant $C>0$ such that, for any $t\ge 0$ and $n= 1,2,\ldots$,
\begin{equation}
  \label{eq:44a}
 \mathbb E_{\mu_n} \big[\mathcal{H}_n (n^2t)\big]  \le C n(t+1).
\end{equation}
\end{theorem}
\proof
{Taking the time derivative of the relative entropy $\mathbf{H}_n$
defined by \eqref{eq:7-1} we have
\begin{align*}
  \frac{\dd}{\dd t} {\mathbf{H}}_n(n^2t) 
  &  =\frac{\dd}{\dd t}\underbrace{\int_{\Om_n}   f_n(n^2t ) \dd\nu_{T_-} }_{=1} +n^{2} \int_{\Om_n}
    f_n(n^2t ) {\cal G}_t\log f_n(n^2t )
    \dd\nu_{T_-} 
  \\
  &
    =n^{2} \int_{\Om_n}
    f_n(n^2t ){\cal G}_t\log  f_n(n^2t ) \dd\nu_{T_-} .
\end{align*}
Using formulas \eqref{eq:7}--\eqref{eq:10} for the generator we obtain}
\begin{equation}
  \label{eq:46}
  \begin{split}
    \frac \dd{\dd t} {\mathbf{H}}_n(n^2t)  
    & = n^2 \int_{\Om_n} \cF_n(n^2t)  \partial_{p_n} f_n(n^2t) \; \dd\nu_{T_-}
    - {2\ga n^2 T_-} \int_{\Om_n} \frac{(\partial_{p_0} f_n(n^2t))^2}{ f_n(n^2t)} \; \dd\nu_{T_-} \\
    &\quad + n^2\gamma \int_{\Om_n} f_n(t) S_{\text{flip}} \log f_n(t) \; \dd\nu_{T_-}.
  \end{split}
\end{equation}
Since the last term on the right hand side of \eqref{eq:46} involving $S_{\text{flip}}$ is negative,
we have
\begin{equation}
\begin{split}
 {\mathbf{H}}_n(n^2t) + T_-n^2 &\int_0^t \dd s \int_{\Om_n} \frac{(\partial_{p_0} f_n(n^2s))^2}{ f_n(n^2 s)} \; \dd\nu_{T_-}\; 
 \\
& \le  {\mathbf{H}}_n(0)+n^2 \int_0^t \;
  \cF_n(n^2s) \dd s\int_{\Om_n} \partial_{p_n} f_n(n^2s) \; \dd\nu_{T_-}\\
&  = {\mathbf{H}}_n(0)  + \frac 1{T_-} \int_0^t \cF_n(n^2s) \bar p_n(n^2s) \;\dd s 
  = {\mathbf{H}}_n(0)+ \frac {nW_n(t)}{ T_-} .
\end{split}\label{eq:47}
\end{equation}
By the assumption \eqref{eq:ass2entropy}  on the initial entropy and
Theorem \ref{thm012209-23} which gives the limiting work, we therefore obtain the entropy bound 
\begin{equation}
 {\mathbf{H}}_n(n^2t) \le Cn(t+1),\qquad n\ge 1, \; t\ge0.
 \label{eq:47a}
\end{equation}
Then, as we did already for the initial time in \eqref{entropy-in-tz}, we use the entropy inequality  and from \eqref{eq:47a} 
  we conclude that:
  \begin{equation}
\label{entropy-in-t}
\begin{split}
  \mathbb E_{\mu_n} \big[ \mathcal{H}_n (n^2t)\big] & =\int_{\Om_n}
\Big(\sum_{x=0}^n{\cal E}_x\Big)   f_n(n^2t) \dd
 \nu_{T_{-}}
 \\
 &
 \le \frac{1}{\al}\left\{\log\bigg(\int_{\Om_n}
     \exp\left\{ \frac\al 2 \sum_{x=0}^n(p_x^2+r_x^2)
     \right\}\dd  \nu_{T_{-}} \bigg)+ {\mathbf{H}}_{n}(n^2t)\right\}
\end{split}
 \end{equation}
 for any $t\ge0$ and $\alpha>0$.
 Hence for any $\al\in(0,T_{-}^{-1})$ we can find $C_\al >0$ such that
\begin{equation}\label{eq:boundent}
 \mathbb E_{\mu_n} \big[  \mathcal{H}_n (n^2t) \big]
  \le \frac{1}{\al}\big(C_\al n+ \mathbf{H}_{n}(n^2 t)\big),\qquad t\ge0,
 \end{equation}
 and \eqref{eq:44a} follows.
 \qed

 \medskip

 From the proof of  Theorem \ref{thm:energy} (see \eqref{eq:47}) we also obtain the following: 
 \begin{corollary} There exists $C>0$ such that, for any $n$ and $t\ge 0$:
   \begin{equation}
     \label{eq:49}
     T_- \int_0^t \dd s\int_{\Om_n} \frac{[\partial_{p_0} f_n(n^2s)]^2}{ f_n(n^2s)} \; \dd\nu_{T_-} \le \frac {C(t+1)}{n}.
   \end{equation}
 \end{corollary}
 
Finally, we give an important corollary about the boundary behavior:

\begin{corollary}[Current estimate and boundary temperature]\label{cur}
	For any $t_*>0$, there exists $C=C(t_*)>0$ such that, for any $n$ and $t\in[0,t_*]$,
	\begin{equation}
		\label{eq:45}
		\left|\int_0^t \left(T_- - \EE_{\mu_n} \big[p_0^2(n^2 s)\big]\right) \dd s\right| \le \frac{C}{n}
	\end{equation}
 and
   \begin{equation}
 	\label{eq:cur}
 	\sup_{x=0,\ldots,n}\left|\int_0^t \EE_{\mu_n}\big[j_{x-1,x}(n^2 s)\big] \dd s\right| \le
 	\frac{C }{n}.
 \end{equation}
\end{corollary}

\proof

 Recall that the time derivative of the microscopic energy is a local gradient, see \eqref{eq:current}. We conclude that the total energy evolves in time according to the equation
\begin{equation}
  \label{eq:44}
 \mathbb E_{\mu_n} \big[ \mathcal{H}_n (n^2t) \big] =  \mathbb E_{\mu_n}\big[\mathcal{H}_n(0)\big]
  +  2\gamma n^2\int_0^t \left(T_- -\mathbb E_{\mu_n}\big[ p_0^2(n^2
    s)\big]\right) \dd s + n W_n(t).
\end{equation}
Thanks to Theorem \ref{thm:energy}, see the energy bound \eqref{eq:44a}, and
Theorem \ref{thm012209-23} about the limiting work, we
conclude \eqref{eq:45}.

Let us now prove \eqref{eq:cur}.
  Using once again the local gradient \eqref{eq:current} we conclude that
  \[
 \bigg| \int_0^t j_{x-1,x}(n^2 s) \dd s-\int_0^t j_{x,x+1}(n^2 s) \dd
 s\bigg|=\frac{1}{n^2}\big|{\cal E}_x(n^2 t) -{\cal E}_x(0)\big|
 \]
 Hence, by the boundary estimate \eqref{eq:45} and the energy bound \eqref{eq:44a} we have
 \[
 \bigg| \int_0^t {\EE_{\mu_n}} \big[j_{x,x+1}(n^2 s) \big] \dd
 s\bigg|\le \frac{1}{n^2}\sum_{y=0}^x {\EE_{\mu_n}} \big[{\cal E}_y(n^2 t) +{\cal E}_y(0)\big]+\frac{C}{n}\le \frac{C'}{n}
 \]
 and \eqref{eq:cur} follows.
 \qed

 \section{The macroscopic energy equations}
 \label{sec:proof-macr-energy}
 
 The main goal of this section is to prove both convergences \eqref{eq:conv-energy}
 and \eqref{eq:conv-temp} asserted in  Theorems \ref{th:hl} and \ref{thm012911-23}. Using the current identity \eqref{eq:current}, one can easily see that an important quantity which needs to be controlled is 
 $\llangle j_{x,x+1}\rrangle_t$, where $\llangle \cdot \rrangle_t$ has
 been defined in \eqref{eq:bracket-t}. This is why we start with some
 preliminary computation, which gives an adequate
 \emph{fluctuation-dissipation relation}. Then, we sketch the proof at
 the end of this section, and   postpone the proofs of intermediate
 technical results  to the forthcoming sections.
 
 \subsection{Covariance matrix and fluctuation-dissipation relation}
 Recall the following notations, for the fluctuating parts of the configurations:
 \begin{align}
 	r'_x(t)&:=r_x(t)-\bar r_x(t),\qquad\; x=1,\dots,n,\notag\\
 	p'_x(t)&:=p_x(t)-\bar p_x(t), \qquad x=0,\dots,n.\label{pqpp2}
 \end{align}
We introduce the block \emph{covariance matrix}:
 	\begin{equation}
 		\label{S1ts}
 		S(t)
 		:=\begin{bmatrix}
 			{S^{(r)}(t)}&S^{(r,p)}(t)\\
 			S^{(p,r)}(t)& S^{(p)}(t)
 		\end{bmatrix},
 	\end{equation}
 	where
 	\begin{align}
 		\label{S1ts1}
 		&S^{(r)}(t)=\Big(\bbE_{\mu_n}[r_x'(t)r_y'(t)]\Big)_{x,y=1,\ldots,n,},\quad S^{(r,p)}(t)=\Big(\bbE_{\mu_n}[r_x'(t)p_y'(t)]\Big)_{x=1,\ldots,n,y=0,\ldots,n},\notag\\
 		&S^{(p)}(t)=\Big(\bbE_{\mu_n}[p_x'(t)p_y'(t)]\Big)_{x,y=0,\ldots,n}\quad \mbox{and}\quad S^{(p,r)}(t)=\Big(S^{(r,p)}(t)\Big)^T. 
 	\end{align}
 Similarly to \eqref{eq:bracket-t}, we use the notation
 		\[ \lang S \rang_t:= \frac 1t \int_0^t S(n^2s)\; \dd s.\] 
 		We rewrite the time average of expression \eqref{eq:current-bis}
                in terms of the second moment matrix $S^{(p,r)} $ and of the work
 		functional, 
                as follows:
 		\begin{equation}
 			\label{eq:112}
 				\lang  j_{x,x+1} \rang_t = - \lang  S^{(p,r)}_{x,x+1} \rang_t - 
 				\lang  \bar p_x  \bar r_{x+1} \rang_t , \qquad x=0,\dots, n-1, \end{equation} and at the boundaries \begin{equation}
 				\lang  j_{n,n+1} \rang_t = - \frac 1{nt} W_n(t),\qquad
 				\lang  j_{-1,0} \rang_t = 2\gamma \left(T _-- \lang  p_0^2 \rang_t\right). \label{eq:112bound}
 		\end{equation}
 Finally, the time average of the microscopic density of  thermal energy  ${\cal
 	E}'_x(t)$, see  \eqref{eq:hn1}, is  given by 
 \begin{equation}
 	\begin{split}
 		\lang  {\cal E}'_x \rang_t = \frac 12 \left(\lang S^{(p)}_{x,x}\rang_t + \lang S^{(r)}_{x,x}\rang_t \right).
 	\end{split}\label{eq:114}
      \end{equation}	The main result of this section is the following:
      
 	\begin{lemma}[Fluctuation-dissipation relation] \label{lem:fluc}
 	We have
 		\begin{equation}
 			\label{eq:115ebis}
 	 	\llangle S_{x,x+1}^{(p,r)} \rrangle_t =
                      \mm{-}  \frac{1}{4\ga} \nabla \mathcal{U}_x(t) +
                        \mathcal{V}_{x}(t), 	\qquad\mbox{for any $x=0,\dots,n-1$,} 
 		\end{equation}where
 	\begin{align*}
 		\mathcal{U}_x(t) & := \lang  {\cal E}'_x \rang_t + \frac 12 \lang S^{(r)}_{x,x+1}\rang_t +
 		\frac 12 \lang S^{(p)}_{x-1,x}\rang_t +\gamma \lang
 		S^{(p,r)}_{x,x}\rang_t , \qquad x=0,\dots,n\\
 		\mathcal{V}_{x}(t) & :=  \frac{1}{n^2t}\bbE_{\mu_n}\big[V^e_{x}(0)-V^e_{x}(n^2 t)\big], \qquad x=0,\dots,n-1,
 		\end{align*}
 	with  \begin{equation}
 		\label{V}
 		V^e_x :  
 		= {\frac 1{8\gamma}
 		\left( 2r'_{x+1}p'_x+p'_{x+1}r'_{x+1}+ p'_xr'_x\right) + \frac 14 (r'_{x+1})^2}.
              \end{equation}
              By convention $r_0=0$,  $\lang S^{(r)}_{0,1}\rang_t = 0$
              and {$\lang S^{(r)}_{n,n+1}\rang_t=0$}.  Moreover,  for any $t>0$
              there exists $C=C(t)>0$ such that 
 \begin{equation}
 \label{022909-23}
 \sum_{x=0}^{n-1} |{\cal V}_{x}(t)| \le \frac Cn,\quad n=1,2,\ldots.
\end{equation}
 	\end{lemma}

        \proof
        We use the following straightforward fluctuation-dissipation decompositions
        of the energy current: for $x=0,\dots,{n-1}$ we have
 {\begin{equation}
   \label{eq:FDe}
  	j_{x,x+1} = - \frac 1{4\gamma} \nabla U_{x}^e   + \mathcal G_t
 V^e_x,
 \end{equation}
 with
$     U_x^e:={\cal E}_x+ \frac 12 (r_xr_{x+1} + p_{x-1}p_x)
     +\gamma p_xr_x$. Here the convention at the right extremity reads $r_{n+1}\equiv \cF_n$. Relation
     \eqref{eq:FDe} also holds  with
     {the centered quantities, namely 
   \begin{equation} \label{eq:bar}
     -\bar p_x \bar r_{x+1} = - \frac 1{4\gamma} \nabla \bar U_{x}^e   + \bar{\mathcal G}_t\bar V^e_x,
   \end{equation}
   with $\bar U^e$ and $\bar V^e$ be defined as $U^e$ and $V^e$,
   but with every $r$ and $p$ replaced with $\bar r$ and $\bar p$.}}
 From \eqref{eq:FDe}, {its centered version}, and \eqref{eq:112}, we conclude \eqref{eq:115ebis}.
 
 Besides, \eqref{022909-23} is a direct consequence of the energy bound \eqref{eq:44a}
 given in Theorem \ref{thm:energy} {together with the control of $\ell^2$ norms of averages, given in \eqref{012409-23main} and \eqref{012409-23main1} from Proposition \ref{prop:main}}. \qed

 \subsection{Limit of the energy functionals}
 
 Consider now the evolution of the energy distribution  functional
 \begin{equation}
 	\label{eq:54}
 	\xi_n^e(\varphi,t) := \frac 1{n} \sum_{x=0}^n \varphi_x \bbE_{\mu_n}\left[\E_x(n^2 t)\right],
 \end{equation}
 where $\varphi:[0,1]\to\R$ is continuous and  $\varphi_x = \varphi(\frac x{n})$.
 By a standard approximation argument it is enough to consider test functions $\varphi\in {\cal C}^\infty[0,1]$ such that
 \begin{equation}
 	\label{020910-23}
 	{\rm supp}\,\varphi''\subset (0,1),\quad  \varphi(0) = 0 \quad \mbox{and}\quad \varphi'(1) = 0.
 \end{equation}
 The above implies that there exists $u_*>0$ such that
 \begin{equation}
 	\label{020910-23a}
 	\varphi(u)= \begin{cases} {\varphi'(0) u} ,& \quad u\in[0,u_*)\\
 		\varphi(1) , & \quad u\in(1-u_*,1].
 	\end{cases}
 \end{equation}
 From equation
 \eqref{eq:current} we get
 \begin{align*}
 	\xi_n^e(\varphi,t) - \xi_n^e(\varphi,0) 
 	=-\frac {n^2 t}{n} \sum_{x=0}^n
 	\varphi_x \nabla^\star\lang j_{x,x+1}\rang_t.
 \end{align*}
 Hence, summing by parts, and then using  \eqref{eq:112} and \eqref{eq:112bound}, plus the fact that $\varphi_0=0$, we obtain
 \begin{align}
\notag 	 \xi_n^e(\varphi,t) - \xi_n^e(\varphi,0) 
 &	=  t \sum_{x=0}^{n-1} n\nabla\varphi_x  \lang  j_{x,x+1}\rang_t
 	- n t \varphi_n \;  \lang  j_{n,n+1}\rang_t\\
 	&=  - t \sum_{x=0}^{n-1} (n\nabla\varphi_x)  \lang S_{x,x+1}^{(p,r)}\rang_t
 	- t  \sum_{x=0}^{n-1} (n\nabla\varphi_x) \lang \bar p_x \bar r_{x+1} \rang_t 	+  \varphi(1) W_n(t). \label{eq:xie}
 \end{align}
Recall that $W_n(t)$ converges to $W(t)$, as $n\to +\infty$ (from Theorem \ref{thm012209-23}). This takes care of the last term.
 
Let us focus on the second term on the right hand side. Using the second
 equation of \eqref{eq:pbdf1b} and then performing the  integration by parts (in time), we can rewrite
 \begin{align*}
 	\lang \bar
 	p_x \bar r_{x+1}\rang_t & =\frac{1}{2\ga} \lang \nabla\bar
 	r_{x} \bar r_{x+1} \rang_t -\frac{1}{2\ga n^2}\lang \frac{\dd\bar  p_{x}}{\dd s}\; \bar
 	r_{x+1} \rang_t\\
 	&
 	=  \frac{1}{2\ga} \lang \nabla\bar
 	r_{x}\bar
 	r_{x+1}\rang_t+ \frac{1}{2\ga n^2} \lang {\bar  p}_{x}\; \frac{\dd {\bar
 		r}_{x+1}}{\dd s} \rang_t +\frac{1}{2\ga n^2 t}\big[\bar  p_{x}(0) \bar
 	r_{x+1}(0)-\bar  p_{x}(n^2t) \bar
 	r_{x+1}(n^2t)\big] \\
 	&
 	=  \frac{1}{2\ga} \lang \nabla\bar
 	r_{x}\bar
 	r_{x+1}\rang_t+ \frac{1}{2\ga} \lang {\bar  p}_{x}\nabla \bar p_x \rang_t +\frac{1}{2\ga n^2 t}\big[\bar  p_{x}(0) \bar
 	r_{x+1}(0)-\bar  p_{x}(n^2t) \bar
 	r_{x+1}(n^2t)\big].
 \end{align*}
We use this decomposition to express   the second term of \eqref{eq:xie}.
Then
\begin{equation}\label{eq:xidecomp}
 - t  \sum_{x=0}^{n-1} (n\nabla\varphi_x) \lang \bar p_x \bar r_{x+1} \rang_t =
  \mathrm{B}_n(t)+\mathrm{C}_n(t)+\mathrm{D}_n(t)
\end{equation}
where 
\begin{align*}
	\mathrm{B}_n(t)&:= - \frac{t}{2\ga}\sum_{x=0}^{n-1} (n\nabla \varphi_x) \lang \nabla\bar r_x\bar r_{x+1}\rang_t \\
	\mathrm{C}_n(t)&:= - \frac{t}{2\ga}\sum_{x=0}^{n-1} (n\nabla \varphi_x) \lang \bar p_x \nabla\bar p_{x}\rang_t\\
	\mathrm{D}_n(t)&:= - \frac{1}{2\ga n^2}\sum_{x=0}^{n-1} (n\nabla \varphi_x)\big[\bar  p_{x}(0) \bar
	r_{x+1}(0)-\bar  p_{x}(n^2t) \bar r_{x+1}(n^2t)\big].
\end{align*}
Let us consider each of these contributions separately.
We start with $\mathrm{B}_n(t)$. 
Using the identity $(\bar r_{x+1}-\bar r_x)\bar r_{x+1}=\frac12(\bar r_{x+1}^2-\bar r_x^2) + \frac12 (\bar r_{x+1}-\bar r_x)^2$ we rewrite it as
\[ 
- \frac{t}{4\ga}\sum_{x=0}^{n-1}(n\nabla \varphi_x) \lang \nabla (\bar r_x^2)\rang_t - \frac{t}{4\ga}\sum_{x=0}^{n-1}(n\nabla \varphi_x)\lang (\nabla \bar r_x)^2\rang_t.
\]
 Using \eqref{eq:36} we conclude
\begin{equation}
	\label{eq:56}
t \sum_{x=0}^{n-1}(n\nabla \varphi_x)\lang (\nabla \bar r_x)^2\rang_t = n \sum_{x=0}^{n-1}
	\nabla\varphi_x \int_0^{t}\left(\nabla\bar  r_{x}\right)^2( n^2s)\; \dd s
	\xrightarrow[n\to\infty]{} 0.
\end{equation}
Performing the summation by parts we obtain
\begin{align}
		- \frac{1}{4\gamma} \sum_{x=0}^{n-1} (n\nabla\varphi_x) &\int_0^{t}\nabla\bar  r^2_{x}( n^2s)\; \dd s \notag \\
		&  = \frac{1}{4\gamma n}   \sum_{x=1}^{n-1} (n^2\Delta\varphi_x) \int_0^t \bar r^2_{x}( n^2s)\; \dd s
		- \frac{n\nabla\varphi_{n-1}}{4\ga} \int_0^t \bar r_n^2 (n^2s) \; \dd s \notag\\
		&  \xrightarrow[n\to\infty]{}  \frac{1}{4\gamma} \int_0^t
		\dd s \int_0^1 \; \varphi''(u) r^2(s,u)\dd u \label{eq:57}
	\end{align}
from \eqref{eq:35} together with  \eqref{020910-23a} for the boundary term (note that  we have $n\nabla \varphi_{n-1}\equiv 0$ for $n$ large).
Above, $r(t,u)$ is the solution to \eqref{eq:HLstretch}. Therefore the limit of $B_n(t)$ is given by \eqref{eq:57}.
From the Cauchy-Schwarz inequality (applied twice) and \eqref{012409-23main}, we bound
\begin{align} |\mathrm{C}_n(t)|&=\bigg|\frac 1{2\gamma} \sum_{x=0}^{n-1}
(n\nabla\varphi_x) \int_0^{t} \bar p_x(n^2 s) \nabla \bar p_x(n^2s)\;
\dd s \bigg| \notag
\\
 & \le C \|\varphi'\|_\infty \bigg|\sum_{x=0}^{n-1}  \int_0^t \bar p_x^2(n^2s) \dd s\bigg|^{1/2} \; \bigg|\sum_{x=0}^{n-1}  \int_0^t(\bar p_{x+1}-\bar p_x)^2(n^2s)\dd s\bigg|^{1/2} \label{eq:cscn} \le \frac{C}{n}.
 \end{align}
 Therefore $\mathrm{C}_n(t)$ vanishes as $n\to+\infty$. Finally, by Proposition \ref{prop:main}, we also have
 \begin{align*}
 \mathrm{D}_n(t)=-	\frac 1{2\gamma n^2} \sum_{x=0}^{n-1}  n\nabla\varphi_x
 	\left(\bar r_{x+1}( 0) \bar p_x( 0) - \bar r_{x+1}(n^2t) \bar
 	p_x(n^2t)\right)  \xrightarrow[n\to\infty]{} 0.
 \end{align*}
 It remains to deal with the first term on the right hand side of \eqref{eq:xie}.
 We use the fluctuation-dissipation relation given in \eqref{eq:115ebis}.
 Then we perform the summation  by parts, and use the estimate from \eqref{022909-23}.
 Doing this, we obtain that this term is equal to
 \begin{multline}
 - t \sum_{x=0}^{n-1}  (n\nabla\varphi_x)  \lang S_{x,x+1}^{(p,r)}\rang_t= - \frac {t}{4\ga} \sum_{x=1}^{n-1}  (n\nabla\varphi_x)  \nabla
 	\left(\mathcal{U}_x(t)\right) +o_n(1) 
 \\ \label{eq:firsttwo}
 	= \frac { t}{4\ga n} \sum_{x=2}^{n-1}  (n^2\Delta \varphi_x)  \;
 	\mathcal{U}_x(t) 
 	+\frac {t}{4\ga} (n\nabla\varphi_{1}) \;\mathcal{U}_1(t)   - \frac {t}{4\ga}(n\nabla\varphi_{n-1})\; \mathcal{U}_n(t).
 \end{multline}
 As before  we know that  $n\nabla\varphi_{n-1}\equiv 0$
 for a sufficiently large $n$.
 
 In order to treat the first two terms on the right hand side,
   we state the following important  result, which will be proven in
   Section \ref{sec:proof-equipartition}, after preliminary work done
   in Section \ref{sec:consequences}.

 \begin{theorem}[Local equilibrium and boundary moments]\label{thm-loc-eq}
 	For any $\varphi \in \mathcal C[0,1]$ compactly supported in $(0,1)$, we have
 	\label{thm-loc-equi}
 	\begin{align}
 		\label{eq:116u}
 		\lim_{n\to\infty}  \frac 1{n} \sum_{x=0}^{n} \varphi_x \lang
 		S^{(p,r)}_{x,x+h}\rang_t& = 0, \quad \text{for any } h\in \{-1,0,1,2\},
 		\\
 		\label{eq:116y}
 		\lim_{n\to\infty}  \frac 1{n} \sum_{x=0}^{n-1} \varphi_x \lang S^{(p)}_{x,x+1}\rang_t& = 0,\\
 		\label{eq:116x}
 		\lim_{n\to\infty}  \frac 1{n} \sum_{x=1}^{n-1} \varphi_x \lang
 		S^{(r)}_{x,x+1}\rang_t& = 0. \end{align}
              Moreover, at the boundaries: 
 	  \begin{align}
 		\label{eq:119}
 		\lim_{n\to\infty}  \lang S^{(p)}_{1,1}\rang_t = T_-, &\qquad 
 		\lim_{n\to\infty}  \lang S^{(p)}_{0,1}\rang_t  = 0, \\ 
 	\label{eq:119zz}   
 	\lim_{n\to\infty}
 	\lang S^{(p,r)}_{1,1}\rang_t = 0, &\\
 	\lim_{n\to\infty}
           \lang S^{(r)}_{1,1}\rang_t = T_-, &\qquad 
 	\label{eq:119c}   
 	 \lim_{n\to\infty}  \lang S^{(r)}_{1,2}\rang_t = 0. \qquad &
 \end{align}
 \end{theorem}

Let us show how to use Theorem \ref{thm-loc-eq}. By its definition (recall Lemma
\ref{lem:fluc}), ${\cal U}_1(t)$ can be fully expressed in terms of the covariance matrix as
 \[ {\cal U}_1(t)=\frac12\big( \lang S_{1,1}^{(p)}\rang_t+ \lang S_{1,1}^{(r)}\rang_t+\lang S_{1,2}^{(r)}\rang_t + \lang S_{0,1}^{(p)}\rang_t \big) + \gamma \lang S_{1,1}^{(p,r)}\rang_t. \]
 As a consequence of \eqref{eq:119}--\eqref{eq:119c}, we are able to conclude that 
 \[\frac{t}{4\gamma} (n\nabla \varphi_1) {\cal U}_1(t) \xrightarrow[n\to\infty]{} \frac{1}{4\gamma} \varphi'(0) T_-.\]
 From the definition of $\mathcal{U}_x(t)$  and the fact that
 \begin{equation}
   \label{012204-24bis}
   {\cal E}'_x = {\cal E}_x - \frac12 \bar p_x^2 - \frac12 \bar
   r_x^2,\end{equation}
 we write 
 \[ \frac { t}{4\ga n} \sum_{x=2}^{n-1}  (n^2\Delta \varphi_x)  \;
 \mathcal{U}_x(t) = \mathrm{I}_n(t) + \mathrm{II}_n(t),\]
 where 
 \begin{align*}
 	\mathrm{I}_n(t)&:= \frac{t}{4\gamma n} \sum_{x=2}^n (n^2 \Delta \varphi_x) \Big(\lang {\cal E}_x \rang_t  - \frac12\lang \bar p_x^2 \rang_t - \frac12 \lang \bar r_x^2 \rang_t \Big), \\
 	\mathrm{II}_n(t)&:=\frac{t}{4\gamma n} \sum_{x=2}^n (n^2 \Delta \varphi_x) \Big( \frac 12 \lang S^{(r)}_{x,x+1}\rang_t +
 	\frac 12 \lang S^{(p)}_{x-1,x}\rang_t +\gamma \lang
 	S^{(p,r)}_{x,x}\rang_t \Big).
 \end{align*}
 As a direct consequence of Theorem \ref{thm-loc-equi}, we conclude  that $\mathrm{II}_n(t)$ vanishes as $n\to+\infty$. The limit of $\mathrm{I}_n(t)$ is easier: the first contribution involving $\lang{\cal E}_x \rang_t$ will allow us to close the limit equation satisfied by $\xi_n^e(\varphi,t)$, while the remaining part can be handled thanks to the previous sections: from  \eqref{eq:34} and \eqref{eq:35} we easily get
 \[ \mathrm{I}_n(t) - \frac1{4\ga} \int_0^t \xi_n^e(\varphi'',s)\dd s  \xrightarrow[n\to\infty]{}-\frac 1{8\gamma}\int_0^t\varphi''(u)r^2(s,u)\dd s. \]
 Summarizing the entire argument 
 we have proven that:
 \begin{multline}
 	\label{eq:59}
 		\lim_{n\to\infty} \left[ \xi_n^e(\varphi,t) - \xi_n^e(\varphi,0) -
 		\frac 1{4\gamma} \int_0^t\xi_n^e(\varphi'',s)\dd s\right] \\
 		 =  \frac 1{8\gamma}\int_0^t \dd s \int_0^1 \dd u\; \varphi''(u) r^2(s,u)
 		+  \varphi'(0) \frac{T_-}{4\gamma} + \varphi(1) W(t),
 	\end{multline}
 which identifies the limit as the solution of the initial-boundary value problem \eqref{eq:pde-energy}.  
 The compactness of the sequence $\{\xi_n^e\}$ in $\cC\left([0,t_\star], \mathcal M_K([0,1])\right)$
 follows by the same argument as for $\{\xi_n^r\}$, and this concludes the proof of Theorem \ref{th:hl}. \qed

\medskip

Let us conclude this section with the outline of the proof of Theorem
\ref{thm012911-23}. First of all, the convergence \eqref{eq:conv-temp}
is obvious thanks to   Theorem \ref{th:hl}, the
decomposition \eqref{012204-24bis} and the fact 
 that we already know the limit of the fields \eqref{eq:34} and \eqref{eq:35}.
The second convergence result \eqref{eq:conv-temp1} follows from an \emph{equipartition result} stated below: 
 
 \begin{proposition}[Equipartition]
	\label{cor010910-23}
	For any $\varphi\in {\cal C}[0,1]$ compactly supported  in
        $(0,1)$ we have
	\begin{equation}
		\label{eq:60}
		\lim_{n\to\infty}
		\frac 1 n \sum_{x=1}^n \varphi_x \left[\lang (p'_x)^2\rang_t - \lang (r'_x)^2\rang_t\right] = 0.
	\end{equation}
\end{proposition}
 Proposition \ref{cor010910-23}
 will also be proved in Section \ref{sec:proof-equipartition}.

\section{Evolution of the covariance matrix and consequences}
\label{sec:cov}

\subsection{Resolution}

In this section we state and solve  the system of equations satisfied by the covariance matrix defined in \eqref{S1ts}.

 From \eqref{eq:flipbulk}--\eqref{eq:flipboundary}  and
\eqref{011406-22} 
we get the following system of equations
\begin{align} 
	\label{eq:fflip-1}
		\dd { r}'_x(t) &= \nabla^\star p'_x(t)\dd t ,
		\\
		\dd { p}'_x(t) &=  \nabla r'_x (t)\dd t- 2\ga  p_x'(t-) \dd t+
		2(1-\delta_{x,0})  
		p_x(t-) \dd\tilde N_x(\gamma t)  
		+\delta_{x,0} \sqrt{4 \gamma T_-} \dd \tilde w_-(t)  \notag
\end{align}
for $ x=0, \dots, n$.
Here $
\{\tilde N_x(t):=N_x(t)-t \; ; \; x=1,\ldots,n\}
$ are independent  martingales.
	Furthermore, we define
	\begin{equation}
		\begin{split}
			\Sigma ({\bf p}) =\begin{bmatrix}
				0_{n}&0_{n+1,n}\\
				0_{n,n+1}&D({\bf p})
			\end{bmatrix},
			\,\mbox{ with }
			D({\bf p}) =
			\begin{pmatrix}
				\sqrt{4 \gamma T_-} & 0 & 0 &\dots&0\\
				0& 2  p_1 &  0 &\dots&0\\
				0 & 0 & 2  p_2 &\dots&0\\
				\vdots & \vdots & \vdots & \vdots&\vdots\\
				0& 0 & 0 & \dots &2  p_n
			\end{pmatrix}.
		\end{split}
		\label{eq:22}
	\end{equation}
	The column vector solution ${\bf X}'(t)=[{\bf r}'(t),{\bf p}'(t)]^T$ of
	\eqref{eq:fflip-1} satisfies
	\begin{equation}
		\label{Xts0}
		{\bf X}'(t)=e^{-At}{\bf X}'(0)+\int_{0}^t
		e^{-A(t-s)}\Sigma \big({\bf p}(s-)\big)\dd M(s),\quad t\ge0.
	\end{equation}
	Here $A$ is defined by \eqref{bA} and
	$\left(M(t)\right)_{t\ge0}$ is $(2n\!+\!1)$--dimensional column vector martingale
	\[
	\dd M(s) =\big[0,\dots,0,                  \dd \tilde{w}_-(s),                   \dd \tilde N_1(\gamma s), \dots,
	\dd \tilde N_n(\gamma s)\big]^{T}.
	\]
Recall definition \eqref{S1ts} for the covariance matrix $S(t)$. 
 From \eqref{Xts0} we easily get
      \begin{equation}\label{eq:covev}
        \frac 1{n^2}
        \frac{\dd}{\dd t}S(t) = - AS(t)   -S(t) A^T +  4\gamma
        \Sigma_2\big( {\mathfrak p^2}(t)\big),
\end{equation}
where
\begin{equation*}
  \begin{split}
\Sigma_2 ({\mathfrak p^2}) =\begin{bmatrix}
0_{n}&0_{n+1,n}\\
0_{n,n+1}&D_2(\mathfrak p^2)
                        \end{bmatrix},
                      \,\mbox{ with }
D_2(\mathfrak p^2) =
       \begin{pmatrix}
T_-& 0 & 0 &\dots&0\\
                     0&  \bbE_{\mu_n} [p_1^2] &  0 &\dots&0\\
                     0 & 0 &   \bbE_{\mu_n} [p_2^2] &\dots&0\\
                     \vdots & \vdots & \vdots & \ddots&\vdots\\
 0& 0 & 0 & \dots &  \bbE_{\mu_n} [p_n^2]
                        \end{pmatrix}.
 \end{split}
\end{equation*}
After time averaging as in \eqref{eq:bracket-t}, we obtain the following central equation: 
\begin{equation}\label{eq:SAdiff}
  A\lang S\rang_t+\lang S \rang_t \; A^T= 4\gamma\Sigma_2(\lang \mathfrak p^2\rang_t) +
 R(n^2t).
\end{equation}
where
\begin{equation}
  \label{eq:110}
   R(t) =  \frac 1{ t}{\bbE_{\mu_n}}\left[ S(0)  -S( t)  \right] =:
  \begin{pmatrix} R^{(r)}(t)& R^{(r,p)}(t)\\ 
R^{(p,r)}(t)&R^{(p)}(t)
\end{pmatrix}.
\end{equation}
\begin{remark}Note that the \emph{error term} $R(n^2t)$ can be roughly estimated from the energy bound \eqref{eq:44a} given in Theorem \ref{thm:energy}. One can easily get that: for any $n\ge 1$, any $t>0$,
\begin{equation}
  \label{eq:42}
  \sup_{\ell}\sum_x |R^{(\alpha)}_{x+\ell, x}(n^2t)| \le \frac{C }{n}\left(1+\frac1t\right),\qquad
\alpha\in \{p,r,{(p,r),(r,p)}\},
\end{equation}
where the summation in the first variable is understood in the modulo
sense.
\end{remark}

The matrix equation \eqref{eq:SAdiff} which is of the form $AX+XA^T=\Gamma$, with $A$ given by $\eqref{bA}$ and $\Gamma$ being some fixed data, can be explicitely solved using the Fourier basis \eqref{eq:vecneu} and \eqref{eq:vecdir}. The computations are quite involved but straightforward, and in order to make the presentation clearer we postpone some details to Appendix  \ref{app:matrix}.

Since $t>0$ is fixed, in order to simplify the notation we write in the following
$ R^{(\alpha)} \equiv R^{(\alpha)}(n^2t)$,  $ S^{(\alpha)} \equiv S^{(\alpha)}(n^2t)$ and
$\lang S^{(\alpha)}\rang \equiv \lang S^{(\alpha)}\rang_t$, and similarly for other quantities.
First, let us define, for $j,j'=0,\dots,n$ and $\ell,\ell'=1,\dots,n$,
\begin{align}
	\label{FT}
	\tilde S_{\ell,j}^{(r,p)}&:= \langle \phi_\ell, \lang S^{(r,p)}\rang \; \psi_{j}\rangle, \qquad \;
	\tilde S_{j,\ell}^{(p,r)}:= \langle \psi_{j}, \lang S^{(p,r)}\rang \; \phi_{\ell}\rangle ,\\
	\tilde S_{\ell,\ell'}^{(r)}&:= \langle \phi_\ell, \lang S^{(r)}\rang \; \phi_{\ell'}\rangle,
	\qquad \quad \tilde S_{j,j'}^{(p)}:= \langle \psi_j, \lang S^{(p)}\rang \; \psi_{j'}\rangle.\notag
\end{align}
Analogous definitions of  $\tilde R_{j,j'}^{(\alpha)}$ can be made
for the entries of $R$. We also let
{
\begin{align}
	\label{FT-1}
	\tilde S_{\ell,j}^{(r,p)}(n^2t)&:= \langle \phi_\ell,
                                         S^{(r,p)}(n^2 t)  \; \psi_{j}\rangle, \qquad \;
	\tilde S_{j,\ell}^{(p,r)}(n^2t):= \langle \psi_{j},  S^{(p,r)}(n^2t) \; \phi_{\ell}\rangle ,\\
	\tilde S_{\ell,\ell'}^{(r)}(n^2t)&:= \langle \phi_\ell, S^{(r)}(n^2t) \; \phi_{\ell'}\rangle,
	\qquad \quad \tilde S_{j,j'}^{(p)}(n^2t):= \langle \psi_j, S^{(p)}(n^2t) \; \psi_{j'}\rangle.\notag
\end{align}}
Finally,   let us define (recall the right hand side of \eqref{eq:SAdiff}):
\begin{equation}
	\label{tFjj}
	\tilde { F}_{j,j'}(t):=  \sum_{y=1}^n\psi_j(y)\psi_{j'}(y) p_y^2(t)
	+ \psi_j(0)\psi_{j'}(0) T_-, \qquad j,j'=0,\ldots,n. \end{equation}
The main technical result of this section gives the explicit formular of any quantity of the form $\tilde S_{k,k'}^{(\alpha)}$ as a linear combination of $\lang \tilde F_{k,k'}\rang$ and of all the $\tilde R_{k,k'}^{(\beta)}$ for $\beta \in \{p,r,{(r,p),(p,r)}\}$. Namely:

\begin{lemma}[Resolution of the covariance matrix] \label{lem:resolution}
For any $j,j'=1,\dots,n$ and any $\alpha \in \{p,r,{(r,p),(p,r)}\}$ we have
	 \begin{align}
		\tilde S_{j,j'}^{(\alpha)} = \Theta_\alpha(\la_j,\la_{j'}) \lang\tilde { F}_{j,j'}\rang +  \sum_{\beta\in \{p,r,{(r,p),(p,r)}\}} \Xi_{\beta}^{(\alpha)} (\la_{j},\la_{j'})   \tilde
		R^{(\beta)}_{j,j'},
		\label{eq:cov-t3}
	\end{align}
where the functions $\Theta_\alpha$ and $\Xi_\beta^{(\alpha)}$ are defined as follows: let us introduce the symmetric function $\theta(c,c'):= (c-c')^2+8\gamma^2
(c+c') = \theta(c',c)$, and define
\begin{align}
	\label{Xip}
	  \Theta_p(c,c') &=  \frac{8\gamma^2
		(c+c')}{\theta(c,c')},\qquad \quad   \Theta_r(c,c') = \frac{16\ga^2
		\sqrt{cc'}}{\theta(c,c')},\notag\\
	 \Theta_{p,r}(c,c') &=  \frac{4\ga\sqrt{c'}(c-c')}{\theta(c,c') },
	\qquad \Theta_{r,p}(c,c') =
	\Theta_{p,r}(c',c) , 
\end{align}
and also, when $\alpha=\beta$,
\begin{align}
	\Xi_p^{(p)}(c,c')&=\frac{2\ga(c+c')}{\theta(c,c')},\qquad
	\Xi_r^{(r)}(c,c')=  \frac{2\gamma (8\ga^2+c+c')}{\theta(c,c')},\notag\\
	\Xi^{(p,r)}_{p,r}(c,c')
&	=  \frac{4\ga c'}{\theta(c,c')},  \qquad \Xi^{(r,p)}_{r,p}(c,c')= \Xi_{p,r}^{(p,r)}(c',c), 	\label{Xip1}
\end{align}
and when $\alpha\neq \beta$,
\begin{align}
	\Xi_r^{(p)}(c,c')&= \Xi_p^{(r)}(c,c')=\frac{4\ga \sqrt{cc'}}{\theta(c,c')}, \notag\\
	\Xi_{p,r}^{(r,p)}(c,c') & = \Xi_{r,p}^{(p,r)}(c,c') = -  \frac{4\ga \sqrt{cc'}}{\theta(c,c')} \notag  \\
	 \Xi_{r,p}^{(p)}(c,c')&=- \Xi^{(r,p)}_p(c,c') =\frac{ \sqrt{c}
		(c-c')}{\theta(c,c')}, \notag \\
	\Xi_{p,r}^{(p)}(c,c')& =-\Xi^{(p,r)}_p(c,c') = \Xi_{r,p}^{(p)}(c',c), \notag\\
	\Xi^{(p,r)}_r(c,c')  &=- \Xi_{p,r}^{(r)}(c,c')=  \frac{\sqrt{c}(c-c'+8\ga^2)}{\theta(c,c')  } , \notag \\
	\Xi^{(r,p)}_r(c,c')& =-\Xi_{r,p}^{(r)}(c,c')= \Xi^{(p,r)}_r(c',c) . \label{Xip2}
\end{align}
When either $j=0$, or $j'=0$,  formula
\eqref{eq:cov-t3} is still valid. When both $j=j'=0$, then by convention,  
we let $\Theta_p(0,0)=1$,
$\Xi_p^{(p)}(0,0)=\frac{1}{4\ga}$ and
$\Theta_\alpha(0,0)=\Xi_\alpha^{(\beta)}(0,0)=0$ for all
$\alpha\neq p$, and all $\beta$.

	\end{lemma}

\subsection{Consequences of the resolution}

\label{sec:consequences}

\subsubsection{Estimates on the kinetic energy}

 \begin{proposition}
	\label{prop031012-21a}
	For any $t>0$ there exists $C=C(t)>0$ such that, for any $n \ge 1$,
	\begin{equation}
		\label{051012-21a}
		\begin{split}
			\sum_{x=0}^{n-1}\big(\lang p_x^2\rang_t-\lang p_{x+1}^2\rang_t\big)^2 \le
			\frac{C}{n}\quad \text{ and } \quad 
			\sup_{x=0,\ldots,n}\lang p_x^2\rang_t\le C.
		\end{split}
	\end{equation}
\end{proposition}

\proof The proof is divided into four steps: first, we express $\lang p_x^2 \rang_t$ thanks to the resolution \eqref{eq:cov-t3}. Then, we obtain an intermediate result which is formulated in Lemma \ref{prop031012-21aa} below. Finally we show how to prove \eqref{051012-21a}.

\noindent \textsc{Step 1 -- } From \eqref{eq:cov-t3}, and inverse Fourier transforms, we have
\begin{equation}
  \label{eq:50}
    \lang S^{(p)}_{x,x} \rang_t 
  = \sum_{y=0}^n M_{x,y}  \lang p_y^2 \rang_t
  + \big(T_--\lang p_0^2\rang_t \big) M_{x,0}+ \mathfrak{r}_{x,x}^{(p)}(t),
\end{equation}
where   
\begin{equation}
M_{x,y} :=\sum_{j,j'=0}^n\Theta_p(\la_j,\la_{j'})
      \psi_j(x)\psi_{j'}(x)\psi_j(y)\psi_{j'}(y) 
\label{eq:54-1}
\end{equation}
{
  and
  \begin{equation}
    \label{010602-23}
\mathfrak{r}_{x,x}(t):=\sum_{j,j'=0}^n \sum_{\beta\in \{r,p,{(r,p),(p,r)}\}} \Xi_{\beta}^{(p)} (\la_{j},\la_{j'})   \tilde
    R^{(\beta)}_{j,j'} \psi_j(x) \psi_{j'}(x).
\end{equation}
It has been shown in \cite[Appendix A]{klo22} that
\begin{equation}
  \label{030506-22}
  \sum_{y'=0}^nM_{x,y'}=\sum_{y'=0}^nM_{y',x}\equiv 1 \quad
  \mbox{and}\quad M_{x,y}>0\quad\mbox{for all }x,y=0,\ldots,n.
  \end{equation}
  {Recall that
    $\lang p_x^2 \rang_t=\lang S^{(p)}_{x,x} \rang_t+
    \lang \bar p_x^2\rang_t.
    $
    By Proposition \ref{prop:main} and Corollary \ref{cur}, see the boundary estimate \eqref{eq:45}, we conclude that for each $t>0$ there exists $C>0$
such that, for any $n\ge 1$,
\begin{equation}
  \label{013001-23a}
\sum_{x=0}^n \lang \bar p_x^2\rang_t\le
\frac{C}{n}\quad\mbox{and}\quad \big|T_--\lang p_0^2\rang_t \big|\le
\frac{C}{n}.
\end{equation}
From \eqref{013001-23a}
      we infer therefore that
\begin{equation}
  \label{eq:50a}
    \lang p_x^2 \rang_t
  = \sum_{y=0}^n M_{x,y}  \lang p_y^2 \rang_t
  +\rho_{x}(t)+ \mathfrak{r}_{x,x}^{(p)}(t),
\end{equation}
where
  $\rho_{x}(t)$ satisfies: for any $t>0$ there exists $C>0$ such that, for any $n\ge 1$,
\begin{equation}
  \label{010506-22}
\sup_{s\in[0,t]}\sum_{x=0}^n|\rho_{x}(s)|\le \frac{C}{n}.
\end{equation}

\noindent \textsc{Step 2 -- } We now prove an intermediate result: 

 \begin{lemma}
	\label{prop031012-21aa}
	For any $t>0$ there exists $C=C(t)>0$ such that, for any $n\ge 1$,
	\begin{equation}
		\label{051012-21}
		\sum_{x=0}^{n-1}\big(\lang p_x^2\rang_t-\lang p_{x+1}^2\rang_t\big)^2 \le
		\frac{C}{n} {\sup_{x=0,\dots,n}}  \lang p_x^2\rang_t.
	\end{equation}
\end{lemma}

\proof[Proof of Lemma \ref{prop031012-21aa}]
The following lower bound on the matrix $[M_{x,y}]$ comes from
\cite[Proposition 7.1]{klo22} (see also \cite{bll}): there exists $c_*>0$ such that,
for any $n\ge 1$, and any $(f_x)\in\bbR^{n+1}$,
\begin{align}
\label{lowerM}
\sum_{x,y=0}^n(\delta_{x,y}-M_{x,y})f_yf_x 
\ge c_*\sum_{x=0}^{n-1} (\nabla f_x)^2.
\end{align}
Multiplying both sides of \eqref{eq:50a} by $\lang p_x^2\rang_t$,
summing over $x$ and using  \eqref{lowerM} {together
with estimate \eqref{010506-22}}
we
immediately obtain: for any $t>0$ there is $C>0$ such that
\begin{align}
	\label{051012-21bb}
	\sum_{x=0}^{n-1}\big(\lang p_x^2\rang_t-\lang p_{x+1}^2\rang_t\big)^2   & \le
	C\sum_{x=0}^n |\rho_{x}(t)|\lang p_x^2\rang_t +C{\mathfrak F}_n(t)
	\notag 
\end{align}where 
\[ {\mathfrak F}_n(t):= \sum_{x=0}^n\mathfrak{r}_{x,x}^{(p)}(t)\lang p_x^2\rang_t.\]
We now claim that there is $C=C(t)>0$ such that 
\begin{align}
	{\mathfrak F}_n(t) &  \le \frac{C}{n}\sup_{x=0,\dots,n}\lang p_x^2\rang_t. \label{eq:claim}
\end{align} 
From last estimate and \eqref{010506-22}, we immediately conclude the proof of Proposition \ref{prop031012-21a}.
Let us prove the claim.  Recall \eqref{eq:110}. We can rewrite ${\mathfrak F}_n(t)$ as  
\begin{align}
	{\mathfrak F}_n(t)
	={\mathfrak F}_n (0,t)
  -  {\mathfrak F}_n(n^2t,t), \label{eq:fnt}
\end{align} where (see \eqref{FT-1})
\begin{align*}
	{ \mathfrak F}_n(s,t)&:= \frac{1}{n^2t}\sum_{j,j'=0}^n \sum_{\beta\in \{r,p,{(r,p),(p,r)}\}}\Xi_{\beta}^{(p)} (\la_{j},\la_{j'}) \tilde
	S^{(\beta)}_{j,j'} (s) a_{j,j'}(\lang p^2\rang_t),  \qquad s\ge0 \\ 
	a_{j,j'}(\lang p^2\rang_t)& :=\sum_{x=0}^n \psi_j(x) \psi_{j'}(x) \lang p_x^2\rang_t .
\end{align*}
We  show in Appendix that the matrix $\big[\sum_{\beta}\Xi_{\beta}^{(p)} (\la_{j},\la_{j'}) \tilde
S^{(\beta)}_{j,j'}(s)\big]_{j,j'}$  is non-negative definite (see Lemma
\ref{lm012709-23}, \eqref{eq:cov-t3z}), since it is the Fourier image of a covariance matrix.
Therefore we can use the following estimate which can be found \emph{e.g.}~in \cite[H1g, p.~340]{marshall}. 
\begin{lemma}
	\label{lm-lin-al}
	Suppose that $A$ and $B$ are two symmetric $m\times m$ matrices 
        and 
        $A$ is assumed to be  non-negative definite. Then
	\begin{equation}
		\label{042709-23a}
		| {\rm Tr}(AB)|\le  {\rm Tr}(A)\sup_{\|\xi\|=1}|B\xi\cdot\xi|.
	\end{equation}
\end{lemma}

By \eqref{042709-23a} we have
\begin{align}   
	{ \mathfrak F}_n (s,t)&= \frac{1}{n^2 t} {\rm Tr}\Big(\sum_{\beta}\big[\Xi_{\beta}^{(p)} (\la_{j},\la_{j'}) \tilde
	S^{(\beta)}_{j,j'} (s)\big] \big[a_{\ell,\ell'}(\lang
	p^2\rang_t)\big]\Big) \notag\\
	&
	\le \frac{1}{n^2 t}{\rm Tr}\Big(\sum_{\beta}\big[\Xi_{\beta}^{(p)} (\la_{j},\la_{j'}) \tilde
	S^{(\beta)}_{j,j'} (s) \big]\Big)\sup_{\|\xi\|=1}\big|\big[a_{j,j'}(\lang
	p^2\rang_t)\big]\xi\cdot\xi\big|. \label{eq:key0}
\end{align}
First, note that
\begin{align*}
	\sup_{\|\xi\|=1}\big[a_{j,j'}(\lang
	p^2\rang_t)\big]\xi\cdot\xi=\sup_{\|\xi\|=1} \sum_{x=0}^n\sum_{j,j'=0}^n
	\xi_j\xi_{j'}\psi_j(x) \psi_{j'}(x) \lang p_x^2\rang_t
	=\sup_{\|\hat\xi\|=1} \sum_{x=0}^n 
	|\hat\xi_x|^2  \lang p_x^2\rang_t ,
\end{align*}
with 
$
\hat \xi_x:=\sum_{j=0}^n\psi_j(x)\xi_j.
$
Therefore we have
\begin{equation}\label{eq:key1}
	\sup_{\|\xi\|=1}\big[a_{j,j'}(\lang
	p^2\rang_t)\big]\xi\cdot \xi \le   \sup_{x=0,\ldots,n}
	\lang p_x^2\rang_t .
      \end{equation}
      Since $\Xi_{\beta}^{(p)}=\frac{1}{4\ga}\Theta_\beta$,
      $\beta=p,r$ and $\Xi_{\beta}^{(p)}=-\frac{1}{4\ga}\Theta_\beta$, $\beta={(p,r),(r,p)}$, by  Lemma \ref{lm012709-23}   we get
\begin{align}
	& {\rm Tr}\Big(\sum_{\beta}\big[\Xi_{\beta}^{(p)}(\la_j,\la_{j'}) \tilde
	S^{(\beta)}_{j,j'} (s)\big]\Big)\le 2\Big(\sum_{j=0}^n\tilde
	S^{(p)}_{j,j} (s) +\sum_{j=1}^n\tilde
	S^{(r)}_{j,j} (s) \Big) \le\,  4 \bbE_{\mu_n}\big[{\cal H}_n(s)\big]. \label{eq:key2}
\end{align}
Combining \eqref{eq:key0} with \eqref{eq:key1} and \eqref{eq:key2}  we get
\begin{align*}
	{ \mathfrak F}_n(s,t)
	&
	\le \frac{4}{n^2 t} \, \bbE_{\mu_n}\big[{\cal H}_n(s)\big]   \sup_{x=0,\ldots,n}
	\lang p_x^2\rang_t.
\end{align*} From \eqref{eq:fnt} and using the energy bound \eqref{eq:44a},
 the claim \eqref{eq:claim} follows.

\medskip

\noindent \textsc{Step 3 --} 
Using the triangle and then Cauchy-Schwarz inequalities we conclude
\begin{align}
	\label{051012-21bd}
	\sup_{x=0,\dots,n}\lang
	p_x^2\rang_t & \le \lang
	p_0^2\rang_t +\sum_{x=0}^{n-1}\Big|\lang p_x^2\rang_t-\lang p_{x+1}^2\rang_t\Big|\notag\\
	&
	\le  \lang
	p_0^2\rang_t +\sqrt{n}\left\{\sum_{x=0}^{n-1}\Big(\lang p_x^2\rang_t-\lang p_{x+1}^2\rang_t\Big)^2 \right\}^{1/2}  .
\end{align}
Denote
$
D_n:=\sum_{x=0}^{n-1}\big(\lang p_x^2\rang_t-\lang p_{x+1}^2\rang_t\big)^2.
$
We can summarize \eqref{051012-21} and \eqref{051012-21bd} as follows: for any $t>0$  there exists
$C>0$ such that
\begin{align}
	\label{051012-21bddd}
	D_n&\le  \frac{C}{n}  \sup_{x=0,\dots,n}\lang
	p_x^2\rang_t, \notag \\
	\sup_{x=0,\dots,n}\lang 
	p_x^2\rang_t &\le \lang 
	p_0^2\rang_t +\sqrt{n}D_n^{1/2} \le\lang
	p_0^2\rang_t +C+C \Big(\sup_{x=0,\dots,n}\lang
	p_x^2\rang_t\Big)^{1/2},
\end{align}
for all $n\ge 1$.
Thus the second estimate of \eqref{051012-21a} follows, which in turn
implies the first estimate of
\eqref{051012-21a}  as well. \qed

\medskip 

We end this section with two important corollaries.

\begin{corollary}
  \label{cor012504-24}
  For any $t>0$ there exists $C=C(t)>0$ such that, for any $n\geqslant 1$,
   \begin{equation}
     \label{eq:49a}
     \sup_{x\neq 0} |\lang p_xp_0 \rang_t| \le \frac{C}{n^{1/2}}, \qquad 
  \sum_{x=1}^n |\lang r_xp_0 \rang_t| \le
  Cn^{1/2}.
\end{equation}

 \end{corollary}
 \proof
 We only prove the estimate for the momenta average, as the argument for
 the mixed momenta-stretch average is similar. Note that for $x\not=0$,
 integrating by parts in $p_0$ we get
 \begin{align}
   \label{012604-24}
   \int_0^t\EE_{\mu_n} \big[(p_xp_0)(n^2 s)\big]\dd s&=-T_-\int_0^t\dd     s\int_{\Om_n}p_x
     f_n(n^2s;\br,\bp)\partial_{p_0}g_{T_-}(\br,\bp)\dd\br \dd\bp\\
   &
     =T_-\int_0^t\dd     s\int_{\Om_n}p_x
     \partial_{p_0}f_n(n^2s )\dd \nu_{T_-} \notag \\ & =T_-\int_0^t\dd     s\int_{\Om_n}p_x f_n^{1/2}(n^2s)
     \frac{\partial_{p_0}f_n(n^2s )}{f_n^{1/2}(n^2s)}\dd\nu_{T_-} .\notag
 \end{align}
 The absolute value of the utmost right hand side can be estimated using the Cauchy-Schwarz inequality
 by
 \begin{align}
   \label{012504-24}
           T_-\bigg\{\int_0^t\bbE_{\mu_n}[p_x^2(n^2s)]\dd s\bigg\}^{1/2}
     \bigg\{\int_0^t\dd s\int_{\Om_n}\frac{\big[\partial_{p_0}f_n(n^2s)\big]^2}{f_n(n^2s)}\dd\nu_{T_-}
   \bigg\}^{1/2}\le \frac{C}{n^{1/2}},
 \end{align}
 by virtue of \eqref{eq:49} and  \eqref{051012-21a}. Then the first estimate
 of \eqref{eq:49a} follows. The argument for the second estimate of \eqref{eq:49a} is similar, except that in the first part of \eqref{012504-24} we need to keep the summation over $x$ and then use the energy bound \eqref{eq:44a}.
 \qed

 \medskip

 Finally, we give a last consequence of the kinetic energy estimate which will be used
 in the forthcoming section.

 Using the definition of $\tilde F_{j,j'}(t)$   (see \eqref{tFjj}) and  the trigonometric identity
\begin{equation*}
  \begin{split}
\psi_j(y)\psi_{j'}(y)=\frac{1}{n+1}
\Big(1-\frac{\delta_{j',0}}{2}\Big)^{1/2}\Big(1-\frac{\delta_{j,0}}{2}\Big)^{1/2}
\bigg[&\cos\Big(\frac{\pi  (j+j')(2y+1)}{2(n\!+\!1)}\Big)\\
&+\cos\Big(\frac{\pi (j- j')(2y+1)}{2(n\!+\!1)}\Big)\bigg]
\end{split}
\end{equation*}
we can write, 
\begin{equation}
\lang\tilde { F}_{j,j'}  \rang_t=
\Big(1-\frac{\delta_{j',0}}{2}\Big)^{1/2}
\Big(1-\frac{\delta_{j,0}}{2}\Big)^{1/2} \Big[ \widehat{{\mathfrak F}}(j-j')
+ \widehat{{\mathfrak F}}(j+j') \Big].\label{eq:52}
\end{equation}
where, for $j=-2(n+1), \dots, 2n+1$,
  \begin{equation}
                       \label{012304-24}
                       \widehat{{\mathfrak F}}(j) :=
                       \frac{1}{n\!+\!1 }\sum_{y=1}^n\cos\Big(\frac{\pi j (2y+1)}{2(n\!+\!1)}\Big) 
                   \lang p_y^2\rang_t +\frac{T_-}{n\!+\!1 }\cos\Big(\frac{\pi j }{2(n\!+\!1)}\Big) .
                 \end{equation}
                 By another elementary trigonometric identity we have also
               \begin{equation}
                       \label{eq:Fj}
                       \begin{split}
                       \widehat{{\mathfrak F}}(j)&=
                       \frac{1}{2(n\!+\!1) \sin\big(\frac{\pi j}{2(n+1)}\big)}
                       \sum_{y=1}^{n}\sin\Big(\frac{\pi jy}{n+1}\Big)
                        \nabla^\star\lang p_y^2\rang_t
                        +\frac{T_-}{n\!+\!1 }\cos\Big(\frac{\pi j }{2(n\!+\!1)}\Big) \\
                       & =   \frac{1}{\sqrt{2^3(n\!+\!1)} \sin\big(\frac{\pi j}{2(n+1)}\big)}
                       \sum_{y=1}^{n}\phi_j(y)
                        \nabla^\star\lang p_y^2\rang_t
                        +\frac{T_-}{n\!+\!1 }\cos\Big(\frac{\pi j }{2(n\!+\!1)}\Big).
                      \end{split}
                 \end{equation}

\begin{corollary}
  \label{cor012304-24}
  For any $t_*>0$ there exists $C>0$ such that
  \begin{equation}
                       \label{022304-24}
                     \sum_{j=-2n-2}^{2n+1}| \widehat{{\mathfrak F}}(j) |\le C,\qquad    t\in[0,t_*].
                                               \end{equation}
  \end{corollary}
  \proof 
{Since 
  $\widehat{{\mathfrak F}} (2n+2-j) =- \widehat{{\mathfrak F}} (-j)$ and
  $\widehat{{\mathfrak F}} (j-2n-2) =- \widehat{{\mathfrak F}} (j)$,
  we have
                     \[
                       \sum_{j=-2n-2}^{2n+1}| \widehat{{\mathfrak F}} (j)|=
                       2\sum_{j=-n-1}^{n}| \widehat{{\mathfrak F}} (j)|.
                     \]
                       Since
                       \[                       \bigg|\frac{j}{2(n\!+\!1) \sin\Big(\frac{\pi  j}{2(n+1)}\Big)}
                       \bigg|\le \frac{1}{2},\qquad j=-n,\ldots,n\!+\!1,
                       \]
                       and  defining
                     \begin{equation}
                       \label{eq:51}
                       P_{n,j}:=\frac{1}{\sqrt{2^3(n\!+\!1)} \sin\Big(\frac{\pi j}{2(n+1)}\Big)}
                       \sum_{y=1}^{n}\phi_j(y)
                        \nabla^\star\lang p_y^2\rang_t,
                     \end{equation}
                       we have
                       \begin{align*}
                         &\left|j P_{n,j} \right|
                           \le \frac{(n+1)^{1/2}}{2^{3/2}}\Big|\sum_{y=1}^{n}\phi_j(y)
                        \big(\nabla^\star\lang  p_y^2\rang_t\big)\Big|  .
                       \end{align*}
                       In consequence, using the Plancherel identity
                       \begin{align*}
                         \sum_{j=-n}^{n+1}j^2 P_{n,j}^2 \le 
                            \frac{n\!+\!1}{8}\sum_{y=1}^{n}\big[\nabla^\star\lang
                           p_y^2\rang_t\big]^2\le C,
                       \end{align*}
                       by virtue of \eqref{051012-21a}.
                       Hence, by  
                       Cauchy-Schwarz inequality
                       \begin{align*}
                         \sum_{j=-n-1}^{n}| \widehat{{\mathfrak F}}(j)|&\le 2 T_-
                   + \sum_{j=-n-1}^{n}|  P_{n,j}|\\
                         &
                           \le 2T_-+
                           2^{1/2}\bigg\{\sum_{j=-n-1}^{n}j^2 P_{n,j}^2\bigg\}^{1/2}
                           \bigg\{\sum_{j=0}^{n}\frac{1}{(j+1)^2} \bigg\}^{1/2}\le C'.
\end{align*}                 }\qed

\subsubsection{Consequences of Assumption \ref{A10}}
 \label{ssec:cons}

 Recall that the Fourier transforms of the fluctuation fields read as
\[
\tilde p_j'(t)=\sum_{x=0}^n\psi_j(x)p_x'(t)\quad \mbox{and}\quad \tilde r_j'(t)=\sum_{x=1}^n\phi_j(x)r_x'(t).
\]
Recall also  the position functional which has been defined in \eqref{eq:64}.
	We have the relation
	\begin{equation}
		\label{eq:61}
		\tilde q'_j (n^2s)= \lambda_j^{-1/2}\; \tilde r'_j (n^2s).
	\end{equation}
	Therefore, by \eqref{eq:61} there exist $c,C>0$, absolute constants, such that
		\begin{equation}
			\label{eq:62}
			c n^2 \sum_{j=1}^n \frac{\lang(\tilde r'_j)^2\rang_t}{j^2}  \le
			\sum_{x=0}^n \lang(q'_x)^2\rang_t = \sum_{j=1}^n \frac{\lang(\tilde r'_j)^2\rang_t}{\lambda_j} \le
			C n^2 \sum_{j=1}^n \frac{\lang(\tilde r'_j)^2\rang_t}{j^2}.
	\end{equation} 
We now prove the following technical bound, which will imply the extension of Assumption \ref{A10} at positive time: 
\begin{proposition}
  \label{prop010510-23}
       
       For any $t_*>0$ there exists a constant $C=C(t_*)>0$ such that for any $n\ge 1$, for any $j=1,\dots,n$,
  \begin{equation}
    \label{020510-23}
          \frac{\bbE_{\mu_n}[(\tilde r_j'(n^2
      t))^2] }{j^2} 
                                   \le    C\bigg(1+
               \frac{\bbE_{\mu_n}[(\tilde r_j'(0))^2] }{j^2}
         \bigg),\quad t\in[0,t_*].
       \end{equation}
        \end{proposition}
  \proof
Recall the definition of $\tilde F_{j,j'}$ in \eqref{tFjj}.  From \eqref{eq:cov-t3}, and several easy computations when the functions in \eqref{Xip}--\eqref{Xip2} are evaluated at $c=c'$: for instance we have $\Theta_p(\la_j,\la_j)=1$ and  \[  \Xi_p^{(p)}(\la_j,\la_j)=\Xi_r^{(p)}(\la_j,\la_j)=\frac1{4\ga},  \quad \text{and} \quad  \Xi^{(p)}_{r,p}(\la_j,\la_j)=\Xi^{(p)}_{p,r}(\la_j,\la_j)=0,\]  we get the following system (writing for short $\bE$ for $\bE_{\mu_n}$): for any $j=1,\dots,n$,
    \begin{align}
    \label{030510-23}
    &
       \bbE[\tilde p_j'(n^2 t)^2]=\tilde
      F_j(t) 
      -\frac{1}{4\ga n^2}\frac{\dd}{\dd t}\bbE[\tilde p_j'(n^2 t)^2]
      -\frac{1}{4\ga n^2}\frac{\dd}{\dd t}\bbE[\tilde r_j'(n^2t)^2],\notag\\
  &
       \bbE[\tilde r_j'(n^2 t)^2]=\tilde F_j(t) 
      -\frac{1}{4\ga n^2}\frac{\dd}{\dd t}\bbE[\tilde p_j'(n^2 t)^2]
       -\frac{1}{ n^2}\Big(\frac{1}{4\ga}+\frac{\ga}{\la_j} \Big)\frac{\dd}{\dd t}\bbE[\tilde r_j'(n^2 t)^2]\notag\\
    & \qquad \qquad \qquad
      +\frac{1}{n^2\sqrt{\la_j} }\frac{\dd}{\dd t}\bbE[\tilde p_j'(n^2 t) \tilde r_j'(n^2 t)],\\
    &
      \bbE[\tilde p_j'(n^2 t) \tilde r_j'(n^2 t)]=-\frac{1}{2n^2\sqrt{\la_j}}
        \frac{\dd}{\dd t}\bbE[\tilde r_j'(n^2 t)^2].\notag
  \end{align}
 Let us shorten notation and use the energy bound \eqref{eq:44a}, and also the fact that $\psi_j^2(y) \le C/n$, in order to write the following: for any $t\ge 0$,
  \begin{align}
    \label{050510-23} 
    \tilde F_j(t):= \bbE[\tilde F_{j,j}(n^2t)] =\sum_{y=1}^n\psi_j^2(y) \bbE[ p_y^2(n^2t)]
    + \psi_j^2(0) T_-
    \le C(t+1).
  \end{align}
  Summing the two equations of \eqref{030510-23} sideways we get: for $j=1,\dots,n$,
  \begin{multline}
  \label{030510-23aa}
\frac{1}{4\ga} \frac{\dd}{\dd t} \bigg\{ \bbE[(\tilde p_j'(n^2 t))^2]
      +2\bbE[(\tilde r_j'(n^2
      t))^2]\bigg\}+\frac{\dd}{\dd t}  \bbE\bigg[\bigg(\frac{1}{2\sqrt \ga} \tilde p_j'(n^2t) + \Big(\frac{\ga }{\la_j}\Big)^{1/2} \tilde r_j'(n^2
      t)\bigg)^2\bigg]\\
      = -n^2\Big( \bbE[(\tilde p_j'(n^2 t))^2] +\bbE[(\tilde r_j'(n^2 t))^2]\Big)+ 2n^2 \tilde  F_j(t). \vphantom{\int_0^1}
  \end{multline}
       Let us turn to the proof of \eqref{020510-23}. Starting from \eqref{030510-23aa}, we first integrate in time and then drop the first term on the left hand side (bounded from below by $0$). This yields: 
         \begin{multline}
    \label{030510-23c}
\bbE\bigg[  \bigg(\frac{1}{2\sqrt \ga} \tilde p_j'(n^2t)+ \Big(\frac{\ga }{\la_j}\Big)^{1/2} \tilde r_j'(n^2
  t)\bigg)^2\bigg]
         \le  -n^2\int_0^t \bbE[(\tilde p_j'(n^2 s))^2] +\bbE[(\tilde
      r_j'(n^2 s))^2]\dd s\\
             +\so{2} n^2\int_0^t  \tilde F_j(s)\dd s  
            +G_j,
         \end{multline}
      where 
     \begin{align*}   G_j:= \bbE\bigg[\bigg(\frac{1}{2\sqrt \ga} \tilde p_j'(0) + \Big(\frac{\ga }{\la_j}\Big)^{1/2} \tilde r_j'(0)\bigg)^2\bigg] +  \frac1{4\ga}   \bbE[(\tilde p_j'(
     	0))^2]+\frac{1}{2\ga} \bbE[(\tilde r_j'(0))^2].
     \end{align*}
         Below we sum and subtract $\frac{1}{2\sqrt \ga} \tilde p_j'(n^2t)$ and use the standard inequality $(a+b)^2\le 2a^2+2b^2$, and we get
         \begin{align*}
            \frac{\ga}{\la_j}\bbE[(\tilde r_j'(n^2
      t))^2] & \le \frac1{2\ga} \bbE[(\tilde p_j'(n^2
      t))^2]+ 2 \bbE\bigg[\bigg(\frac{1}{2\sqrt \ga} \tilde p_j'(n^2t) + \Big(\frac{\ga }{\la_j}\Big)^{1/2} \tilde r_j'(n^2
      t)\bigg)^2\bigg]\\
                                &
             \le    \frac{1}{2\ga}\bbE[(\tilde p_j'(n^2
             t))^2] 
             +n^2\int_0^t\tilde
      F_j(s)\dd s  
            + G_j,
         \end{align*} from \eqref{030510-23c}.
         Using  the Cauchy-Schwarz inequality  and estimate
         \eqref{050510-23}  to bound the first and second term on the
         right hand side respectively
               and the fact that $C\frac{j^2}{n} \leqslant \lambda_j \leqslant c\frac{j^2}{n}$ for universal constants, we
         conclude that 
       \begin{align*}
            c n^2 \frac{\bbE[(\tilde r_j'(n^2
      t))^2] }{j^2} 
             &\le    C\sum_{x=0}^n\bbE[(p_x'(n^2
      t))^2]       +Cn^2  \\ & \quad 
             + C n^2\frac{\bbE[(\tilde r_j'(0))^2] }{j^2} 
             + C \sum_{x=0}^n\bbE\big[(p_x'(0))^2  +(r_x'(0))^2 \big] .
       \end{align*}
   Dividing both sides by $n^2$ and using  the energy bound
   \eqref{eq:44a} we conclude  \eqref{020510-23}.
       \qed

We now give some asymptotics of the functionals involving positions.
   
       \begin{corollary}
         \label{prop012204-24}
         Under Assumption \ref{A10} for any $t>0$ we have
         \begin{equation}
           \label{012204-24}
           \lim_{n\to+\infty}    \frac 1{n^3} \sum_{x=0}^n \mathbb
           E_{\mu_n}\left[ (q'_x)^2(n^2 t)\right]=0.
         \end{equation}
         \end{corollary}
       
       \begin{proof}
  From \eqref{eq:62} and  \eqref{020510-23}, for any $\kappa\in(0,1)$ we
       can write
       \begin{equation}
         \label{eq:63}
         \begin{split}
         \frac 1{n^3} \sum_{x=0}^n \lang(q'_x)^2\rang_t & \le
         C \frac 1n \sum_{j=1}^n \frac{\lang(\tilde r'_j)^2\rang_t}{j^2}
         \le  C\frac 1n  \sum_{j=1}^{n^\kappa} \frac{\lang(\tilde r'_j)^2\rang_t}{j^2}
         + C\frac{1}{n^{1+2\kappa}} \sum_{j= n^\kappa +1}^n \lang(\tilde r'_j)^2\rang_t\\
         &\le C n^{\kappa -1} + C \frac 1n \sum_{j=1}^{n^\kappa} \frac{\bbE_{\mu_n}[(\tilde r'_j(0))^2]}{j^2}
           +  \frac{C}{n^{2\kappa}}.
       \end{split}
     \end{equation}
    By Assumption  \ref{A10} and \eqref{eq:62},
    we conclude \eqref{012204-24}.
  \end{proof}

    \begin{proposition}
        For any $\varphi\in {\cal C}^2[0,1]$ such that
         ${\rm supp}\,\varphi\subset(0,1)$  we have
         \begin{equation}
           \label{022204-24}
           \lim_{n\to+\infty}    \frac 1{n} \sum_{x=0}^n \varphi_x\lang q'_xp_0'\rang_t=0.
         \end{equation}
    \end{proposition}

\begin{proof}
Define 
  \begin{equation}
  \label{fj}
 \widehat{\varphi^{(o)}}(j):=\frac{1}{n}\sum_{x=0}^n\varphi_x\sin\Big(\frac{\pi
   jx}{n\!+\!1}\Big),\quad \widehat{\varphi^{(e)}}(j):=\frac{1}{n}\sum_{x=0}^n\varphi_x\cos\Big(\frac{\pi
   jx}{n\!+\!1}\Big).
\end{equation}
 Since $\varphi\in {\cal C}^2[0,1]$ and ${\rm supp}\,\varphi\subset(0,1)$, by \cite[Lemma B.1]{klo22} there
exists $C>0$ such that
\begin{equation}
  \label{010310-23}
  |\widehat{\varphi^{(\iota)}}(j)|\le \frac{C}{\chi_n^2(j)},\quad j\in\bbZ,\,\iota=\{e,o\},
  \end{equation}
where $\chi_n$ is a $2(n\!+\!1)$-periodic extension of the function
  $\chi_n(j)= (1+j)\wedge
  (2n+3-j)$, for $j=0,\ldots, 2n\!+\!1$.
We can write
\begin{equation}
  \label{042204-24}
  \frac 1{n} \sum_{x=0}^n \varphi_x\lang
    q'_xp_0'\rang_t
  =\frac{2}{n}\sum_{j=1}^n\sum_{j'=0}^n\Big(1-\frac{\delta_{j',0}}{2}\Big)^{1/2}\cos\Big(\frac{\pi
  j'}{2(n\!+\!1)}\Big)\la_j^{-1/2}\tilde S_{j,j'}^{(r,p)}(\la_j,\la_{j'}) \widehat{\varphi^{(o)}}(j).
  \end{equation}
  Using the decomposition \eqref{eq:cov-t3} we can rewrite
  \begin{equation*}
    \frac 1{n} \sum_{x=0}^n \varphi_x\lang q'_xp_0'\rang_t
    = {\rm I}_n +\sum_{\beta \in \{r,p,{(r,p),(p,r)}\}} {\rm I}_{n,\beta},
  \end{equation*} 
where
\begin{align*}
    {\rm I}_n&= \frac{2}{n}\sum_{j=1}^n\sum_{j'=0}^n\Big(1-\frac{\delta_{j',0}}{2}\Big)^{1/2}\cos\Big(\frac{\pi
  j'}{2(n\!+\!1)}\Big)\la_j^{-1/2}\Theta_{r,p}(\la_j,\la_{j'}) \lang\tilde
    { F}_{j,j'}\rang_t \widehat{\varphi^{(o)}}(j),\\
     {\rm I}_{n,\beta}&= \frac{2}{n}\sum_{j=1}^n\sum_{j'=0}^n\Big(1-\frac{\delta_{j',0}}{2}\Big)^{1/2}\cos\Big(\frac{\pi
  j'}{2(n\!+\!1)}\Big)\la_j^{-1/2}\Xi_{\beta}^{(r,p)} (\la_{j},\la_{j'})   \tilde
		R^{(\beta)}_{j,j'}\widehat{\varphi^{(o)}}(j).
\end{align*}
  We can write, recalling \eqref{eq:52} 
  and
the definition of $\Theta_{r,p}$ (see \eqref{Xip})
\begin{align*}
  {\rm I}_n= &\;\frac{4}{n}\sum_{j=1}^n \widehat{\varphi^{(o)}}(j)
               \sum_{j'=0}^n\Big(1-\frac{\delta_{j',0}}{2}\Big)
               \Big(1-\frac{\delta_{j,0}}{2}\Big)^{1/2}\frac{4\ga(\la_{j'}-\la_{j})}{(\la_j-\la_{j'})^2
               +8\ga^2 (\la_j+\la_{j'})}\\
  & \quad
  \times   \cos\Big(\frac{\pi
  j'}{2(n\!+\!1)}\Big)\Big[\widehat{{\mathfrak F}}(j-j') + \widehat{{\mathfrak F}}(j+j') \Big].
\end{align*}
Hence
\begin{align*}
  |{\rm I}_n|\le  \frac{C}{n}\sum_{j=1}^n |\widehat{\varphi^{(o)}}(j)|
  \sum_{j'=-n-1}^n\Big[|\widehat{{\mathfrak F}}(j-j')|+|\widehat{{\mathfrak F}}(j+j') |\Big]\le \frac{C'}{n},
\end{align*}
by virtue of \eqref{022304-24} and \eqref{010310-23}.

Since the estimates for terms ${\rm I}_{n,\beta}$ are similar to each
other, we demonstrate them for ${\rm I}_{n,r}$.
Then,
\begin{align*}
     {\rm I}_{n,r}= &\;\frac{2}{n^3t}\sum_{j=1}^n \widehat{\varphi^{(o)}}(j)\sum_{j'=0}^n\Big(1-\frac{\delta_{j',0}}{2}\Big)^{1/2}\frac{\la_{j'}^{1/2} (\la_{j'}-\la_{j}+8\ga^2)}{\la_j^{1/2}[ (\la_j-\la_{j'})^2+8\ga^2
                 (\la_j+\la_{j'})]}\\
  &\qquad
    \times \cos\Big(\frac{\pi
  j'}{2(n\!+\!1)}\Big) \Big(\bbE[(\tilde r_j' \tilde r_{j'}' )(0)
   ]- \bbE[(\tilde r_j' \tilde r_{j'}' )(n^2t)
   ]\Big).
\end{align*}
Since $(\pi j/n)^2\ge \la_{j}\ge 2(j/n)^2$ for $j=1,\ldots,n$, there
exists $C>0$ such that
\[
\frac{\la_{j'}^{1/2} |\la_{j'}-\la_{j}+8\ga^2|}{\la_j^{1/2}[ (\la_j-\la_{j'})^2+8\ga^2
                 (\la_j+\la_{j'})]}\le \frac{Cn^2j'}{j[j^2+(j')^2]}
               \]
               for $j,j'=1,\ldots,n$. This in turn
 leads to an estimate
\begin{multline*}
                 |{\rm I}_{n,r}|
                   \le  \;\frac{C}{nt}\sum_{j=1}^n
      \frac{|\widehat{\varphi^{(o)}}(j)|}{jj'}\sum_{j'=1}^n \Big[\big\{\bbE[(\tilde r_j'(
      0))^2]\big\}^{1/2} \big\{\bbE[(\tilde r_{j'}'(
      0))^2]\big\}^{1/2} \\
    +\big\{\bbE[(\tilde r_j'(
      n^2t))^2]\big\}^{1/2} \big\{\bbE[(\tilde r_{j'}'(
      n^2t))^2]\big\}^{1/2} \Big]. 
  \end{multline*}
 By the Cauchy-Schwarz inequality
\begin{multline}
  \label{072304-24}
 \frac 1n \sum_{j,j'=1}^n
      \frac{|\widehat{\varphi^{(o)}}(j)|}{jj'}\big\{\bbE[(\tilde r_j'(
      n^2s))^2]\big\}^{1/2} \big\{\bbE[(\tilde r_{j'}'(
      n^2s))^2]\big\}^{1/2} \\
      \le \frac 1n \bigg\{\sum_{j,j'=1}^n
      \frac{\bbE[(\tilde r_j'(
      n^2s))^2]}{j^2(j')^2}\bigg\}^{1/2}\bigg\{\sum_{j,j'=1}^n 
      (\widehat{\varphi^{(o)}})^2(j)\bbE[(\tilde r_{j'}'(
      n^2s))^2]\bigg\}^{1/2}
   \end{multline}
 for any $s\in[0,t]$.
Using    \eqref{010310-23}
    we can estimate the right hand side by  
      \[
            C \bigg\{\frac 1n\sum_{j=1}^n
       \frac{\bbE[(\tilde r_j'(
         n^2s))^2]}{j^2}\bigg\}^{1/2}
       \bigg\{\frac 1n\sum_{j'=1}^n 
       \bbE[(\tilde r_{j'}'( n^2s))^2]\bigg\}^{1/2}\longrightarrow0,
       \]
      as $n\to+\infty$,
  by \eqref{eq:63}.
 This ends the proof of  \eqref{022204-24}.
\end{proof}

 \section{Proof of equipartition and local equilibrium}
\label{sec:proof-equipartition}

In this section we prove both Theorem \ref{thm-loc-equi} and Proposition \ref{cor010910-23},
  which were used in Section \ref{sec:proof-macr-energy} in order to identify the limit.
Let $\varphi\in \cC^2[0,1]$  be  
        compactly supported in $(0,1)$. Recall that we have introduced the notation
        $\varphi_x:=\varphi(\frac{x}{n})$.

        \subsection{Proof of Proposition \ref{cor010910-23}}

We prove the equipartition result. Similarly to \eqref{eq:64} the position functional (not recentered) is given by
  \begin{equation}
    \label{position}
    q_x = \sum_{y=1}^x r_y,\quad x=1,\ldots n, \qquad \text{and} \qquad q_0=0.
    \end{equation} 
    We have
\begin{equation*}
  \begin{split}
   \mathcal G_t \bigg(\frac 1n \sum_{x=0}^n \varphi_x \left(p_x q_x
     + \gamma q_x^2\right)\bigg) = &\frac 1n \sum_{x=0}^n \varphi_x
 \Big(  p_x^2 + q_x \nabla r_x - p_xp_0 {-2\ga q_xp_0}\Big),
     \end{split}
   \end{equation*}
{where $\mathcal G_t$ has been defined in \eqref{eq:7}.}
   As in the proof of Lemma \ref{lem:fluc}, see \eqref{eq:bar}, a similar identity can be written
   for the time evolution of the averages, namely the same above sums
   but with $\bar p_x$, $\bar  q_x$, and $\bar r_x$ instead.
     Therefore, we can write the following relation for the recentered quantities (namely $q'_x, r'_x, p'_x$,
     recall \eqref{pqpp2}), which is obtained after integrating in time:
\begin{multline}
  \frac{1}{n} \sum_{x=0}^n \varphi_x \big[ \lang (p'_x)^2 \rang_t
  +  \lang q'_x \nabla r'_x \rang_t - \lang p'_x p'_0\rang_t
  - 2 \ga \lang q'_x p'_0\rang_t \big]\\
  = \frac{1}{n^3 t} \sum_{x=0}^n \varphi_x \bE_{\mu_n} \big[\big(p'_xq'_x+\ga (q'_x)^2\big)(n^2t)
  -\big(p'_xq'_x+\ga (q'_x)^2\big)(0)\big],
\end{multline}
 We now use the identity $\nabla(q'_xr'_x)=(r'_{x+1})^2+ q'_x\nabla r'_x$,
valid for $x=1,\dots,n-1$, and summing by parts we obtain
\begin{equation*}
  \begin{split}
    \frac 1 n \sum_{x=0}^n \varphi_x \left[\lang (p'_x)^2\rang_t - \lang (r'_{x+1})^2\rang_t\right] =
    \frac 1n \sum_{x=0}^n (\nabla^*\varphi_x) \lang q'_x  r'_x \rang_t 
    + \frac 1n \sum_{x=0}^n \varphi_x \big(\lang p'_xp'_0 \rang_t {-2\ga \lang q_x'p_0' \rang_t}\big)\\
    + \frac 1{n^3t} \sum_{x=0}^n\varphi_x\bE_{\mu_n} \big[\big(p'_xq'_x+\ga (q'_x)^2\big)(n^2t)
    -\big(p'_xq'_x+\ga (q'_x)^2\big)(0)\big].
  \end{split}
\end{equation*}
Thanks to  Corollary \ref{cor012504-24}, proved below, which estimates the contribution from $\lang p_x p_0\rang_t$,
{together with a Cauchy-Schwarz inequality similar to \eqref{eq:cscn}
  for the contribution coming from the $\lang \bar p_x \bar p_0\rang_t$}, we can prove directly 
\begin{equation*}
  \lim_{n\to\infty} \frac 1n \sum_{x=0}^n \varphi_x \lang p'_xp'_0 \rang_t = 0.
\end{equation*}
By Proposition \ref{prop012204-24}, 
\[
\lim_{n\to+\infty}\frac 1{n^3} \sum_{x=0}^n \mathbb E_{\mu_n}\left[ (q'_x)^2(n^2
  t)\right]=0\quad\mbox{and}\quad \lim_{n\to+\infty}\frac {1}n
      \sum_{x=0}^n \varphi_x \lang q_x'p_0' \rang_t=0.
\]
This also gives the bound
\begin{equation*}
  \left|\frac 1{n^3} \sum_{x=0}^n\varphi_x \mathbb E_{\mu_n}\left[ (p'_x q'_x)(n^2 t)\right]\right| \le
  \frac {\|\varphi\|_\infty }n \left( \frac 1{n} \sum_{x=0}^n \mathbb E_{\mu_n}\big[ (p'_x)^2(n^2 t)\big]\right)^{\frac12}
  \left( \frac 1{n^3} \sum_{x=0}^n \mathbb E_{\mu_n}\big[ (q'_x)^2(n^2 t)\big]\right)^{\frac12},
\end{equation*}
which therefore vanishes as $n\to+\infty$ (recall the energy bound \eqref{eq:44a} and also Proposition \ref{prop:main} in order to bound the first term on the right hand side). 
Similarly,
\begin{equation}
  \label{022604-24}
  \left|\frac 1n \sum_{x=0}^n (\nabla^*\varphi_x) \lang q'_x  r'_x \rang_t\right|
  \le  \|\varphi'\|_\infty   \left(\frac 1n \sum_{x=0}^n \lang (r'_x)^2\rang_t\right)^{\frac12} \left(\frac 1{n^3}\sum_{x=0}^n \lang (q'_x) ^2\rang_t \right)^{\frac12}
,
\end{equation} which, from the same reasons, vanishes as $n\to+\infty$.
This concludes the proof of equipartition as stated in Proposition \ref{cor010910-23}. \qed

     \subsection{Proof of Theorem \ref{thm-loc-equi}}

\label{sec10.3.3}

We start with the case $h=0$, which  is the easiest one.
  We have
  $
    r_x p_x = - j_{x-1,x} + \frac 12 \mathcal G_t r_x^2
    $ and similarly $
    \bar r_x \bar p_x = -  \bar r_x \bar p_{x-1}  + \frac 12
    \bar{\mathcal G}_t \bar r_x^2
    $  for $x=1,\ldots,n$. In consequence
    \begin{multline}
      \label{042604-24}
      \frac 1{n} \sum_{x=0}^{n} \varphi_x \lang r'_x p'_x \rang_t
      =\frac 1{n} \sum_{x=0}^{n} \varphi_x  \lang \bar r_x \bar p_{x-1}\rang_t \\
      -\frac 1{n} \sum_{x=0}^{n} \varphi_x  \lang  j_{x-1,x} \rang_t
     + \frac 1{2n^3t} \sum_{x=0}^{n} \varphi_x \mathbb E \big[(r_x')^2(n^2 t)- (r_x')^2(0)\big].
  \end{multline}
  The second and third terms on the right vanish, as $n\to+\infty$,
  by virtue of the current estimate \eqref{eq:cur} and the energy bound
  \eqref{eq:44a}. Using the Cauchy-Schwarz inequality together with estimates 
  \eqref{012409-23main} and \eqref{012409-23main1} we conclude that
  also  the first term vanishes.

 Let $h \in \{-1,1,2\}$, and by convention in the following, the boundary quantities out of $\{1,\dots,n\}$ vanish. One can check directly that:
      \begin{equation*}
        r_x p_{x+h} =-  q_{x+h} \nabla^\star p_x + r_x p_0 +  \mathcal G_t\left(r_x q_{x+h}\right),
    \end{equation*}
    and similarly for the average quantities.
      Hence,
      \begin{equation}
        \label{052604-04}
        \begin{split}
    \frac 1{n} \sum_{x=0}^{n} \varphi_x \lang  r_x p_{x+h} \rang_t
    =& -\frac 1{n} \sum_{x=0}^{n} \varphi_x \lang  \bar r_x \bar
    p_{x+h}\rang_t  
      -\frac 1{n} \sum_{x=0}^{n} \varphi_x  \lang
      q_{x+h}' \nabla^\star p_x' \rang_t \\
      &
      +
      \frac 1{n} \sum_{x=0}^{n} \varphi_x  \lang r_x' p_0'\rang_t
       + \frac 1{2n^3t} \sum_{x=0}^{n} \varphi_x \mathbb E \big[(r_x' q'_{x+h})(n^2 t)- (r_x' q'_{x+h})(0)\big].
    \end{split}
  \end{equation}
  The first  term on the right hand side
    vanishes with $n\to+\infty$ as in \eqref{042604-24}.
  
      Estimate \eqref{eq:49a} combined   with estimates 
  \eqref{012409-23main} and \eqref{012409-23main1} imply that
      \[
      \lim_{n\to+\infty}\frac1n\sum_{x=0}^{n} \varphi_x \lang r_x' p_{0}' \rang_t=0.
\]
     Concerning the second term we have
\begin{align}
\frac 1{n} \sum_{x=0}^{n} &\varphi_x  \lang
    q_{x+h}' \nabla^\star p_x' \rang_t =\frac 1{n} \sum_{x=0}^n \sum_{j,j'=1}^n\phi_j(x+h)\nabla^\star\psi_{j'}(x)\varphi_x  \lang
                \tilde q_{j} ' \tilde p_{j'} ' \rang_t \notag\\
  &
    =\frac 1{n}
    \sum_{j,j'=1}^n\Big(\frac{\la_{j'}}{\la_j}\Big)^{1/2}\tilde S_{j,j'}^{(r,p)}(\la_j,\la_{j'})\Big\{\Big[\widehat{\varphi^{(e)}}(j-j')-
    \widehat{\varphi^{(e)}}(j+j')\Big] \cos\Big(\frac{\pi
    jh}{n\!+\!1}\Big)\\
  & \quad
    +\Big[\widehat{\varphi^{(o)}}(j-j')+
    \widehat{\varphi^{(o)}}(j+j')\Big] \sin\Big(\frac{\pi
    jh}{n\!+\!1}\Big)\Big\}\notag.
\end{align}
The fact that it vanishes, as $n\to+\infty$, can be argued as in 
\eqref{042204-24}.

Concerning \eqref{eq:116y}, note that
	\begin{equation*}
		\mathcal G_t (p_x p_{x+1}) = - 4\gamma p_x p_{x+1} + (r_{x+1} - r_x) p_{x+1} + (r_{x+2} - r_{x+1}) p_x,
	\end{equation*}
	and, taking into account  \eqref{eq:116u}, \eqref{eq:116y} follows;

	Concerning \eqref{eq:116x}, it follows from the relation
	\begin{equation*}
		\mathcal G_t (p_x r_x) = p_x^2 - r_x^2 - p_x p_{x-1} + r_{x+1} r_x - 2\gamma p_x r_x,
	\end{equation*}
        recalling the previous results \eqref{eq:116u} and \eqref{eq:116y},
        and the equipartition result \eqref{eq:60}.

For the boundary identities,  the first one in \eqref{eq:119} is a consequence of the decomposition
\[ \lang p_1^2 \rang_t - \lang \bar p_1^2\rang_t = \big( \lang p_1^2\rang_t - \lang p_0^2\rang_t \big) + \lang p_0^2 \rang_t - \lang \bar p_1^2 \rang_t\]
 plus the \emph{kinetic energy bound} obtained in Proposition \ref{prop031012-21a} for the first term; the boundary estimate \eqref{eq:45} for the second term; and the control of the average \eqref{012409-23main} for the last term.
The second limit follows from \eqref{eq:24}, \eqref{eq:49} and a simple integration by parts.

About \eqref{eq:119zz},   note  that
\begin{align*}
	\lang
	S^{(p,r)}_{1,1}\rang_t&=- \lang  j_{0,1}\rang_t -\lang  \bar p_{1}\bar r_1\rang_t+ \lang r_1(p_{1}-p_{0})\rang_t
	= - \lang  j_{0,1}\rang_t -\lang \bar p_{0}\bar r_1\rang_t+ \frac 12 \lang {\cal G}_t(r_1^2)\rang_t \\
	&	= - \lang  j_{0,1}\rang_t -\lang \bar p_{0}\bar r_1\rang_t+
	\frac 1{2n^2}  \mathbb E_{\mu_n}\big[ r_1^2(n^2 t) - r_1^2(0)\big]
\end{align*}
that tend to $0$ from Proposition \ref{prop:main} and Corollary \ref{cur}.

To prove the first of \eqref{eq:119c}, notice that
$
{\cal G}_t(r_1p_0)= r_1^2 + p_1p_{0} - p_0^2 - 2\ga  r_1 p_0.
$
Hence,
\begin{equation}
	\label{010910-23-1}
	\frac{1}{n^2}\mathbb E_{\mu_n}\big[(p_0 r_1) (n^2t)-(p_0 r_1) (0)\big]=\lang r_1^2\rang_t - \lang
	p_0^2\rang_t
	+ \lang p_{0}p_1\rang_t - 2\ga\lang p_0 r_1\rang_t,
      \end{equation}
      and the first of  \eqref{eq:119c} follows from \eqref{eq:119} and \eqref{eq:119zz}.

We finally prove the second of \eqref{eq:119c}. We have
$
{\cal G}_t(r_1p_1)= r_1r_{2} - r_{1}^2 +p_1^2- p_1p_{0} - 2\ga p_1 r_1.
$
Hence,
\begin{equation}
	\label{010910-23}
	\frac{1}{n^2}\mathbb E_{\mu_n}\big[(p_1 r_1) (n^2t)-(p_1 r_1) (0)\big]=\lang r_{2}r_1\rang_t+\lang
	p_1^2\rang_t
	-\lang r_{1}^2\rang_t  - \lang p_{0}p_1\rang_t - 2\ga\lang p_1 r_1\rang_t.
\end{equation}
The second of \eqref{eq:119c} then follows from \eqref{eq:119}, \eqref{eq:119zz} and the first of \eqref{eq:119c}.
\qed

\appendix

\section{Fourier coefficients. Gradient and divergence operators}
\label{sec:fourier}

\subsection{Orthonormal basis}

We introduce two orthonormal basis  in the respective spaces
$\bbR^{n+1}$ and $\bbR^n$, which are used in our assumptions,
and in several proofs.

Let us define
\begin{equation}
  \psi_j(x):=\left(\frac{2-\delta_{0,j}}{n\!+\!1}\right)^{1/2}\cos\left(\frac{\pi j(2x\!+\!1)}{2(n\!+\!1)}\right),
  \qquad {(x,j) \in \{0,\ldots,n\}^2}.
  \label{eq:vecneu}
\end{equation}
and
\begin{equation}
\phi_j(x):=\left(\frac{2}{n\!+\!1}\right)^{1/2}\sin\left(\frac{j\pi x}{n\!+\!1}\right),
\qquad {(x,j) \in \{1,\ldots,n\}^2}.
\label{eq:vecdir}
\end{equation}
It turns out that $\{\psi_j\}_{j=0,\ldots,n}$ and $\{\phi_j\}_{j=1,\ldots,n}$ constitute orthonormal basis of $\R^{n+1}$ and $\R^n$, respectively.
Given two sequences $g=(g_0,\ldots,g_n)\in\R^{n+1}$ and $f=(f_1,\ldots,f_n)\in\R^n$  we define their respective Fourier transform
as
\begin{align}\label{eq:fourier1}
\tilde g_j&:=\sum_{x=0}^n \psi_j(x)g_x,\qquad j=0,\ldots,n \\ 
\tilde f_j&:=\sum_{x=1}^n \phi_j(x)f_x,\qquad j=1,\ldots,n.
\label{eq:fourier2}
\end{align}
Then, the inverse transforms are given by
\begin{equation}
  \label{042410-21}
  \begin{split}
&g_x=\sum_{j=0}^n\tilde g_j\psi_j(x) ,\qquad x=0,\ldots,n\\
&
f_x=\sum_{j=1}^n\tilde f_j\phi_j(x) ,\qquad x=1,\ldots,n.
\end{split}
\end{equation}
In the proofs, we  regularly use the fact that: 
\begin{equation}\label{eq:psibnd}
\psi_j^2(x) \leqslant \frac{C}{n}, \quad j=0,\ldots,n,\qquad \text{with $C>0$ a universal constant.}\end{equation}

\subsection{Gradient and divergence operators}
\label{sec:gradients}

Let $\bbI_n:=\{0,\ldots,n\}$ and take a function $f:
\bbI_n\to\bbR$. Then, it can be represented as a vector in finite
dimensional space, and its \emph{divergence} $\nabla^\star:\bbR^{n+1}\to\bbR^n$ is defined as follows:
\[
f =\begin{pmatrix}f_0\\
    \vdots
    \\
    f_n\end{pmatrix} \qquad \text{and}\qquad \nabla^\star f := \begin{pmatrix}f_1-f_0\\
    \vdots
    \\
    f_n-f_{n-1}\end{pmatrix}
\]
so that 
$\nabla^\star f_x=f_{x}-f_{x-1}$, for any $x=1,\ldots,n$.
 
  The \emph{gradient} operator $\nabla:\bbR^{n}\to\bbR^{n+1}$ assigns to each vector
  \[
g =\begin{pmatrix}g_1\\
    \vdots
    \\
    g_n
    \end{pmatrix} \qquad \text{a vector} \quad  \nabla g =\begin{pmatrix}g_1\\
                  g_2-g_1\\
    \vdots
    \\
                  g_n-g_{n-1}\\
                  -g_n\end{pmatrix}
  \]
 so that $\nabla g_x=g_{x+1}-g_{x}$ for any $x\in\{0,\ldots,n\}$, with the
 convention $g_0=g_{n+1}=0$.

The matrices  corresponding to the divergence and gradient {operators are} given
by  \eqref{eq:nablamatrices}.
Note that $\nabla^{\rm T}=-\nabla^\star $, i.e. for any $f:\{0,\dots,n\} \to \R$ and $g:\{1,\dots,n\}\to \R$,
\begin{equation}
  \label{eq:ipp}
  \sum_{x=0}^n \nabla g_x\;  f_x = -\sum_{x=1}^n g_x \nabla^\star f_x \; .
\end{equation}
Recall  the vectors $\psi_j$, $j=0,\ldots,n$ and $\phi_j$,
$j=1,\ldots,n$ defined in \eqref{eq:vecneu} and \eqref{eq:vecdir}.
We have the following orthogonality relations: for $x,x',j,j'=0,\dots,n$,
\begin{equation} \begin{split}
\sum_{k=0}^n\psi_k(x) \psi_k(x')&=\delta_{x,x'},\qquad
\sum_{y=0}^n\psi_j(y) \psi_{j'}(y)=\delta_{j,j'},\\
\sum_{k=1}^n\phi_k(x) \phi_k(x')&=\delta_{x,x'},\qquad 
\sum_{y=1}^n\phi_j(y) \phi_{j'}(y)=\delta_{j,j'}
\end{split} \label{eq:ortho}\end{equation}
and  the following identities:
\begin{align}
\label{gaps}
(\nabla^\star \psi_j)_{x}&=-\la_j^{1/2} \phi_j(x),\qquad x,j=1,\ldots,n 
 \\ \label{gaph}
(\nabla \phi_j)_x&=\la_j^{1/2} \psi_j(x),\qquad x=0,\ldots,n, \, j=1,\ldots,n,
\end{align}
where
\begin{equation}
  \label{eq:70bis}
  \la_j:=   4\sin^2\left(\frac{j\pi}{2(n\!+\!1)}\right).
\end{equation}
Define also
\begin{align}
  \label{011809-23}
  \Delta_{\cN}f_x &:= \nabla \nabla^\star f_x,\qquad\,x=0,\ldots,n,\\
  \Delta_{\cD}g_x &:= \nabla^\star \nabla g_x,\qquad\,x=1,\ldots,n.\notag
\end{align}
It follows from \eqref{gaps} and \eqref{011809-23} that
  \begin{equation}
    \label{eq:67}
    \begin{split}
      \Delta_{\cD} \phi_j = \lambda_j^{1/2} \nabla^* \psi_j   =
      -\lambda_j \phi_j,\qquad \Delta_{\cN} \psi_j = -\lambda_j^{1/2} \nabla\phi_j   =
      -\lambda_j \psi_j
    \end{split}
  \end{equation}
 \emph{ i.e.}~$\phi_j$, $\lambda_j$, $j=1,\ldots,n$ and $\psi_j$,
  $\lambda_j$, $j=0,\ldots,n$  are the eigenvectors and corresponding eigenvalues of   $\Delta_{\cD}$
  and $\Delta_{\cN}$, respectively.

\section{Resolution of the covariance matrix}
          \label{app:matrix}     

\subsection{Explicit resolution}

Here we prove Lemma \ref{lem:resolution}, which gives the full resolution of $\tilde S_{j,j'}^{(\alpha)}$. 
Equation \eqref{eq:SAdiff} leads to the following equations
on the blocks defined in \eqref{S1ts}:
\begin{align}
  \label{163011-21}
  \Big[\lang S^{(p,r)}\rang\Big]^T =\lang S^{(r,p)}\rang, \qquad
   \lang S^{(r,p)}\rang\nabla
   & =\nabla^\star\lang S^{(p,r)}\rang + R^{(r)},\end{align} and moreover \begin{align*}
  -\nabla\lang S^{(r)}\rang + 2\gamma \lang S^{(p,r)}\rang
    + \lang S^{(p)}\rang\nabla&= R^{(p,r)}, \\
  \lang S^{(r)}\rang \nabla^\star+ 2\gamma \lang S^{(r,p)}\rang
    - \nabla^\star\lang S^{(p)}\rang &= R^{(r,p)} ,\\
  -\nabla\lang S^{(r,p)}\rang +\lang S^{(p,r)}\rang \nabla^\star
   &=4\gamma D_2(\lang {\mathfrak p^2}\rang) - 4\gamma \lang S^{(p)}\rang + R^{(p)}.    
\end{align*}
With the notation \eqref{FT} for the Fourier coefficients, we get for $j,j'=1,\ldots,n$:
\begin{equation}
  \label{eq:78}
  \begin{split}
  \la_{j'}^{1/2}\tilde S^{(r,p)}_{j,j'}
   &=-\la_{j}^{1/2}\tilde S^{(p,r)}_{j,j'} + \tilde R^{(r)}_{j,j'}\\
 -\la_j^{1/2} \tilde S^{(r)}_{j,j'}+ 2\gamma \tilde S^{(p,r)}_{j,j'}
    +\la_{j'}^{1/2}\tilde S^{(p)}_{j,j'}&= \tilde R^{(p,r)}_{j,j'}\\
  -\la_{j'}^{1/2} \tilde S^{(r)}_{j,j'}+ 2\gamma \tilde S^{(r,p)}_{j,j'}
    +\la_j^{1/2} \tilde S^{(p)}_{j,j'}&= \tilde R^{(r,p)}_{j,j'}\\
   -\la_{j}^{1/2}\tilde S^{(r,p)}_{j,j'}-\la_{j'}^{1/2}\tilde S^{(p,r)}_{j,j'}&=4 \ga \lang \tilde F_{j,j'} \rang
    -4\ga\tilde  S^{(p)}_{j,j'} + \tilde R^{(p)}_{j,j'}.
  \end{split}
\end{equation}
The above equalities also hold when $j=0$, or $j'=0$, provided we let
$\tilde r_0(n^2t)\equiv 0$.
This system of equations takes the form
\[
\begin{pmatrix}
-\la_{j'}^{1/2} & \la_j^{1/2} & 0 & 2\ga \vphantom{\Big)} \\
- \la_j^{1/2} & \la_{j'}^{1/2} & 2\ga & 0 \vphantom{\Big)} \\
0 & 4 \ga & -\la_{j'}^{1/2} & -\la_j^{1/2} \vphantom{\Big)} \\
0 & 0 & \la_j^{1/2} & \la_{j'}^{1/2}\vphantom{\Big)} 
\end{pmatrix} \begin{pmatrix}
\tilde S^{(r)}_{j,j'}\vphantom{\Big)}  \\ \tilde S^{(p)}_{j,j'} \vphantom{\Big)}  \\ \tilde S^{(p,r)}_{j,j'} \vphantom{\Big)}  \\ \tilde S^{(r,p)}_{j,j'}  \vphantom{\Big)} 
\end{pmatrix} = \begin{pmatrix}
\tilde R^{(r,p)}_{j,j'}\vphantom{\Big)}  \\ \tilde R^{(p,r)}_{j,j'} \vphantom{\Big)}  \\ 4 \ga \lang \tilde F_{j,j'} \rang +  \tilde R^{(p)}_{j,j'} \vphantom{\Big)}  \\ \tilde R^{(r)}_{j,j'}  \vphantom{\Big)} 
\end{pmatrix}.
\]
A long but straightforward linear algebra resolution leads to \eqref{eq:cov-t3}.

\subsection{Matrix $\mathbf{\Pi}(\Gamma)$ and its properties}

 Consider a non-negative definite  symmetric block matrix
   \begin{equation*}
                \Gamma:=
                  \begin{pmatrix}
                   [ \Gamma^{(r)}_{x,x'}]_{x,x'=0,\ldots,n}&[ \Gamma^{(r,p)}_{x,x'}]_{x,x'=0,\ldots,n}\\
                  [ \Gamma^{(p,r)}_{x,x'}]_{x,x'=0,\ldots,n}&[ \Gamma^{(p)}_{x,x'}]_{x,x'=0,\ldots,n}
                \end{pmatrix}.
              \end{equation*}
Let $\hat\Gamma $
               be its Fourier transform obtained as in \eqref{FT}.
                 Define the   matrix $ \mathbf{\Pi}(\Gamma)$ 
                 \begin{equation*}
                 \mathbf{\Pi}(\Gamma)_{j,j'} : = 
                 \sum_{\iota=p,\,r}
                   \Theta_\iota(\la_j,\la_{j'}) \hat\Gamma^{(\iota)}_{j,j'}-\sum_{\iota=pr,\,rp}
                   \Theta_\iota(\la_j,\la_{j'}) \hat\Gamma^{(\iota)}_{j,j'}, \qquad j,j'= 0, \dots, n
                 \end{equation*}
                   with  $\Theta_\iota(\la_j,\la_{j'}) $
                   given by \eqref{Xip}.
               \begin{lemma}
                 \label{lm012709-23}
                The matrix $\mathbf{\Pi}(\Gamma)$ is non-negative
                definite.  There exists $C>0$ such that
                \begin{equation}
                  \label{022709-23}
                 {\rm Tr}\, [\mathbf{\Pi}(\Gamma)] \le 2  {\rm
                   Tr}\, [ \Gamma].
                 \end{equation}
              \end{lemma}
              \proof
              Consider the stationary solution of the system of equations
              \begin{equation} 
\label{eq:fflip-1xyz}
\begin{aligned}
  \dd { u}_x(t) &= \nabla^\star v_x(t)\dd t +  \dd  w^{(r)}_x(t), \qquad x=0, \dots, n,
  \\
  \dd { v}_x(t) &=  \nabla u_x (t)\dd t-   2\ga  v_x(t) \dd t +
  \dd   w^{(p)}_x(t), \qquad x=0, \dots, n,
  \end{aligned} \end{equation}
where $\mathbf {W}(t)=\big[w_0^{(r)}(t),\ldots,w^{(r)}_n(t),
w_0^{(p)}(t),\ldots,w^{(p)}_n(t)\big]^T$  is a $2n\!+\!2$-dimensional Brownian motion
with the covariance matrix $\Gamma$.

Let $\mathbf u(t)=\big[{ u}_0(t),\ldots, { u}_n(t)\big]^T$ and $\mathbf
v(t)=\big[{ v}_0(t),\ldots, { v}_n(t)\big]^T$. The process  $\mathbf
Y(t) =[\mathbf u(t),\mathbf v(t)]^T$ is Gaussian of zero mean,
given by
\begin{equation}
  \label{Y}
{\bf Y}(t)=\int_{-\infty}^t
e^{-A(t-s)}\dd \mathbf W(s),\quad t\ge0,
\end{equation}
with the matrix $A$ given by \eqref{bA}.
Its  covariance matrix $C$
\begin{equation}
\label{S1tsC}
C
=\begin{bmatrix}
    {C^{(u)}}&C^{(u,v)}\\
   C^{(v,u)}& C^{(v)}
  \end{bmatrix},
\end{equation}
where
\begin{align}
\label{S1ts1uv}
  &C^{(u)}=\Big(\bbE[u_x(t)u_y(t)]\Big)_{x,y=0,\ldots,n},\quad
    C^{(u,v)}=\Big(\bbE[u_x(t)v_y(t)]\Big)_{x=0,\ldots,n,y=0,\ldots,n},\notag\\
  &C^{(v)}=\Big(\bbE[v_x(t)v_y(t)]\Big)_{x,y=0,\ldots,n}\quad
    \mbox{and}\quad C^{(v,u)}=\Big(C^{(u,v)}\Big)^T,
\end{align}
satisfies equation
\begin{equation}\label{eq:SAdiffbis}
  A  C + CA^T= \Gamma,
\end{equation}
Let us define the matrix $\tilde C$ as a block matrix, whose
respective entries are given by
\begin{align*}
  \tilde C_{j,j'}^{(r,p)}&:= \langle \phi_j,   C^{(r,p)} \; \psi_{j'}\rangle, \qquad \;
                           \tilde C_{j',j}^{(p,r)}:= \langle
                           \psi_{j'},   C^{(p,r)} \; \phi_{j}\rangle \\
  \tilde C_{j,j'}^{(r)}&:= \langle \phi_j,   C^{(r)}\; \phi_{j'}\rangle,
   \qquad \quad \tilde C_{\ell,\ell'}^{(p)}:= \langle \psi_\ell,
                         C^{(p)}  \; \psi_{\ell'}\rangle.
\end{align*}
Repeating the argument made between \eqref{163011-21} and
\eqref{eq:78} 
  we have for $j,j'=0,\ldots,n$:
\begin{equation}
  \label{eq:78z}
  \begin{split}
  \la_{j'}^{1/2}\tilde C^{(r,p)}_{j,j'}
   &=-\la_{j}^{1/2}\tilde C^{(p,r)}_{j,j'} + \hat \Gamma^{(r)}_{j,j'},\\
 -\la_j^{1/2} \tilde C^{(r)}_{j,j'}+ 2\gamma \tilde C^{(p,r)}_{j,j'}
    +\la_{j'}^{1/2}\tilde C^{(p)}_{j,j'}&= \hat \Gamma^{(r,p)}_{j,j'}\\
  -\la_{j'}^{1/2} \tilde C^{(r)}_{j,j'}+ 2\gamma \tilde C^{(r,p)}_{j,j'}
    +\la_j^{1/2} \tilde C^{(p)}_{j,j'}&= \hat \Gamma^{(p,r)}_{j,j'},\\
   -\la_{j}^{1/2}\tilde C^{(r,p)}_{j,j'}-\la_{j'}^{1/2}\tilde
   C^{(p,r)}_{j,j'}&= 
    -4\ga\tilde  C^{(p)}_{j,j'} + \hat\Gamma^{(p)}_{j,j'},
  \end{split}
\end{equation}
which leads to 
                  \begin{align}
                    \tilde C_{j,j'}^{(\iota)} =   \sum_{\iota'\in I} \Xi_{\iota'}^{(\iota)} (\la_{j},\la_{j'})   \hat \Gamma^{(\iota')}_{j,j'} 
                                                              \qquad j,j'=0,\dots,n,
     \label{eq:cov-t3z}
                  \end{align}
      where  $\iota\in \{p,r,{(r,p),(p,r)}\}$. Since
      \[
      \Xi_{\iota}^{(p)}=\begin{cases}\frac{1}{4\ga} \Theta_{\iota},& \mbox{for }
                                                          \iota=p,r,\\
        -\frac{1}{4\ga} \Theta_{\iota},& \mbox{for }
        \iota=pr,\,rp
      \end{cases}
    \]   in the particular case
      $\iota=p$ we get
     \begin{align}
                    \tilde C_{j,j'}^{(p)} =  \frac{1}{4\ga}
       \sum_{\iota'=p,\,r} \Theta_{\iota'} (\la_{j},\la_{j'})   \hat
       \Gamma^{(\iota')}_{j,j'} -\frac{1}{4\ga} \sum_{\iota'=pr,\,rp}
       \Theta_{\iota'} (\la_{j},\la_{j'})   \hat
       \Gamma^{(\iota')}_{j,j'}                       \label{eq:cov-t3zz}
     \end{align}
     for  $j=0,\dots,n, \ j'=0,\dots,n.$
Therefore  $[\tilde C_{j,j'}^{(p)} ]= \frac{1}{4\ga} \mathbf{\Pi}(\Gamma)$ is non-negative
definite as a Fourier image of a covariance matrix.
Note that
\begin{align*}
  {\rm Tr}\,\Big(\mathbf{\Pi}(\Gamma)\Big)=\sum_{j=0}^n\sum_{\iota'=p,\,r} \Theta_{\iota'} (\la_{j},\la_{j})   \hat \Gamma^{(\iota')}_{j,j}-\sum_{j=0}^n\sum_{\iota'=pr,\,rp} \Theta_{\iota'} (\la_{j},\la_{j})   \hat \Gamma^{(\iota')}_{j,j} .
  \end{align*}
  We have
  \[
            \Theta_{p} (\la_{j},\la_{j}) =1,\quad 0\le \Theta_{r} (\la_{j},\la_{j})\le 1\quad
 \mbox{and}\quad 
 \Theta_{p,r} (\la_{j},\la_{j}) =\Theta_{r,p} (\la_{j},\la_{j}) =0.
   \]
The estimate \eqref{022709-23} then  follows.
              \qed

\bibliographystyle{amsalpha}

\bibliography{biblio.bib}

\end{document}